\documentclass{article}
\usepackage[utf8]{inputenc}
\usepackage[english]{babel}

\usepackage{amsfonts}
\usepackage{amsmath}
\usepackage{amssymb}
\usepackage{amsthm}
\usepackage{mathtools}
\usepackage{bbm}
\usepackage{geometry}
\usepackage[hyphens]{url}
\usepackage{hyperref}
\usepackage{nomencl}

\usepackage[dvipsnames]{xcolor}
\usepackage{graphicx}
\usepackage{float}
\usepackage{subcaption}
\usepackage{algorithm}
\usepackage{algorithmic}
\usepackage{cleveref}
\usepackage[shortlabels]{enumitem}
\usepackage{pdfpages}
\usepackage{csquotes}

\usepackage{authblk}
\usepackage[T1]{fontenc}

\usepackage[
maxcitenames=8,
backend=bibtex,
style=trad-abbrv,
sorting=nyt,
sortcites
]{biblatex}
\hypersetup{
    colorlinks=true,
    linkcolor=black,
    filecolor=black,
    urlcolor=black,
    citecolor=black
}
\theoremstyle{plain}
\newtheorem{thm}{Theorem}[section]
\newtheorem{proposition}[thm]{Proposition}
\newtheorem{corollary}[thm]{Corollary}
\newtheorem{remark}[thm]{Remark}
\newtheorem{example}[thm]{Example}

\theoremstyle{definition}
\newtheorem{definition}[thm]{Definition}

\newcommand{\ind}{\mathbbm{1}}

\newcommand{\var}{\mathrm{Var}}
\newcommand{\re}{\mathrm{RE}}

\providecommand{\keywords}[1]{\textbf{\textbf{Keywords  }} #1}

\DeclarePairedDelimiter\floor{\lfloor}{\rfloor}

\begin{document}

\title{Improving control based importance sampling strategies \\for metastable diffusions via adapted metadynamics}
\date{September 2023}

\setcounter{Maxaffil}{0}
\renewcommand\Affilfont{\itshape\small}

\author[1,2]{Enric Ribera Borrell}
\author[2]{Jannes Quer}
\author[1,3,4]{Lorenz Richter}
\author[1,2]{Christof Sch\"utte}
\renewcommand\Authands{ and }

\affil[1]{Zuse Institute Berlin, 14195 Berlin, Germany}
\affil[2]{Institute of Mathematics, Freie Universität Berlin, 14195 Berlin, Germany}
\affil[3]{Institute of Mathematics, BTU Cottbus-Senftenberg, 03046 Cottbus, Germany}
\affil[4]{dida Datenschmiede GmbH, 10827 Berlin, Germany}

\maketitle

\begin{abstract}
Sampling rare events in metastable dynamical systems is often a computationally expensive task and one needs to resort to enhanced sampling methods such as importance sampling. Since we can formulate the problem of finding optimal importance sampling controls as a stochastic optimization problem, this then brings additional numerical challenges and the convergence of corresponding algorithms might suffer from metastabilty. In this article, we address this issue by combining systematic control approaches with the heuristic adaptive metadynamics method. Crucially, we approximate the importance sampling control by a neural network, which makes the algorithm in principle feasible for high-dimensional applications. We can numerically demonstrate in relevant metastable problems that our algorithm is more effective than previous attempts and that only the combination of the two approaches leads to a satisfying convergence and therefore to an efficient sampling in certain metastable settings.
\end{abstract}


\setlength\parindent{0pt}

\section{Introduction}
\label{sec: introduction}
The accurate computation of rare events is of great importance in multiple applications, relating to fields such as molecular dynamics, epidemiology, engineering, and finance, to name just a few. One is typically interested in events that happen only very rarely, but are still relevant for certain phenomena of interest. Since analytical computations are mostly infeasible in practice, one usually relies on Monte Carlo approximations for the desired quantities. The related sampling problem, however, can be very challenging mainly for two reasons: a potentially high dimension of the problem at hand as well as large statistical errors of corresponding estimators, which are rooted in the characteristic of the events being rare. Loosely speaking, the difficulty of sampling rare events is based on its very definition: it is hard to observe an event (frequently) if it almost never appears (at least in relation to the typical timescales for which a simulation is feasible). In fact, the characteristic exponential divergence of the relative error with the parameter that controls the rarity of the quantity of interest poses great computational challenges. \\

In this article we shall focus on rare events in stochastic processes, where one is interested in sampling regions of the state space which are unlikely to be visited. In particular, we are interested in processes that exhibit some sort of metastability, where particles that follow the dynamics stay in certain regions of the space for a very long time. In fact, the average waiting of switching between metastable events is orders of magnitude longer than the timescale of the process itself. This is, for instance, typical in molecular simulations with particles following the Langevin dynamics in which a potential function governs the evolution of the stochastic process; see, e.g., \cite{Hartmann2014characterization}. Here, metastable regions correspond to local minima of the potential, which are separated by so-called energy barriers, and transitions between those regions are of interest since they correspond to macroscopic properties of corresponding molecules. These are, for instance, reaction rates or conformation changes, such as the folding of a protein or a phase transition. However, those transitions happen only very rarely so that a simulation of transition trajectories can be extremely difficult from a computational point of view. On the one hand, the time to overcome energy barriers might be extremely large (in fact, it scales exponentially with the height of the energy barrier \cite{Berglund2013}); on the other hand, variances of estimators related to those rare transitions might be large\footnote{Note that those two aspects usually interact.}.\\

One idea to overcome those challenges is to apply importance sampling. Abstractly speaking, the basic idea is to sample from another probability distribution and weight the resulting random variables back in order to still get an unbiased estimator for the quantity of interest. Since we are interested in path dependent quantities we consider importance sampling in the space of continuous trajectories. This corresponds to adding a function to the drift of original dynamics. One can think of the additional function as a control function or force that pushes trajectories into desired regions of the state space and thereby allows for overcoming possible energy barriers. Equivalently, one can think of modifying the original physical potential such that it appears less ``rugged'' and particles are no longer trapped in local minima. In principle, it is possible to design modifications of the potential rather freely. However, one has to keep in mind that these modifications influence the quality of the importance sampling estimator significantly; see, e.g., \cite{Hartmann2021nonasymptotic}. A systematic approach for finding good control functions that aim to minimize the variance of the estimator is related to a stochastic optimal control problem \cite{Hartmann2012} (for further variational perspectives we refer to \cite{Nusken2021solving}). This perspective then allows for numerical strategies such as iterative stochastic optimization methods that aim to find efficient controls in practice. At the same time, especially in metastable situations, those approaches hold two additional challenges that might make corresponding algorithms infeasible in applications:
\begin{itemize}
    \item In order to compute a first iteration in the stochastic optimization procedure the rare event of interest must be simulated at least once. If this does not happen, one can usually not proceed.
    \item Even if one manages to simulate rare events with great computational effort, the estimated objectives in the optimization routines as well as their gradients might suffer from high variances, which might make convergence of the method very slow. 
\end{itemize}

In this article we develop an algorithm that addresses these two aspects and improves importance sampling based estimation in metastable scenarios. In particular, we will combine systematic control based approaches with heuristic adaptive methods that are related to the so-called metadynamics algorithm.

\subsection{Previous work}
We have mentioned before that rare event sampling occurs in multiple different fields of application, where each field adds a different perspective. In what follows, let us review some of those perspectives and relate them to works that are relevant for our endeavor.

\paragraph{Adaptive biasing techniques}
Methods that aim to modify the potential on the fly in order to remove metastable features of the dynamics depending on the particles in the simulation are often subsumed under the term \textit{adaptive biasing techniques}. A well-known method is called metadynamics \cite{Laio2002}, which was developed in order to improve the sampling related to stationary distributions of complex molecular systems. Many extensions and applications have been published throughout the last few years. For a good review on recent developments we refer to \cite{Chipot2015,Valsson2016} and the references therein. Convergence results and high-dimensional adaptations can be found in \cite{Galvelis2017, Jourdain2021}. An extension to importance sampling for path dependent properties of interest has been proposed in \cite{Quer2018} and similar ideas based on the adaptive biasing force technique have been suggested by \cite{Wang2001}. This method has been used in many applications and different extensions have, for instance, been proposed in \cite{Bussi2006, Henin2004}. To our knowledge, for the adaptive biasing force methods no extension for path dependent quantities has been considered yet. Let us also note that related nonequilibrium methods have been addressed (see, e.g., \cite{Kumar1992} or \cite{Zwanzig1954}), but as before, an extension to path dependent problems is usually not covered.

\paragraph{Rare event sampling in an asymptotic regime}
Many methods for rare event estimation have been developed in an asymptotic regime, relying usually on large deviation arguments. Those strategies are often connected to the associated Hamilton--Jacobi--Bellman equations and can, for instance, be found in \cite{Dupuis2004, Dupuis2007, Dupuis2012}. For variance reduction strategies in a zero noise limit we refer to \cite{Weare2012}, which relies on optimal control strategies of the corresponding deterministic problem. The special situation of attractors with resting points has been addressed in \cite{Dupuis2015}, and in \cite{Spiliopoulos15} the variance of importance sampling based on asymptotic arguments applied in a nonasymptotic regime has been analyzed. Even though all of these methods can be applied to path dependent quantities, an application to high-dimensional applications is usually not addressed.

\paragraph{Nonasymptotic importance sampling}
Importance sampling in a nonasymptotic regime, which targets sampling path dependent quantities, corresponds to a controlled stochastic process; see, e.g., \cite{Milstein1995}. A strategy that aims to identify optimal importance sampling controls has been suggested in \cite{Hartmann2012}. Numerically, the approach rests on the approximation of the control by a linear combination of ansatz functions. In \cite{Hartmann2017} the corresponding method is analyzed from the perspective of path space measures and variational formulations of the problem are considered. Further variational perspectives have been suggested in \cite{Nusken2021solving}, putting additional emphasis on certain numerical robustness properties and allowing for high-dimensional applications by modeling the control with neural networks. For strategies that are based on backward stochastic differential equations we refer, for instance, to \cite{Hartmann2019variational}. The optimal control attempt has also been combined with model reduction techniques in \cite{Hartmann2018, Hartmann2016} and \cite{Zhang2016}, noting that one of the main drawbacks of this approach is the placing of ansatz functions over the domain of interest. For a statistical analysis of importance sampling in path space that highlights its nonrobustness in particular in high dimensions we refer to \cite{Hartmann2021nonasymptotic}. We also refer to \cite{Richter2021phd} for a comprehensive introduction to nonasymptotic importance sampling for path functionals.

\paragraph{Optimal control problems}
Due to the connections of (optimal) importance sampling and optimal control theory we refer to numerical strategies that allow us to solve (high-dimensional) control problems. One strategy is to solve the related Hamilton--Jacobi--Bellman equation, which can, for instance, be tackled with deep learning based strategies in high dimensions; see, e.g., \cite{Han2017deep, Nusken2021solving, Weinan2021algorithms, Zhou2021actor}. Let us in particular refer to \cite{Nusken2021interpolating}, where elliptic partial differential differential equations (PDEs) are considered, which are relevant for the problems we focus on in this article. Let us further highlight \cite{Nusken2021solving}, where robustness properties of loss functions have been analyzed, leading in particular to the novel log-variance divergence, which exhibits favorable numerical properties. For approximating control functions with tensor trains we refer, for instance, to \cite{Fackeldey2022approximative}.

\subsection{Outline of the article}
The article is structured as follows. In \Cref{sec: sampling metastable dynamics} we state the rare event problem and discuss issues appearing in naive Monte Carlo estimations. In \Cref{sec: importance sampling} we introduce importance sampling as a strategy to overcome those issues and in \Cref{sec: bvp} we subsequently show how one can aim for optimal importance sampling strategies by deriving an equivalent optimal control problem via PDE arguments. In \Cref{sec: numerical strategies} we then address computational aspects of solving this control problem via an optimization approach. In particular, \Cref{sec: gradient computations} is devoted to the computation of gradients that are needed in iterative optimization methods, \Cref{sec: metadynamics} introduces the metadynamics based initialization method and \Cref{sec: control function approximation} discusses the approximation of the control functions via neural networks, which then allows us to pose our final algorithm. In \Cref{sec: numerical examples} we subsequently demonstrate in different numerical examples that the algorithm can significantly improve sampling performance in high-dimensional metastable scenarios. Finally, \Cref{sec: summary and outlook} provides a conclusion and an outlook for further research questions. For the proofs and additional statements we refer to \Cref{app: proofs,app: alternative gradient computations,app: girsanov}.

\section{Sampling metastable dynamics}
\label{sec: sampling metastable dynamics}
In this article we focus on stochastic dynamical systems which exhibit metastable features. To be precise, we consider the overdamped Langevin equation
\begin{equation}
\label{eq: langevin sde}
\mathrm d X_s = -\nabla V(X_s) \mathrm ds + \sigma(X_s)\mathrm dW_s, \qquad X_0 = x \in \mathbb{R}^d, 
\end{equation}
on a bounded domain $\mathcal{D} \subset \mathbb{R}^d$, where $(W_s)_{s \geq 0}$ is a $d$-dimensional Brownian motion. The function $V:\mathbb{R}^d \to \mathbb{R}$ shall be understood as a potential that, for instance, governs the dynamics of multiple atoms in a physical system, and for the diffusion coefficient we usually choose $\sigma(x) = \sqrt{2 \beta^{-1}} \operatorname{Id}$, where $\beta > 0$ denotes the inverse temperature\footnote{Note that in principle both the potential and the diffusion coefficient could be made time-dependent as well.}. We assume that there exists a unique strong solution to SDE \eqref{eq: langevin sde} and that the resulting process $X$ is ergodic such that we can guarantee convergence to a unique equilibrium distribution; see, e.g.,  \cite{Lelievre2016} or \cite{Pavliotis2014} for details on these assumptions.

Given a target set $\mathcal{T} \subset \mathcal{D}$, let us define the first hitting time 
of the process $X$ as
\[
\tau \coloneqq \inf \{s > 0 \mid X_s \in \mathcal{T}\},
\]
and note that it is a.s. finite\footnote{Let $\widetilde{\rho}(y, t | x)$ be the probability distribution of the particles which have not arrived in $\mathcal{T}$ yet, given the position $x$ and assuming the time evolution $t$. Then, the homogeneous (absorbing) boundary condition for the mean first hitting time implies that $\lim\limits_{t \rightarrow \infty} \widetilde{\rho}(y, t | x) = 0$ for all $y \in \mathbb{R}^d$, which means that the particles will eventually leave the domain and therefore $\tau < \infty$ holds almost surely \cite{Pavliotis2014}.}. We can now define our quantity of interest $I: C([0, \infty), \mathbb{R}^d) \rightarrow \mathbb{R}$ as
\begin{equation}
\label{eq: quantity of interest}
    I(X) \coloneqq \exp{(- \mathcal{W}(X))}, 
\end{equation}
with the path functional $\mathcal{W}: C([0, \infty), \mathbb{R}^d) \rightarrow \mathbb{R}$ given as
\begin{equation}
 \label{eq: work functional}
\mathcal{W}(X) \coloneqq \int_0^\tau f(X_s) \mathrm ds + g(X_\tau),
\end{equation}
where $f: \mathbb{R}^d  \to \mathbb{R}$ and $g: \mathbb{R}^d \to \mathbb{R}$ are such that $\mathcal{W}$ is integrable. Our goal is to compute the expectation value of the quantity of interest $I$,
\begin{equation}
\label{eq: psi}
\Psi(x) \coloneqq \mathbb{E}^{x}[I(X)],
\end{equation}
which we can view as a function of the initial value $x$, where we introduce the shorthand notation \sloppy $\mathbb{E}^{x}[I(X)] \coloneqq \mathbb{E}[I(X) \mid X_0=x]$. Let us recall that a stochastic process, such as the one defined in \eqref{eq: langevin sde}, is metastable if its dynamic behavior is characterized by unlikely transition events between the so-called metastable regions. In the particular case of an overdamped Langevin process, one can distinguish between two types of metastability, coming from either energetic or entropic barriers \cite{Lelievre2016}. In this article we focus on the former. Let us recall that in this case both the temperature $\beta^{-1}$ and the height of the energetic barriers determine the strength of the metastability \cite{Berglund2013}. In particular, by Kramer's law the mean hitting time satisfies the large deviations asymptotics 
\begin{equation}
\label{eq: exponential dependence hitting time}
    \mathbb{E}[\tau]\asymp\exp\left( \frac{2\Delta V}{\beta}\right) \qquad \text{as} \qquad \beta \to 0,
\end{equation}
where $\Delta V$ is the energy barrier that the dynamics has to overcome in order to reach the target set $\mathcal{T}$. An illustration of this exponential dependency is provided in \Cref{ex: double well}.

\begin{example}[Double well potential]
\label{ex: double well}
For an illustration, let us consider the one-dimensional double well potential
\begin{equation}
    V_\alpha(x) = \alpha (x^2-1)^2,
\end{equation}
where $\alpha > 0$ modulates the height of the energetic barrier and thereby influences the strength of the metastability; see \Cref{fig: 1a}. Let us consider the initial value $x=-1$ in the left well of the potential. We choose the target set $\mathcal{T} = [1, 3]$ to be supported in the right well so that the particles need to cross the potential barrier. In \Cref{fig: 1b} we plot the expected hitting time of reaching $\mathcal{T}$ for different values of $\alpha$ and $\beta$ when using naive Monte Carlo estimations. Indeed we observe the exponential dependence as indicated by \eqref{eq: exponential dependence hitting time}. As mentioned above, one can aim to speed up sampling by reducing the trajectory lengths when applying an importance sampling based sampling scheme. It turns out that there is an optimal way to design such a scheme, leading to substantially reduced mean hitting times which do not scale exponentially with the energy barrier anymore. We will show how to design this optimal scheme in the upcoming sections.

\begin{figure}[tbhp]
\centering
\subfloat[]{\label{fig: 1a}\includegraphics[width=60mm]{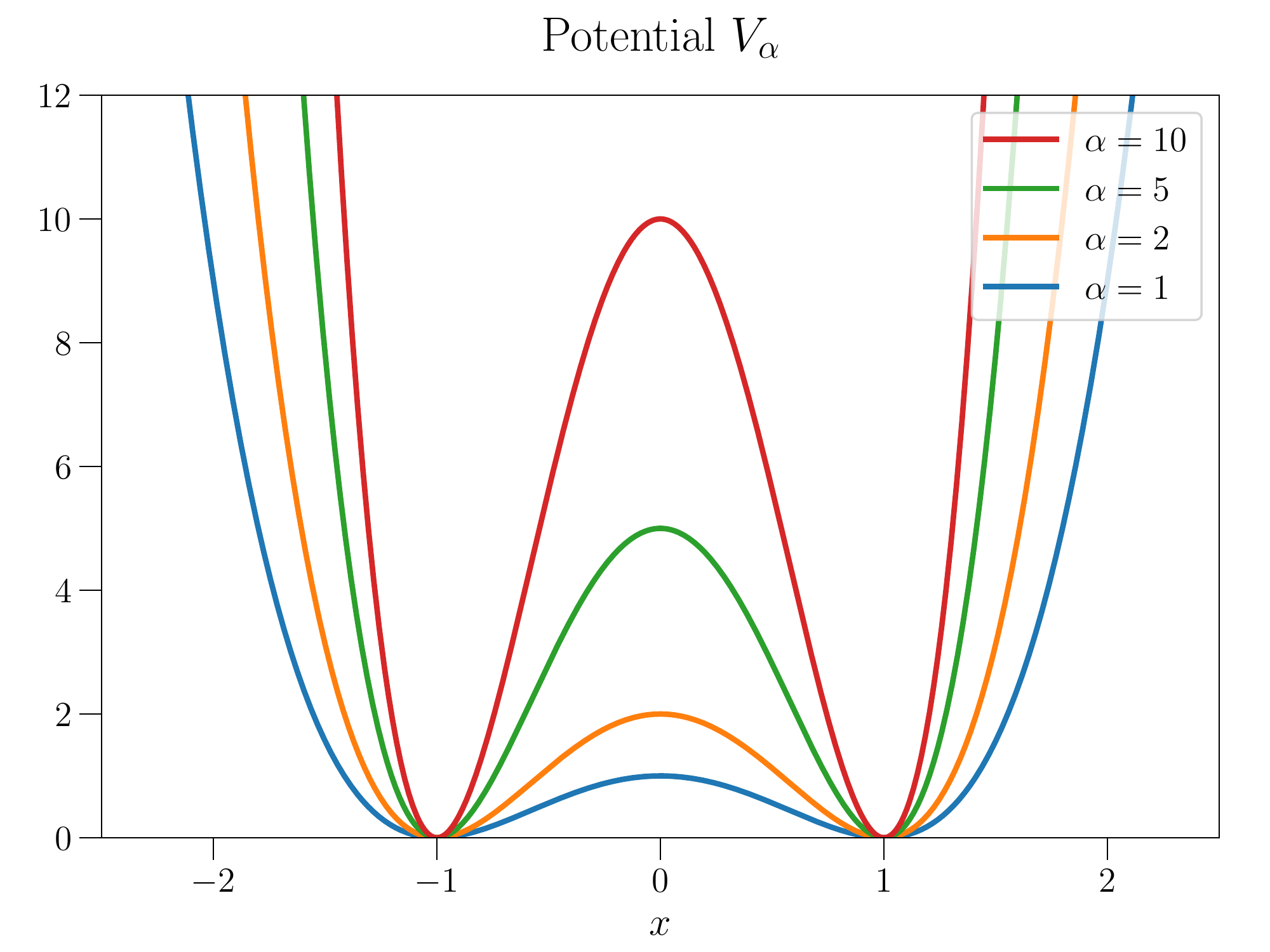}}
\subfloat[]{\label{fig: 1b}\includegraphics[width=60mm]{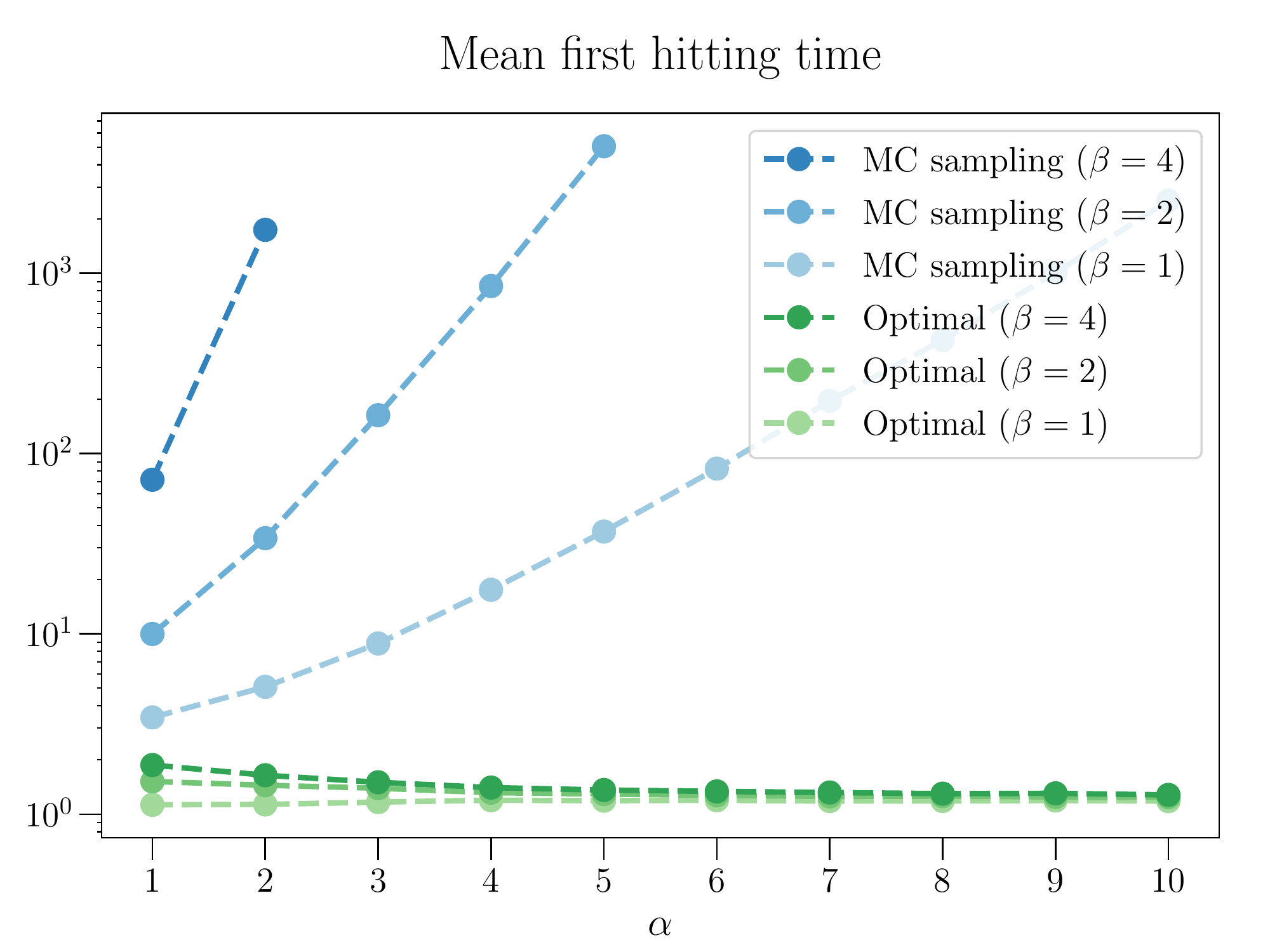}}
\caption{\textbf{(a)} The double well potential for different values of $\alpha$. \textbf{(b)} Mean first hitting times for different values of $\alpha$ and $\beta$.}
\label{fig: double well example hitting times}
\end{figure}
\end{example}

\subsection{Monte Carlo approximations and importance sampling}
\label{sec: importance sampling}
Since there is no closed-form formula available for the computation of the expectation value \eqref{eq: psi}, we must rely on its Monte Carlo estimator
\begin{equation}
\label{eq: mc estimator}
\widehat{\Psi}(x) \coloneqq \frac{1}{K} \sum_{k=1}^{K} I(X^{(k)}) , 
\end{equation}
where $\{ X^{(k)} \}_{k=1}^K$ are independent realizations of the process $X$, all starting at $x$. For a finite sample size $K$ the estimator is unbiased and the usual behaviors of the variance and the relative error hold, i.e.,
\begin{equation}
\var\left(\widehat{\Psi}(x)\right) = \frac{\var(I(X))}{K}, \qquad 
\re(\widehat{\Psi}) \coloneqq \frac{\sqrt{\var\left(\widehat{\Psi}(x)\right)}}{\mathbb{E}^x\left[I(X)\right]}
= \frac{\re(I(X))}{\sqrt{K}}
\end{equation}
for any $x \in \mathbb{R}^d$. In a metastable system the intrinsic relative error of the quantity of interest, $\re(I(X))$, can be very large. As a consequence, reducing the relative error of the Monte Carlo estimator below a prescribed positive value $\varepsilon > 0$, i.e., $\re(\widehat{\Psi}) \leq \varepsilon$, might imply that one needs a very large number of trajectories, namely $K \geq (\re(I(X)) / \varepsilon)^2$. Thus, in order to make numerical estimations feasible, one often needs to rely on methods that reduce the inherent variance of the corresponding stochastic quantities. One such method is importance sampling, on which we shall focus in what follows.

The general idea of importance sampling is to draw random variables from another probability measure and subsequently weight them back in order to still have an unbiased estimator of the desired quantity of interest \cite{Owen2013}. In the case of stochastic processes this change of measure corresponds to adding a control to the original process \eqref{eq: langevin sde}, yielding the controlled dynamics 
\begin{equation}
\label{eq: controlled langevin sde}
\mathrm dX_s^u = (-\nabla V(X_s^u) + \sigma(X_s^u)  \,  u(X_s^u))\mathrm ds + \sigma(X_s^u) \mathrm dW_s, \qquad X_0^u = x, 
\end{equation}
where the control $u$ is an It\^o integrable function that satisfies a linear growth condition, i.e., $u \in \mathcal{U}$ with
\[
\mathcal{U} = \{ u \in C^1(\mathbb{R}^d, \mathbb{R}^d) : u \text{ grows at least linearly in $x$} \}. 
\]
Further details can be found in, e.g., \cite{Hartmann2021nonasymptotic, Hartmann2017, Hartmann2012, Nusken2021solving}. The controlled dynamics \eqref{eq: controlled langevin sde} can now be related to the original one \eqref{eq: langevin sde} via a change of measure in path space, which can be made explicit via Girsanov's formula (see \Cref{app: girsanov} for details). To be precise, it holds that
\begin{equation}
\label{eq: expectation IS}
    \mathbb{E}^x\left[I(X)\right] = \mathbb{E}^x\left[I(X^u) M^u\right],
\end{equation}
where the exponential martingale
\begin{equation}
\label{eq: girsanov martingale}
M^u \coloneqq \exp{\left(-  \int_0^{\tau^u} u(X_s^u) \cdot \mathrm dW_s - \frac{1}{2} \int_0^{\tau^u} |u(X_s^u)|^2 \mathrm ds \right)}
\end{equation}
corrects for the induced bias.

Relating to \Cref{ex: double well}, the control $u \in \mathcal{U}$ can intuitively be understood as an external force aiming to push particles over the energy barrier such that they can escape from metastable regions and reach desired target sets. In principle, the importance sampling relation \eqref{eq: expectation IS} stays intact for any $u\in \mathcal{U}$; however, it turns out that the variance of corresponding estimators significantly depends on an appropriate choice of $u$; see \cite{Hartmann2021nonasymptotic}. In particular, it does not suffice to somehow push particles over existing barriers -- instead, the specific control protocol needs to be chosen very carefully. Clearly, a natural goal for designing an optimal control $u^* \in \mathcal{U}$ is to aim for minimizing the variance of the importance sampling estimator, i.e.,
\begin{equation}
\label{eq: variance minimization}
\var \left( I(X^{u^*}) M^{u^*} \right) = \inf_{u \in \mathcal{U}} \left\{ \var(I(X^u) M^u) \right\}.
\end{equation}
In the next section we shall discuss how this objective can in fact be linked to a classical optimal control problem, which will subsequently lead to feasible numerical strategies.
\pagebreak

\subsection{Optimal control characterizations and associated boundary value problems}
\label{sec: bvp}

In order to derive the connection between variance minimization as stated in \eqref{eq: variance minimization} and a classical optimal control problem, we will essentially argue via PDEs that are associated to our estimation problem\footnote{Note that an alternative derivation can be achieved via certain divergences between path space measures; see \cite{Nusken2021solving}.}. Let us first recall via the Feynman--Kac theorem \cite[Proposition 6.1]{Lelievre2016} that the expectation $\Psi$ (considered as a function of the initial value), as defined in \eqref{eq: psi}, fulfills the elliptic boundary value problem
\begin{subequations}
\label{eq: bvp psi}
\begin{align}
 (L  - f(x)) \Psi(x) &= 0, & x \in \mathcal{S}, \\ 
\Psi(x) &= \exp{(- g(x))}, & x \in \partial \mathcal{S},
\end{align}
\end{subequations}
where $L$ is the infinitesimal generator of the process $X$, defined as
\begin{equation*}
L = \frac{1}{2} \sum_{i,j=1}^d (\sigma \sigma^\top)_{ij}(x) \frac{\partial^2}{\partial x_i \partial x_j} - \sum_{i=1}^d \frac{\partial}{\partial x_i}V(x) \frac{\partial}{\partial x_i}.
\end{equation*}
The domain $\mathcal{S} \coloneqq \mathcal{D} \cap \mathcal{T}^c$ is assumed to be bounded and the functions $f \in C(\mathbb{R}^d, \mathbb{R})$, $g \in C^2(\mathbb{R}^d, \mathbb{R})$ are the same as in \eqref{eq: work functional}.

The connection of our estimation problem to an optimal control problem can be revealed when applying the Hopf--Cole transformation (see, e.g., section 4.4.1 in \cite{Evans2010}, and cf. \cite{Hartmann2012}) to the solution of the PDE \eqref{eq: bvp psi}, namely
\begin{equation}
\label{eq:cole:hopf}
\Phi(x) = - \log{ \Psi(x)}.
\end{equation}
One can readily show that $\Phi$ now fulfills the nonlinear boundary value problem
\begin{subequations}
\label{eq: bvp value function}
\begin{align}
L \Phi(x) - \frac{1}{2} |\sigma^\top \nabla \Phi(x)|^2 + f(x) &= 0,  & x \in \mathcal{S}, \\ 
\Phi(x) &= g(x), & x \in \partial \mathcal{S}.
\end{align}
\end{subequations}

The PDE \eqref{eq: bvp value function} is known as the Hamilton--Jacobi--Bellman (HJB) equation, which is a key equation in optimal control theory allowing for a characterization of optimal control strategies. In fact, we can now identify the control problem that corresponds to the above PDE and thus to our estimation problem by stating the cost functional
\begin{equation}
\label{eq: cost functional}
J(u; x) \coloneqq \mathbb{E}^x\left[\mathcal{W}(X^u) + \frac{1}{2} \int_0^{\tau^u} |u(X_s^u)|^2 \mathrm ds \right],
\end{equation}
where $X^u$ follows the controlled dynamics as defined in \eqref{eq: controlled langevin sde} and $f$ and $g$ can be interpreted as running and terminal costs, respectively. The solution to PDE \eqref{eq: bvp value function} is sometimes called the value function in the sense that it offers the optimal \textit{costs-to-go}, depending on the initial value $x$, i.e.,
\begin{equation}
    \Phi(x) = \inf_{u \in \mathcal{U}} J(u; x).
\end{equation}
Let us make the above observations precise.

\begin{proposition}[Variance minimization as control problem]
\label{prop: variance control correspondence}
Let us assume there exist solutions \sloppy ${\Psi \in C^2_b(\mathbb{R}^d, \mathbb{R})}$ and $\Phi \in C^2_b(\mathbb{R}^d, \mathbb{R})$ to the elliptic boundary value problems \eqref{eq: bvp psi} and \eqref{eq: bvp value function}, respectively, and set
\begin{equation}
\label{eq: control characterization}
u^* = -\sigma^\top \nabla \Phi  = \sigma^\top \nabla \log \Psi.
\end{equation}
Then the following are equivalent:
\begin{enumerate}[(i)]
    \item $u^* \in \mathcal{U}$ minimizes the control costs as defined in \eqref{eq: cost functional}.
    \item $u^* \in \mathcal{U}$ minimizes the variance of the importance sampling estimator as defined in \eqref{eq: variance minimization}. 
\end{enumerate}
In fact it holds that
\begin{equation}
    \var \left( I(X^{u^*}) M^{u^*} \right) = 0.
\end{equation}
\end{proposition}
\begin{proof}
See, e.g., \cite[Theorem 2]{Hartmann2017} or \cite[Theorem 2.2]{Nusken2021solving}.
\end{proof}

\Cref{prop: variance control correspondence} shows that $u^*$, which minimizes either \eqref{eq: variance minimization} or \eqref{eq: cost functional}, can be recovered from the solution of the HJB equation. This reveals that the optimal control is in fact of gradient form, just as the drift of our original stochastic process \eqref{eq: langevin sde}. We can therefore express the overall drift of the optimally controlled process as
\begin{equation}
\label{eq: optimal potential}
    -\nabla (V + V_\text{bias}^*),
\end{equation}
where $V_\text{bias}^* = \sigma \sigma^\top \Phi$ is sometimes called the optimal bias potential, which can be interpreted as being the optimal correction of the original potential $V$ in terms of variance reduction. As stated in \Cref{prop: variance control correspondence}, one can show that it is optimal in the sense that it drives the variance of the importance sample estimator to zero, thereby yielding a perfect sampling scheme. \\
Let us illustrate this by referring again to \Cref{ex: double well}, where we have considered a stereotypical double well potential with different energetic barriers depending on the parameter $\alpha > 0 $. In \Cref{fig: optimal control and bias potentials} we display the optimal control functions and optimal bias potentials, respectively, that allow the trajectories to cross the barrier -- note that the control is particularly large in regions where the particles get trapped when not applying the control. The optimal solutions are calculated via a finite difference discretization of the corresponding PDE \eqref{eq: bvp psi}. In \Cref{fig: statistics with optimal control} we display the resulting Monte Carlo estimators and corresponding relative errors when using either naive Monte Carlo or the optimal importance sampling estimator. Note that the estimators are indeed much more accurate when relying on the optimal importance sampling control. We do not observe a zero relative error\footnote{The observation that the relative error for the naive Monte Carlo estimator seems to increase with decreasing $\Delta t$ is misleading and seems to be due to the fact that hitting times can be simulated more accurately, which leads to smaller values of $\widehat{\Psi}$ and therefore larger relative errors; see also \Cref{fig: 3a}.} due to the discretization of the process with different step sizes $\Delta t >0$; see also \Cref{sec: numerical examples}.

\begin{figure}[tbhp]
\centering
\subfloat[]{\label{fig: 2a}\includegraphics[width=60mm]{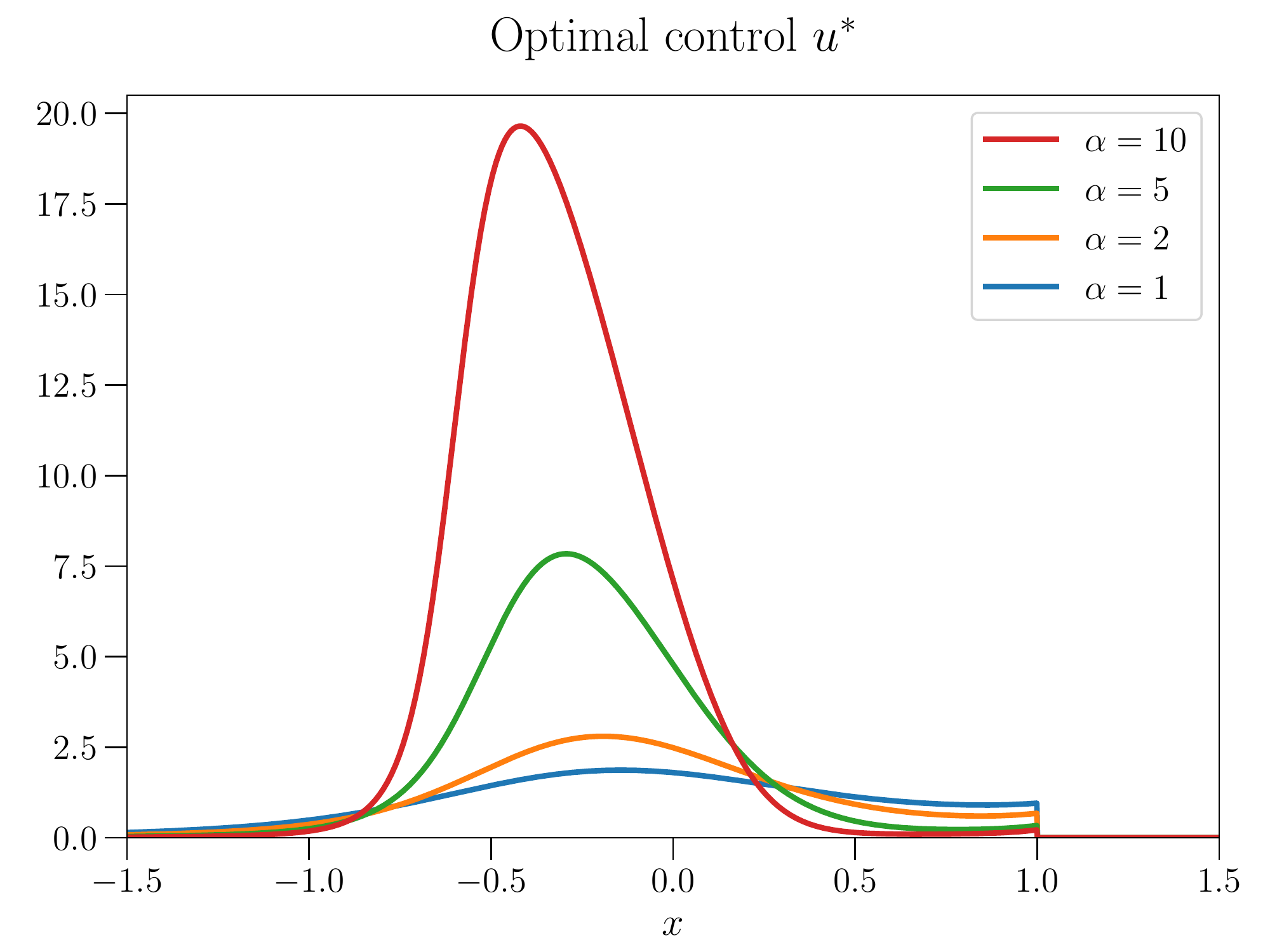}}
\subfloat[]{\label{fig: 2b}\includegraphics[width=60mm]{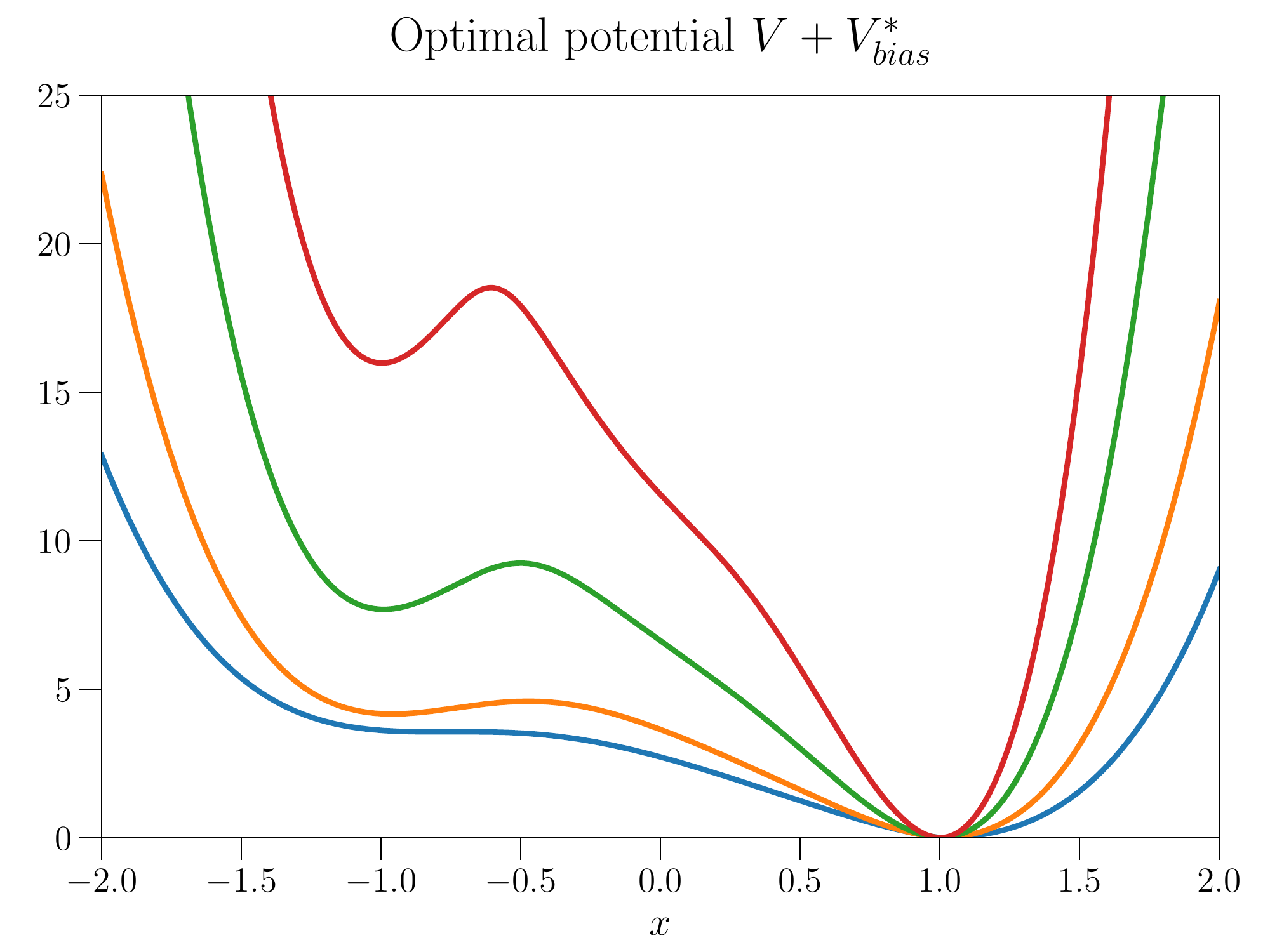}}
\caption{\textbf{(a)} Optimal control functions and \textbf{(b)} optimal potentials for different values of $\alpha$ and inverse temperature $\beta=1$; see also \Cref{ex: double well}.}
\label{fig: optimal control and bias potentials}
\end{figure}

\begin{figure}[tbhp]
\centering
\subfloat[]{\label{fig: 3a}\includegraphics[width=60mm]{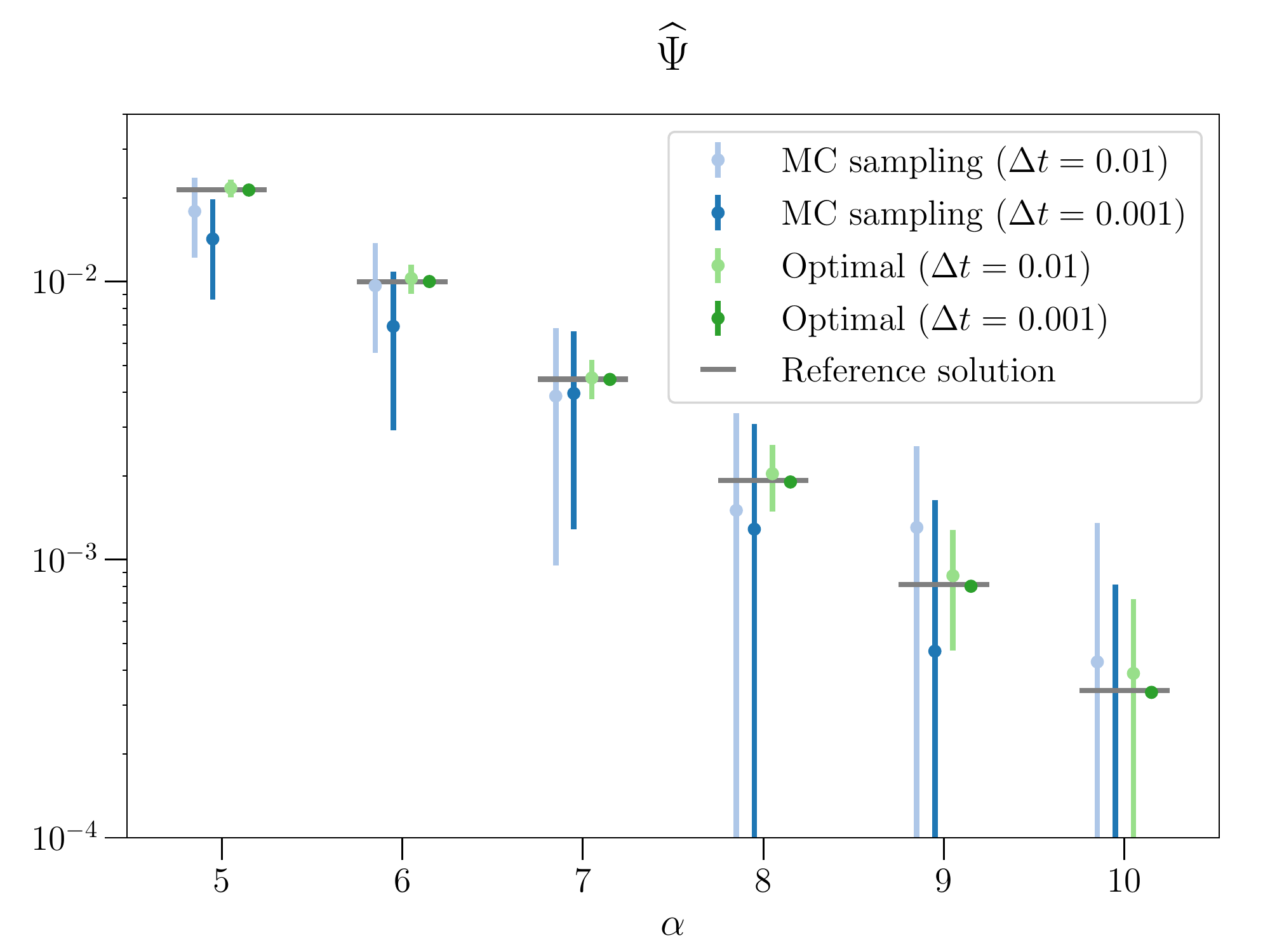}}
\subfloat[]{\label{fig: 3b}\includegraphics[width=60mm]{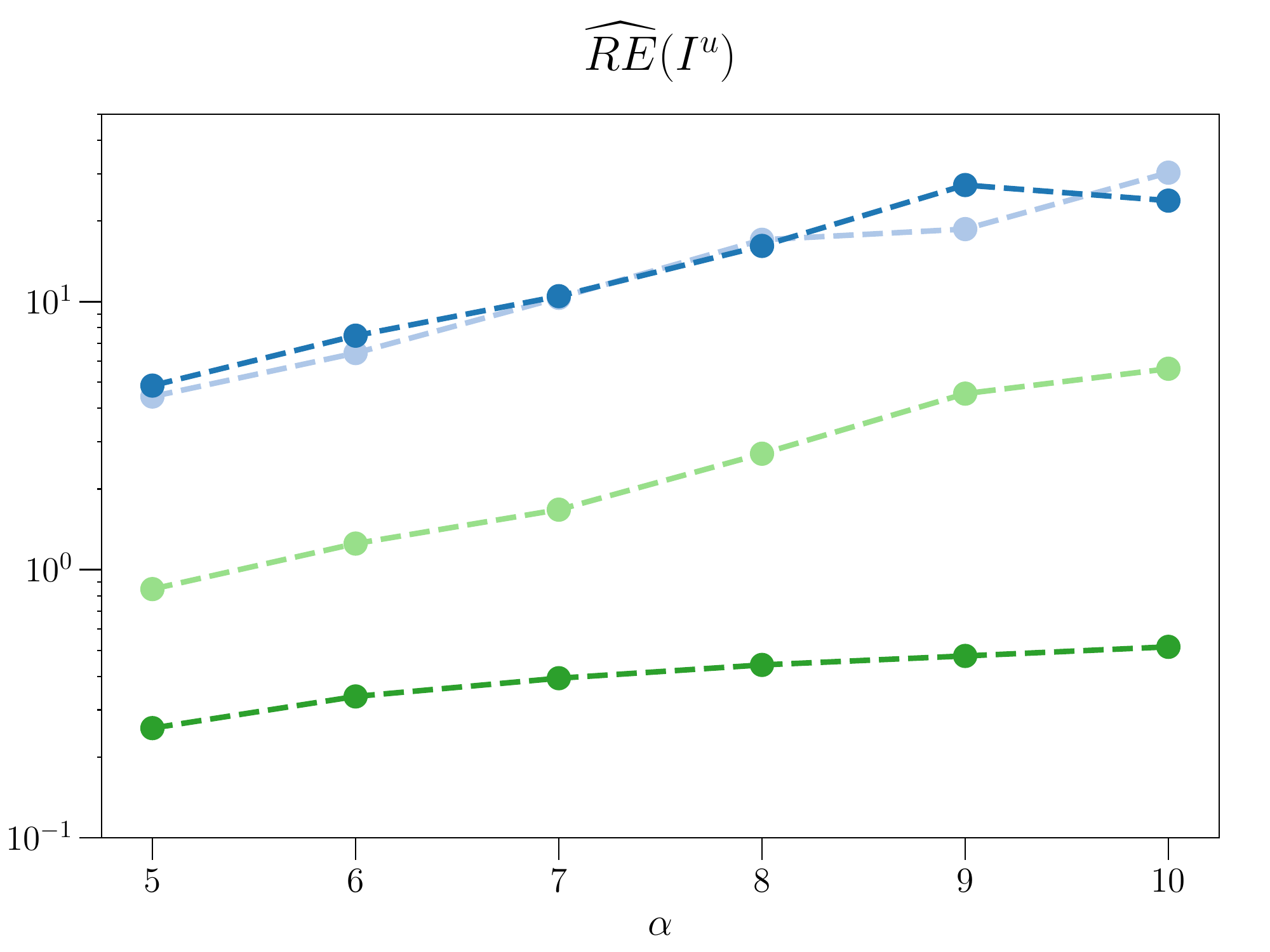}}
\caption{\textbf{(a)} We compare the naive Monte Carlo estimator with optimal importance sampling relying on a discretization of PDE \eqref{eq: bvp psi} for different values of $\alpha$, inverse temperature $\beta=1$ and different time steps $\Delta t $ with sample size $K = 10^3$. The confidence intervals are computed according to \eqref{eq: confidence interval}. \textbf{(b)} Analogously we repeat the comparison for the relative error of the estimators.} 
\label{fig: statistics with optimal control}
\end{figure}
\pagebreak

\section{Numerical strategies for solving the optimal control problem}
\label{sec: numerical strategies}

Numerically solving optimal control problems like the one in \eqref{eq: bvp psi} or (almost equivalently) solving high-dimensional PDEs such as the one stated in \eqref{eq: bvp value function} can be challenging. In particular in high- dimensional settings, this task seems hopeless when relying on classical grid based methods such as finite differences or finite elements since these methods suffer from the curse of dimensionality \cite{Weinan2021algorithms}. We will therefore work with an approach that relies on an optimization procedure aiming to iteratively minimize the cost functional \eqref{eq: cost functional} over a prescribed function class in the spirit of machine learning (cf. \cite{Nusken2021solving}). Our envisioned procedure can be described as follows:
\begin{enumerate}[(i)]
    \item Initialize the control $u$ with an appropriate choice $u^\text{init}$.
    \item Simulate realizations of the controlled process $X^{u}$ as defined in \eqref{eq: controlled langevin sde}.
    \item Compute an estimator of the cost functional $J(u; x)$ as well as its derivative with respect to $u$.
    \item Update the control $u$ by gradient descent.
    \item Repeat steps (ii)--(iv) until convergence.
\end{enumerate}

We argue that two aspects are crucial in order to implement the above scheme in practice. On the one hand, we need to design gradient estimators that are feasible in the sense that they can cope with random stopping times and at the same time exhibit sufficiently low variances. On the other hand, in particular in metastable settings, an appropriate initialization $u^\text{init}$ is important. Otherwise initial gradient information might turn out to be useless, for instance, due to increased variances or due to long trajectory simulations. As a remedy, we suggest a novel simulation algorithm which tries to combine ideas from existing approaches and thereby overcome these known issues. In particular, we will suggest identifying feasible control initializations coming from an adapted version of the heuristic metadynamics algorithm.

\subsection{Gradient computations}
\label{sec: gradient computations}

Let us first address the issue of computing gradients of the cost functional 
\begin{equation}
J(u; x) = \mathbb{E}^x\left[\mathcal{W}(X^u) + \frac{1}{2} \int_0^{\tau^u} |u(X_s^u)|^2 \mathrm ds \right],
\end{equation}
as already defined in \eqref{eq: cost functional}, with respect to the control. An inherent difficulty is that both the running costs and the process depend on the control $u$, and the latter implies that also the hitting time $\tau^u$ depends on $u$. We will approach this difficulty by an appropriate change of the path space measure and first compute a functional derivative in the G\^{a}teaux sense. Subsequently we can relate the rather abstract result to implementable gradients by considering controls $u = u_\theta$ that are parametrized by a parameter vector $\theta \in \mathbb{R}^p$. The G\^{a}teaux derivative can then be identified with a gradient with respect to $\theta$ by considering a special G\^{a}teaux derivative direction. Notably, this strategy will lead to a gradient estimator that can cope with the fact that the random hitting time $\tau^u$ appears both in the integral limit as well as in the terminal costs and depends on the function $u$, with respect to which we differentiate. Eventually, we can numerically compute our gradient estimator by a Monte Carlo approximation.

Let us start by recalling the definition of the G\^{a}teaux derivative.
\begin{definition}[G\^{a}teaux derivative]
We say that $J\colon \mathcal{U} \to \mathbb{R}$ is \emph{G\^{a}teaux differentiable} at $u \in \mathcal{U}$ if for all $\phi\in \mathcal{U}$ the mapping $\varepsilon \mapsto  J(u + \varepsilon \phi; x)$ is differentiable at $\varepsilon=0$. 
The G\^{a}teaux derivative of $J$ in direction $\phi$ is then defined as 
\begin{equation}
    \frac{\delta}{\delta u}J(u; x; \phi) \coloneqq \frac{\mathrm d}{\mathrm d \varepsilon}\Big|_{\varepsilon=0} J(u+ \varepsilon \phi;x).
\end{equation}
\end{definition}\par\bigskip

We can now compute the functional derivative of $J(u; x)$.
\begin{proposition}[G\^{a}teaux derivative of cost functional]
\label{prop: gateaux derivative of control functional}
The G\^{a}teaux derivative of the cost functional defined in \eqref{eq: cost functional} in the direction $\phi\in \mathcal{U}$ is given by
\begin{equation}
\frac{\delta}{\delta u}J(u; x; \phi) = \mathbb{E}^x \biggl[\int_0^{\tau^u} (u\cdot \phi)(X_s^u) \, \mathrm{d}s + \biggl(\mathcal{W}(X^u) + \frac{1}{2} \int_0^{\tau^u} \vert u (X_s^u) \vert^2 \, \mathrm{d}s\biggr) \\ \int_0^{\tau^u} \phi(X_s^u) \cdot \mathrm{d}W_s \biggr].
\end{equation}
\end{proposition}
\begin{proof}
See \Cref{app: proofs} for a proof.
\end{proof}

\Cref{prop: gateaux derivative of control functional} is valid for any direction $\phi\in \mathcal{U}$. Let us note that we are particularly interested in the directions\footnote{We assume that the functions $u_\theta$, $\theta \in \mathbb{R}^p$, as well as all partial derivatives $\frac{\partial}{\partial \theta_i}u_\theta$ lie in $\mathcal{U}$.} $\phi = \frac{\partial}{\partial \theta_i} u$ for all $i \in \{1, \dots, p\}$. This choice is motivated by the chain rule of the G\^{a}teaux derivative, which, under suitable assumptions, states that
\begin{equation}
    \frac{\partial}{\partial \theta_i} J(u_\theta; x) = \frac{\delta}{\delta u}{\Big|}_{u= u_{\theta}}J\Big(u; x; \frac{\partial}{\partial \theta_i} u_\theta\Big).
\end{equation}
We therefore readily get the following formula for the gradient of $J$ with respect to the parameter $\theta$.

\begin{corollary}[Gradient of cost functional]
\label{cor: gradient of cost functional}
Let $u = u_\theta$ be parametrized by the parameter vector $\theta \in \mathbb{R}^p$; then the partial derivatives of the control functional \eqref{eq: cost functional} with respect to the parameters are given by
\begin{align}
\label{eq: partial derivative of cost functional}
\begin{split}
\frac{\partial}{\partial \theta_i}J (u_\theta; x) 
&= \mathbb{E}^x \biggl[\int_0^{\tau^u} \biggl(u_\theta\cdot \frac{\partial}{\partial \theta_i} u_\theta\biggr)(X_s^{u_\theta}) \, \mathrm{d}s \\
&\qquad\quad + \biggl(\mathcal{W}(X^{u_\theta}) + \frac{1}{2} \int_0^{\tau^u} \vert u_\theta (X_s^{u_\theta}) \vert^2 \, \mathrm{d}s\biggr) \int_0^{\tau^u} \left(\frac{\partial}{\partial \theta_i} u_\theta\right)(X_s^{u_\theta}) \cdot \mathrm{d}W_s \biggr]
\end{split}
\end{align}
for any $i \in \{1, \dots, p \}$.
\end{corollary}

\begin{remark}
Note that the gradient given by \eqref{eq: partial derivative of cost functional} is equivalent to the one derived in \cite{Hartmann2012} up to discretization and up to a more general approximating function. For convenience, we repeat the related derivation for general parametrized functions $u_\theta$ in \Cref{app: alternative gradient computations}. Further note that -- contrary to the statement in \cite{Hartmann2012} -- our analysis shows that the gradient is in fact exact, even though it involves the random hitting time $\tau^u$, which depends on $u$.
\end{remark}

In principle, the gradient from \Cref{cor: gradient of cost functional} can be implemented straightforwardly by Monte Carlo approximation. However, even when relying on automatic differential tools, the repeated computation of, for instance, $\frac{\partial}{\partial \theta_i} u_\theta$ might be costly. Let us therefore state a loss functional that is more convenient from a computational point of view, namely
\begin{align}
\label{eq: J_eff}
\begin{split}
J_\text{eff}(u_\theta, u_\vartheta; x) 
&= \mathbb{E}^x \biggl[\frac{1}{2} \int_0^{\tau^u} \vert u_\vartheta (X_s^{u_\theta}) \vert^2 \, \mathrm{d}s \\ 
&\qquad\quad + \biggl(\mathcal{W}(X^{u_\theta}) + \frac{1}{2} \int_0^{\tau^u} \vert u_\theta (X_s^{u_\theta}) \vert^2 \, \mathrm{d}s\biggr) \int_0^{\tau^u} u_\vartheta(X_s^{u_\theta}) \cdot \mathrm{d}W_s \biggr],
\end{split}
\end{align}
which now depends on two parameter vectors $\theta, \vartheta \in \mathbb{R}^p$. It is straightforward to see that the gradient of the actual cost functional \eqref{eq: cost functional}, stated in \Cref{cor: gradient of cost functional}, can then be recovered via
\begin{equation}
\label{eq: autodiff of J}
 \nabla_\vartheta {J}_\text{eff}(u_\theta, u_\vartheta; x) \Big|_{\vartheta=\theta}
= \nabla_\theta {J}(u_\theta; x) .
\end{equation}
In practice, setting $\vartheta=\theta$ only after the differentiation is achieved by removing the parameter $\theta$ from the computational graph of automatic differentiation. Note that with this trick only one backward pass is needed.

Both $\frac{\partial}{\partial \theta_i}J$ and $J_\text{eff}$, as defined in \eqref{eq: partial derivative of cost functional} and \eqref{eq: J_eff}, respectively, can now be approximated by Monte Carlo, yielding, for instance, the estimator
\begin{align}
\label{eq: effective discrete cost functional}
\begin{split}
\widehat{J}_\text{eff}(u_\theta, u_\vartheta; x)
& = \frac{1}{K} \sum_{k=1}^K \biggl(\frac{1}{2} \int_0^{\tau^u} \vert u_\vartheta (X_s^{u_\theta, (k)}) \vert^2 \, \mathrm{d}s \\
& + \biggl(\mathcal{W}(X^{u_\theta, (k)}) + \frac{1}{2} \int_0^{\tau^u} \vert u_\theta (X_s^{u_\theta, (k)}) \vert^2 \, \mathrm{d}s\biggr) \int_0^{\tau^u} u_\vartheta(X_s^{u_\theta, (k)}) \cdot \mathrm{d}W_s^{(k)} \biggr),
\end{split}
\end{align}
where $X^{u_\theta, (k)}$ and $W^{(k)}$ are independent and identically distributed realizations of the controlled process \eqref{eq: controlled langevin sde} and of Brownian motion, respectively.

\subsection{Efficient initializations of stochastic optimization via metadynamics}
\label{sec: metadynamics}

We have so far computed a gradient estimator that allows for stochastic optimization in the spirit of reinforcement learning. To be precise, we can run gradient descent like algorithms that iteratively minimize a suitable objective function with the aim to improve the control which is applied to the dynamics. In this section we shall address the question of how to initialize the approximating function $u \in \mathcal{U}$ in such an iterative optimization procedure. This is in particular important in problems with random hitting times which depend crucially on the applied control $u$ -- as we have already seen in Figures \ref{fig: double well example hitting times}--\ref{fig: statistics with optimal control}. \\

Aiming for reasonable initializations of $u \in \mathcal{U}$ that can in particular cope with strong metastabilities of the dynamics, we suggest relying on ideas from the so-called \textit{metadynamics} algorithm. In its original form \textit{metadynamics} is an adaptive method for sampling the free energy profile of high-dimensional molecular systems \cite{Laio2002,Valsson2016}. Its main intention is related to sampling from corresponding stationary distributions, in particular in systems that exhibit high metastabilities. The approach can be described quite vividly: one iteratively ``fills up'' regions with low energy until the free energy surface can be determined. Those regions can be identified as the ones where the trajectories spend a sufficient amount of time. The filling is specifically achieved by adding Gaussian functions every fixed time interval until the trajectories are eventually able to escape the local minima. 

We use the underlying idea of the metadynamics algorithm in order to compute a reasonable initialization for our optimization procedure. Unlike in the original algorithm, we do not necessarily restrict this approach to the reduced collective variable space, but allow as well for the full space (cf. \Cref{rem: reaction coordinates}). Furthermore we adapt the stopping criterion used in the original version, as will be detailed later.

Analogously to \eqref{eq: optimal potential}, the idea is to modify the potential via
\begin{equation}
    V + V_\text{bias},
\end{equation}
where now the bias potential is given by a sum of unnormalized Gaussian functions, i.e.,
\begin{equation}
\label{eq: meta bias potential}
V_{\text{bias}}(x; \eta, \mu, \Sigma) = \sum\limits_{m=1}^M \eta_m \widetilde{\mathcal{N}}(x; \mu_m, \Sigma_m),
\end{equation}
where $\widetilde{\mathcal{N}}(\,\cdot\,; \mu_m, \Sigma_m)$ is an unnormalized density of a multivariate normal distribution, i.e.,
\begin{equation}
    \widetilde{\mathcal{N}}(x; \mu_m, \Sigma_m) = \exp\left(-\frac{1}{2}(x - \mu_m)\cdot \Sigma_m^{-1} (x - \mu_m)  \right),
\end{equation}
with mean $\mu_m \in \mathbb{R}^d$, covariance matrix $\Sigma_m \in \mathbb{R}^{d \times d}$, and $\eta_m \in \mathbb{R}$ being an appropriate weight. The intuition is that the added bias potential prevents the trajectory from going back to the already visited states. This bias potential can also be interpreted as a control function. The control resulting from the bias potential is then given by
\begin{equation}
\label{eq: metadynamics control initialization}
    u^\text{init}(x) = -\sigma^{-1}(x) \nabla V_{\text{bias}}(x; \eta, \mu, \Sigma) = - \sigma^{-1}(x) \sum\limits_{m=1}^M \eta_m  \nabla \widetilde{\mathcal{N}}(x; \mu_m, \Sigma_m).
\end{equation}

The intuition behind this strategy is that we essentially want to ``fill up'' those local minima of the potential which influence the estimation of the quantity of interest. Note that this has already been considered in \cite{Quer2018}, however, not in combination with optimal control ideas. We then take the resulting control $u^\text{init}$ as a rough initial guess of the optimal control, which is likely to push trajectories over the energy barrier such that trajectories are not trapped in metastable regions anymore. 
The expected advantages of such a control initialization are twofold. On the one hand the trajectories (in particular at the beginning of the control optimization) are expected to be much shorter, which reduces the runtime significantly. On the other hand we expect the variances of the gradient estimators, as for instance defined in \Cref{cor: gradient of cost functional}, to be smaller, which will result in faster convergence of gradient descent algorithms. Crucially, we should note that without our metadynamics based initialization method the optimization might not even converge at all. We refer to \Cref{sec: numerical examples} for illustrative examples of those aspects. The basic principle of our metadynamics based initialization method is illustrated in \Cref{fig: metadynamics algorithm}. \\

\begin{figure}[tbhp]
\centering
\subfloat[]{\label{fig: 4a}\includegraphics[width=60mm]{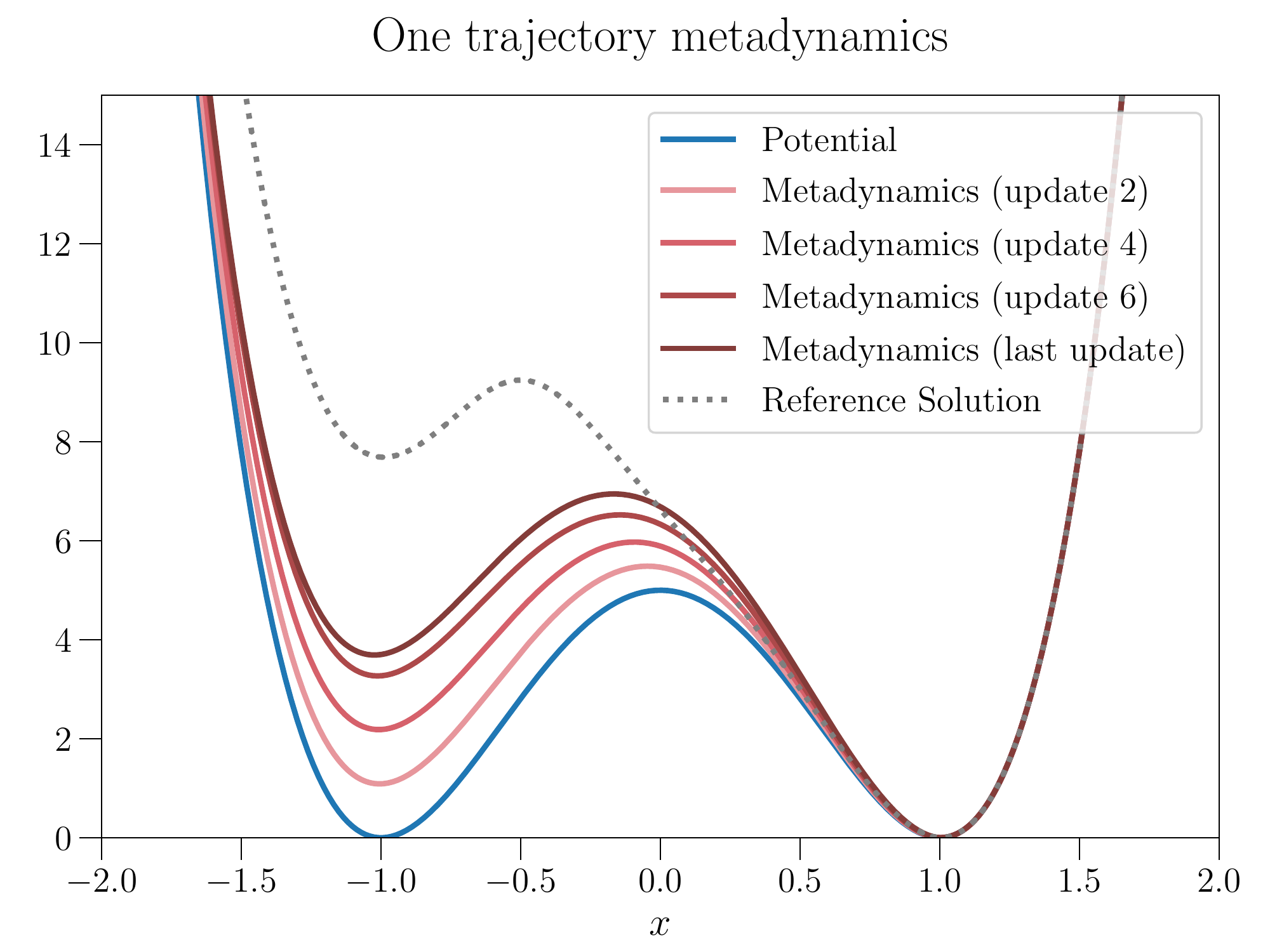}}
\subfloat[]{\label{fig: 4b}\includegraphics[width=60mm]{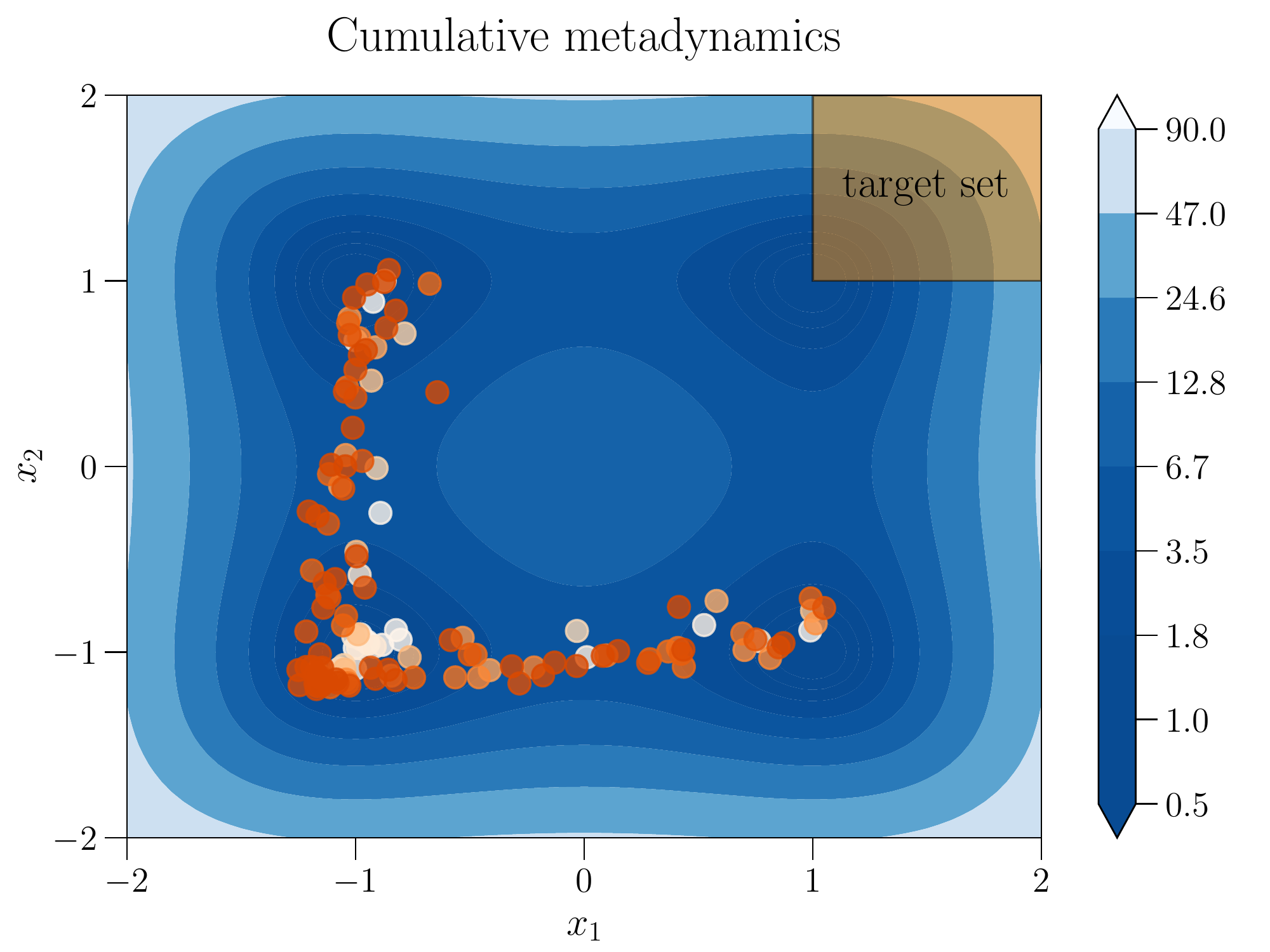}}
\caption{\textbf{(a)} Perturbed potential for different updates of $V_\text{bias}^{(m)}$ during \Cref{alg: one trajectory metadynamics} for the double well potential \Cref{ex: double well} with $\alpha=5$ and $\beta=1$. \textbf{(b)} Centers of the added unnormalized Gaussian functions after \Cref{alg: cumulative metadynamics} for the two-dimensional extensional of the double well potential. The color of the centers represent the weights of the added functions.}
\label{fig: metadynamics algorithm}
\end{figure}

In the following let us suggest two versions of a metadynamics based initialization algorithm. The first one builds the bias potential $V_\text{bias}$ by sampling just one trajectory, as already considered in \cite{Quer2018}. This trajectory follows the dynamics of the controlled stochastic process \eqref{eq: controlled langevin sde} until it hits the target set $\mathcal{T}$. In particular, the control $u$ is modified on the fly after each specified time interval $\delta$ by adding another unnormalized Gaussian function to the potential according to \eqref{eq: meta bias potential} with $\mu_m$ being the averaged position of the particle over the last time interval. 

This method ensures that metastable regions get ``filled up'' when visited by the trajectory. The time interval $\delta$, the covariance matrix of the Gaussians $\Sigma_m$, and the weights $\eta_m$ should be chosen such that the original potential is not perturbed too much, although still allowing for a significant reduction of the hitting time $\tau^u$. It is possible that using different versions of the metadynamics algorithm like \textit{well-tempered metadynamics} can lead to further improvements of the initialization procedure. For simplicity, however, we choose constant weights and covariance matrices. Let us summarize our first method in \Cref{alg: one trajectory metadynamics}.

\begin{algorithm}
\caption{One trajectory adapted metadynamics}
\label{alg: one trajectory metadynamics}
\begin{algorithmic}[1]
\STATE{Consider specified target set $\mathcal{T}$ and dynamics $X^u$ as in \eqref{eq: controlled langevin sde}.}
\STATE{Choose time interval $\delta$, weight $\eta > 0$ and covariance matrix $\Sigma \in \mathbb{R}^{d \times d}$.}
\STATE{Initialize $X_0^u = x$, $V_\text{bias}^{(0)} = 0, m=0$.}
\WHILE{trajectory has not arrived in $\mathcal{T}$ }
\STATE{Run dynamics $X^u$ with control $u = - \sigma^{-1}\nabla V_\text{bias}^{(m)}$ for time interval $\delta$.} 
\STATE{Choose $\mu_m = \frac{1}{\delta}\int_{m \delta}^{(m+1)\delta} X_s^u \, \mathrm ds$.}
\STATE{Adapt bias potential $V_\text{bias}^{(m+1)} \gets V_\text{bias}^{(m)} + \eta\, \widetilde{\mathcal{N}}(\,\cdot \,; \mu_m, \Sigma)$.}
\STATE{$m \gets m+1$.}
\ENDWHILE
\STATE{Set $V_\text{bias} =V_\text{bias}^{(m)}$.}
\end{algorithmic}
\end{algorithm}

In high-dimensional settings the considered dynamics often exhibits multiple metastable regions. In this case, a bias potential provided by only one trajectory might not be sufficient since we cannot guarantee the trajectory to visit all metastable regions before hitting the target set. Therefore, only some of those regions might be ``filled up.'' To deal with this issue, we suggest sampling multiple trajectories and execute corresponding bias potential modifications cumulatively. To be precise, each trajectory starts at the same initial position, one after the other. The first trajectory starts with the zero bias initialization, the second one considers the bias potential that came out after running the first trajectory, and so on. Ideally one would like to stop sampling trajectories once the bias potential is already properly ``filled up,'' but this can be difficult to determine. In any case, it is unlikely that we perturb the potential much more than needed since otherwise trajectories would hit the target set even before the time interval $\delta$ elapsed. Furthermore, we introduce the scaling factor $r \in (0, 1)$ that shall stabilize the procedure by reducing the effect of adding further Gaussians for each trajectory.
The cumulative method is summarized in \Cref{alg: cumulative metadynamics}.

\begin{remark}[Metadynamics in reaction coordinates]
\label{rem: reaction coordinates}
For very high-dimensional systems it may help to apply our metadynamics based algorithms in the space of collective variables. To this end, let us assume that a \textit{reaction coordinate} $\xi: \mathbb{R}^d \rightarrow \mathbb{R}^s$ is available, where $s < d$, mapping the full coordinate space into the space of collective variables and, loosely speaking, reducing the dynamics to its key features. Note that finding this projection of the corresponding dynamics, often known as \emph{effective dynamics}, is a challenging problem and that it may not follow the structure of the original diffusion process. However, let us assume that $\xi$ is such that the corresponding effective dynamics does not lose the dominant timescales of the original dynamics. Then our Algorithms \ref{alg: one trajectory metadynamics} and \ref{alg: cumulative metadynamics} can be applied for the effective dynamics of the controlled process. Note that the unnormalized Gaussian functions then live in the collective variable space and so does the corresponding bias potential. However, by using the composition between the unnormalized Gaussians and the reaction coordinate one can also express the control resulting from the bias potential (recall \eqref{eq: metadynamics control initialization}) in the complete state space by
\begin{equation}
\label{eq: metadynamics cv control initialization}
u^\text{init}(x) 
= - \sigma^{-1}(x) \sum\limits_{m=1}^M \eta_m  \nabla (\widetilde{\mathcal{N}} \cdot \xi)(x; \mu_m, \Sigma_m),
\end{equation}
where the gradient of the composition between the unnormalized Gaussian function and the reaction coordinate is given by the multivariate version of the chain rule,
\begin{equation}
\nabla (\widetilde{\mathcal{N}} \circ \xi)(x) = \textbf{J}_{\widetilde{\mathcal{N}}}(y)|_{y=\xi(x)} \cdot \textbf{J}_\xi(x),
\end{equation}
where $\textbf{J}_h$ represents the Jacobian matrix for a vector-valued function $h$.
\end{remark}

\begin{algorithm}
\caption{Cumulative adapted metadynamics}
\label{alg: cumulative metadynamics}
\begin{algorithmic}[1]
\STATE{Consider specified target set $\mathcal{T}$ and dynamics $X^u$ as in \eqref{eq: controlled langevin sde}.}
\STATE{Choose time interval $\delta$, number of trajectories $K^\text{meta}$, weight $\eta > 0$, scaling factor $r \in (0, 1)$, and covariance matrix $\Sigma \in \mathbb{R}^{d \times d}$.}
\STATE{Initialize $V_\text{bias}^{(0)} = 0$.}
\FOR{$k \in \{1, \dots, K^\text{meta}\}$}
\STATE{Set $X_{0}^{u,(k)} = x, m=0$.}
\WHILE{trajectory has not arrived in $\mathcal{T}$}
\STATE{Run dynamics $X^{u,(k)}$ with control $u = - \sigma^{-1}\nabla V_\text{bias}^{(m)}$ for time interval $\delta$.}
\STATE{Choose $\mu_m = \frac{1}{\delta}\int_{m \delta}^{(m+1)\delta} X_s^{u,(k)} \, \mathrm ds$.}
\STATE{Adapt bias potential $V_\text{bias}^{(m+1)} \gets V_\text{bias}^{(m)} + r^{k-1}\eta\, \widetilde{\mathcal{N}}(\,\cdot \,; \mu_m, \Sigma)$.}
\STATE{$m \gets m+1$.}
\ENDWHILE
\ENDFOR
\STATE{Set $V_\text{bias} =V_\text{bias}^{(m)}$.}
\end{algorithmic}
\end{algorithm}

\subsection{Control function approximations}
\label{sec: control function approximation}

Finally, we need to specify how to approximate the control function $u \in \mathcal{U}$. The general idea is to rely on parametrized functions $u_\theta$, specified by the parameter vector $\theta \in \mathbb{R}^p$. In particular, we consider a linear combination of ansatz functions (Galerkin approach) as well as neural network approximations. The former match well with the structure of the metadynamics based initialization algorithm that we have introduced before, but suffer from the curse of dimensionality. The latter seem well suited for high-dimensional problems, but need an additional step in order to benefit from our initialization strategy. Note that either function space needs to be sufficiently large in order to approximate the optimal control $u^*$ well enough.

In the Galerkin approach the control $u$ is projected onto a space consisting of finitely many ansatz functions. A clever choice of ansatz functions depends on the problem at hand and one might, for instance, consider radial symmetric functions, polynomials, or piecewise linear functions with Chebyshev coefficients; see, e.g., \cite{Hartmann2016}. As a related method let us mention tensor train approximations and refer to \cite{Fackeldey2022approximative, Richter2021solving} for further details. In this work we rely on Gaussian ansatz functions since they match well with the aforementioned initialization strategy. To be precise, let us choose the control approximation
\begin{equation}
\label{eq: Gaussian ansatz functions}
    u_\theta(x) = \sum\limits_{i=1}^p \theta_i  \nabla \mathcal{N}(x; \mu_i, \Sigma_i),
\end{equation}
where $\mathcal{N}(\,\cdot\,; \mu_i, \Sigma_i)$ is the density of a multivariate normal distribution with mean $\mu_i \in \mathbb{R}^d$ and covariance matrix $\Sigma_i \in \mathbb{R}^{d \times d}$, as in \eqref{eq: metadynamics control initialization}.

Feed-forward neural networks, on the other hand, are nonlinear functions that exhibit remarkable approximation properties \cite{Bach2017breaking, Jentzen2018proof}. They essentially consist of compositions of affine-linear maps and nonlinear activation functions. In particular, we define a \textit{feed-forward neural network} $u_\theta:\mathbb{R}^{d_0} \to \mathbb{R}^{d_L}$ with $L$ layers by
\begin{equation}
\label{eq: feed-foreward nn}
u_\theta(x) = A_L \rho(A_{L-1} \rho(\cdots  \rho(A_1 x + b_ 1) \cdots) + b_{L-1}) + b_L
\end{equation}
with matrices $A_l \in \mathbb{R}^{d_{l} \times d_{l-1}}$, vectors $b_l \in \mathbb{R}^{d_l}, 1 \le l \le L$, and a nonlinear activation function $\rho: \mathbb{R} \to \mathbb{R}$ that is applied componentwise. The collection of matrices $A_l$ and vectors $b_l$ comprises the learnable parameters $\theta$. For our control approximations, we can now choose $d_0 = d_L=d$. Note that the choice of the so-called architecture of neural networks, i.e., the number of parameters in each layer, is not always straightforward and requires some fine-tuning.

For initializing $u_\theta$ with the control $u^\text{init}$ obtained by one of the two adapted metadynamics algorithms, which have been suggested in \Cref{sec: metadynamics}, we can consider a least squares minimization on a given domain $\mathcal{D}$. That means we minimize the loss
\begin{equation}
\label{eq: approximation problem loss}
\mathcal{L}_\text{init}(\theta) \coloneqq \mathbb{E}\left[ \big| u_\theta(X) - u^\text{init}(X)\big|^2 \right], 
\end{equation}
where $X \sim \nu$ is sampled from a prescribed measure $\nu$ that has full support on the domain $\mathcal{D}$, e.g., the uniform measure. For the parametric approximation we can solve the minimization of the loss explicitly by solving a least squares problem. When considering neural networks we have to minimize $\mathcal{L}_\text{init}$ by some variant of gradient descent where the different parameters, such as batch size and stopping criterion, have to be chosen depending on the problem. Further details on the applied minimization method are provided in \Cref{sec: numerical examples}. For this minimization and for the computation of the gradient of the control cost, we rely on automatic differentiation tools such as PyTorch. For convenience let us state our final algorithm.

\begin{algorithm}[H]
\caption{Efficient importance sampling}
\label{alg: Efficient importance sampling}
\begin{algorithmic}[1]
\STATE{Choose a metadynamics time interval $\delta$, weight $\eta > 0$ (potentially a scaling factor $r \in (0, 1)$), covariance matrix $\Sigma \in \mathbb{R}^{d \times d}$, a function approximation $u_\theta$, a gradient based optimization algorithm, a corresponding learning rate $\lambda > 0$, a sample size $K$, a step size $\Delta t$, and a stopping criterion.}
\STATE{Compute $u^\text{init}$ by either Algorithm \ref{alg: one trajectory metadynamics} or \ref{alg: cumulative metadynamics}.}
\STATE{Initialize $u_\theta \approx u^\text{init}$ by minimizing $\mathcal{L}_\text{init}(\theta)$.}
\REPEAT
\STATE{Simulate $K$ samples of $X^{u, (k)}$.}
\STATE{Compute the gradient of the control cost estimator $\widehat{J}_\text{eff}$ defined in \eqref{eq: effective discrete cost functional} via automatic differentiation.}
\STATE{Update the parameters based on the optimization algorithm.}
\UNTIL{stopping criterion is fulfilled.}
\STATE{Do importance sampling according to \eqref{eq: expectation IS} with the control $u_\theta$.}
\STATE{\textbf{Result:} Low-variance estimate of $\Psi$ as defined in \eqref{eq: psi}.}
\end{algorithmic}
\end{algorithm}

\section{Numerical examples}
\label{sec: numerical examples}

In this section we demonstrate that our proposed \Cref{alg: Efficient importance sampling} can indeed lead to low-variance estimators of observables that involve random stopping times. In particular, we will show in both low- and high-dimensional metastable examples that the combination of control based importance sampling together with reasonable initializations leads to improved estimators. Throughout we will consider the overdamped Langevin equation as stated in \eqref{eq: controlled langevin sde} with $\sigma(x)=\sqrt{2 \beta^{-1}} \operatorname{Id}$ on the domain $\mathcal{D}=[-3, 3]^d$. We consider a multidimensional extension of the double well potential\footnote{Notice that even though the potential is symmetric in all dimensions we cannot decouple the estimation or control problem, respectively.}
\begin{equation}
\label{eq: nd double well potential}
{V_\alpha(x) = \sum\limits_{i=1}^d \alpha_i (x_i^2 - 1)^2},
\end{equation}
where the parameter $\alpha \in \mathbb{R}^d$ as well as the inverse temperature $\beta$ encode the strength of the metastability. We aim to compute the quantity
\begin{equation}
    \mathbb{E}\left[e^{-\tau} \right]
\end{equation}
by choosing $f = 1$ and $g = 0$ in the observable \eqref{eq: work functional}. If not stated differently we set the initial value of the process to $x = (-1, \dots, -1)^\top$ for all examples. The inverse temperature is set to $\beta = 1$ so that the metastability is mainly influenced by the choice of $\alpha$. If not otherwise stated the target set is chosen to be $\mathcal{T} = [1, 3]^d$.

In our numerical simulations we discretize the controlled stochastic process in time $0 < t_1 < \cdots < t_{\widetilde{N}}$ using the Euler--Maruyama scheme
\begin{equation}
\label{eq: discrete controlled langevin sde}
\widehat{X}_{n+1}^u = \widehat{X}_{n}^u  + (- \nabla V(\widehat{X}_n^u) + \sqrt{2 \beta^{-1}} \, u(\widehat{X}_n^u))\Delta t + \sqrt{2 \beta^{-1}} \,  \xi_{n+1}\sqrt{\Delta t} , \qquad \widehat{X}^u_{0} = x,
\end{equation}
where $\Delta t = t_{n+1} - t_n$ is a time step and $\xi_{n+1} \sim \mathcal{N}(0, \operatorname{Id})$ is a standard normally distributed random variable \cite{Higham2001}. Note that the length of each discrete trajectory is random according to $\widetilde{N} = \floor{\tau^u / \Delta t}$. For each experiment we monitor the importance sampling mean as the Monte Carlo estimator of \eqref{eq: expectation IS} and its variance and relative error accordingly. For the Monte Carlo estimators we compute confidence intervals by
\begin{equation}
\label{eq: confidence interval}
\widehat{\Psi}(x) \pm 1.96 \, \frac{\sqrt{\widehat{\operatorname{Var}}(I^u)}}{\sqrt{K}},
\end{equation}
where $\widehat{\operatorname{Var}}$ is the estimated variance computed with a sample size $K_\text{Var}=10^5$, and $K$ the sample size of the Monte Carlo estimator. We also keep track of mean first hitting time $\mathbb{E}[\tau^u]$ and the time needed for the last trajectory of an ensemble to reach the target set $\mathcal{T}$.

For dimensions $d \le 3$ we can compute reference solutions for the HJB equation \eqref{eq: bvp psi} (and therefore for the optimal control $u^*_\text{ref}$) by a finite difference method. We will use those for comparing against an (up to time discretization) optimal sampling efficiency in terms of relative errors. Furthermore, we compute an $L^2$ type error of our approximations along the controlled trajectories, i.e.,
\begin{equation}
L^2(u) \coloneqq \mathbb{E}\left[ \int_0^{\tau^{u}} |u - u^*_\text{ref}|^2 (X_s^{u}) \, \mathrm ds \right].
\end{equation}

The control approximation with Gaussian ansatz functions is done according to \eqref{eq: Gaussian ansatz functions} with $p$ Gaussians uniformly distributed over the domain $\mathcal{D}$. The covariance matrix is constant, $\Sigma_i = 0.5 \operatorname{Id}$ for all $i \in \{1, \dots, p \}$ and the number of ansatz functions changes depending on the example.
For the neural network representation we consider a feed-forward neural network according to \eqref{eq: feed-foreward nn} with two hidden layers, $d_1 = d_2 = 30$ and activation function $\rho(x) = \tanh(x)$. The initialization of $u_\theta$ with the control $u^\text{init}$ is achieved after minimizing the mean squared error loss \eqref{eq: approximation problem loss} by using the Adam algorithm with learning rate $\lambda=0.01$ \cite{kingma2014adam}. If not otherwise specified the training data points for this approximation problem have been uniformly sampled from the domain just one time and have been used for all gradient steps. A total of $10^3$ gradient steps suffices to obtain a good approximation. In order to have fair comparisons we set the control $u^{(0)}$ to be the zero function when not considering a metadynamics based initialization.

Moreover, the control optimization in \Cref{alg: Efficient importance sampling} is implemented using the Adam algorithm with learning rate $\lambda=0.01$. If not otherwise stated, the batch size is set to be $K=10^3$ and the time step $\Delta t = 0.005$. We repeat all of our experiments multiple times with different random seeds and different time intervals $\delta$ in order to guarantee generalizability. Each experiment requires just one CPU core and the maximum value of allocated memory is set to 100GB.

\subsection{Metastable double well potential}
Let us start with a one-dimensional metastable example for which $\alpha = 5$. We approximate the control with neural networks or Gaussians ansatz functions, where both are initialized with either the zero function or an initialization given by \Cref{alg: one trajectory metadynamics}, for which we choose $\delta = 0.2$, $\eta=1.0$, and $\sigma^2=0.5$. For this example we consider a finer time step $\Delta t = 0.001$ and a batch size of $K = 1000$. The resulting modified potential consists of $M=7$ unnormalized Gaussian functions. For the control approximations with ansatz functions we choose $p=50$ Gaussians. \Cref{fig: 1d re and l2} shows the relative error of the importance sampling estimator as well as the $L^2$ approximation error as a function of the gradient steps. We can see that the neural network performs better and that the control initializations speed up the convergence significantly. Note that the learned importance sampling control leads to a similar relative error compared to a reference optimal control.
In \Cref{fig: 1d control and perturbed potential} we display the approximated functions, once as the control and once as the perturbed potential. We can see that in particular the neural network approximation agrees well with the reference solution, whereas both the Gaussian approximation and the metadynamics attempt without control optimizations based control are off.
\pagebreak

\begin{figure}[tbhp]
\centering
\subfloat[]{\label{fig: 5a}\includegraphics[width=60mm]{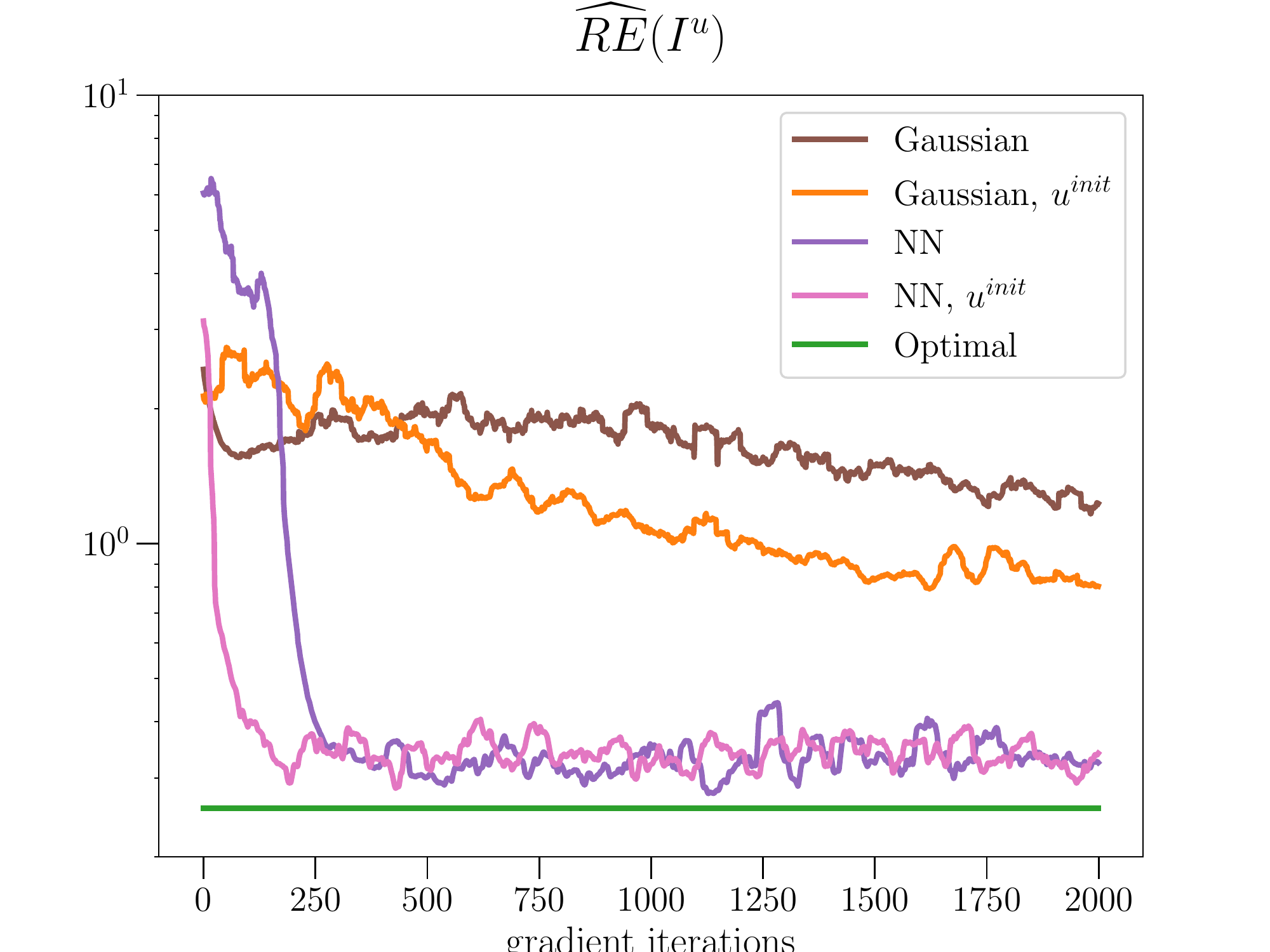}}
\subfloat[]{\label{fig: 5b}\includegraphics[width=60mm]{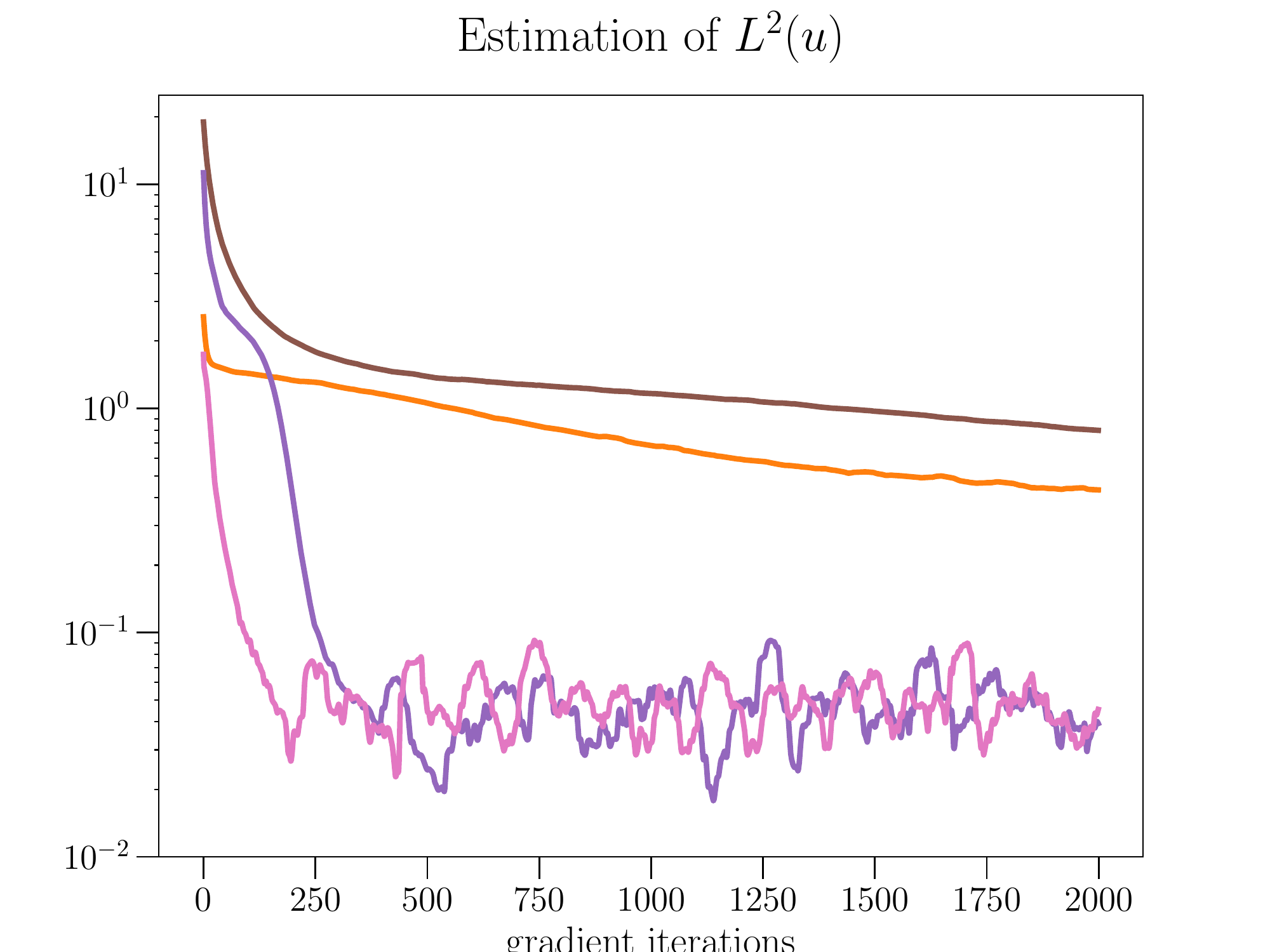}}
\caption{\textbf{(a)} Importance sampling relative error and \textbf{(b)} estimation of $L^2(u)$ at each gradient step.}
\label{fig: 1d re and l2}
\end{figure}

\begin{figure}[tbhp]
\centering
\subfloat[]{\label{fig: 6a}\includegraphics[width=60mm]{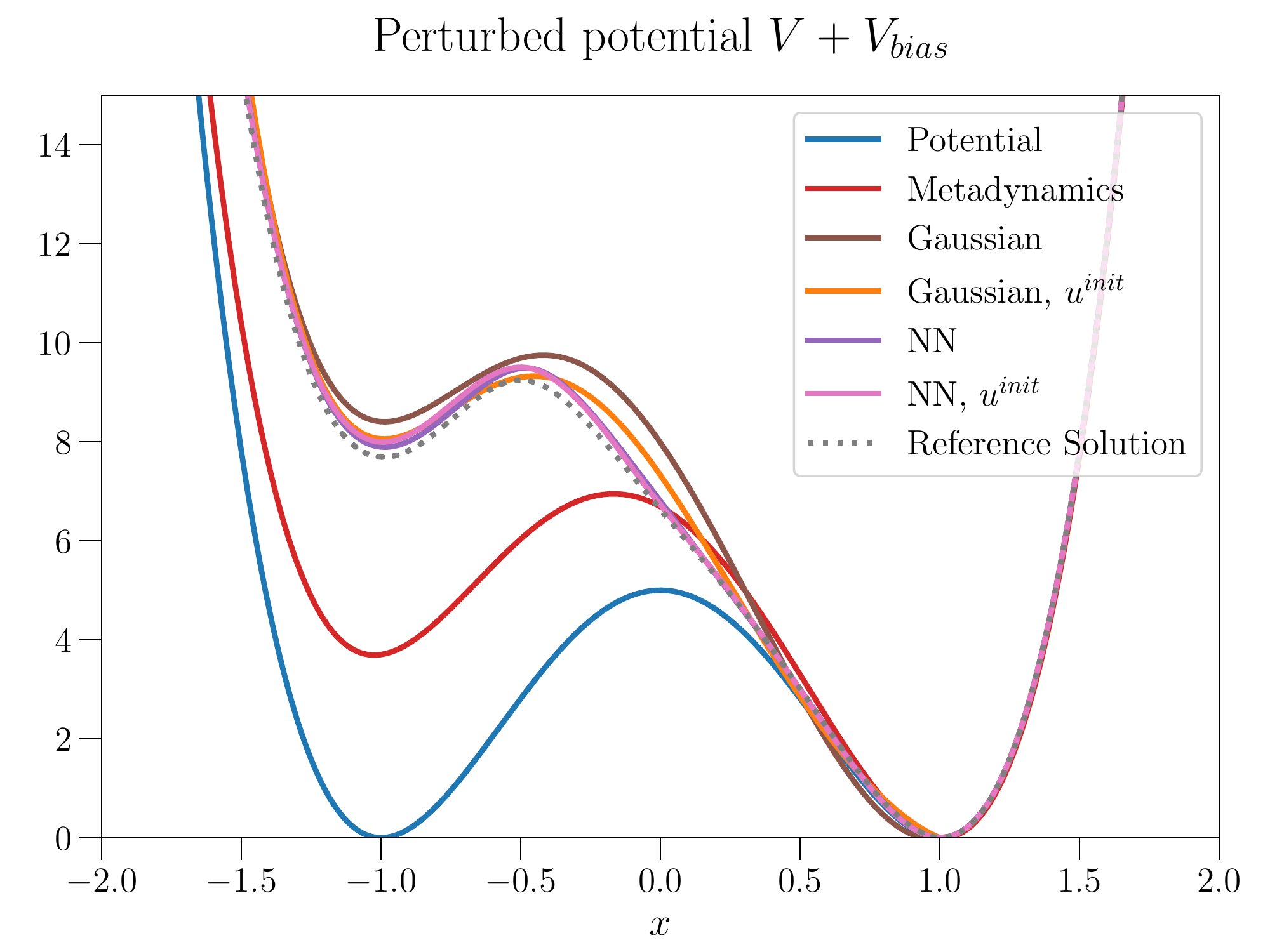}}
\subfloat[]{\label{fig: 6b}\includegraphics[width=60mm]{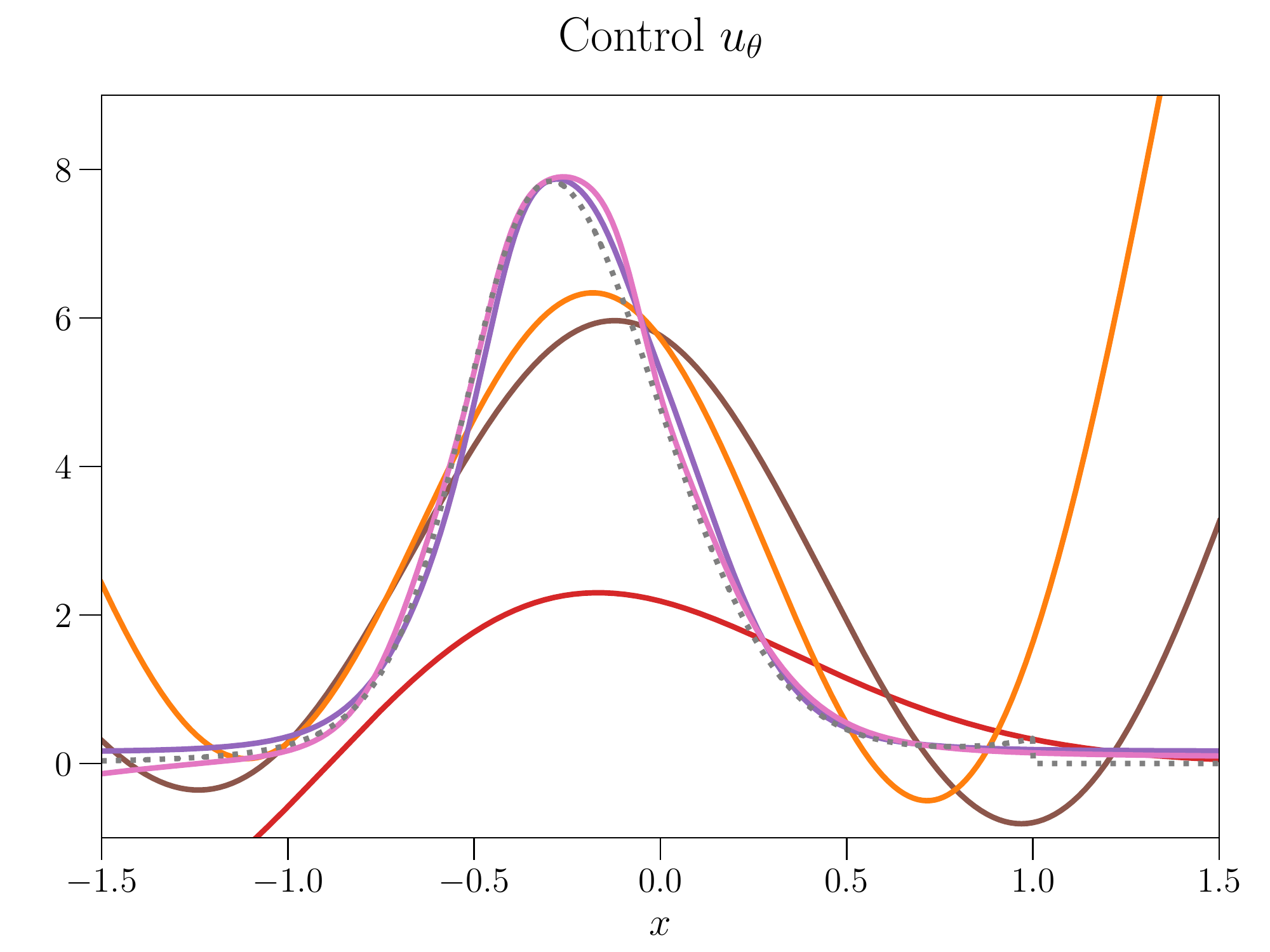}}
\caption{\textbf{(a)} Perturbed potential and \textbf{(b)} control after convergence.}
\label{fig: 1d control and perturbed potential}
\end{figure}

We have stated before that the metastability depends on the parameter $\alpha$. Let us therefore vary this parameter and compare performances of the following schemes against each other: naive Monte Carlo (MC sampling), a bias constructed by an adapted metadynamics algorithm without optimization as proposed in \cite{Quer2018} (Metadynamics), a Gaussian control representation without initialization (Gaussians) as well as with the metadynamics based initialization (Gaussians, $u^\text{init}$), a neural network representation without initialization (NN) as well as with initialization (NN, $u^\text{init}$), a sampling with the discretized optimal control calculated with a finite difference method (Optimal), and the reference solution from the PDE (Reference Solution).

In \Cref{fig: 1d psi and re} we display the estimator as well as the relative errors. We can see that the estimation of the expectation value gets worse with increasing value of $\alpha$, in particular when relying on naive Monte Carlo estimation, the Gaussian control approximations, or the metadynamcis algorithm only. We should highlight that without control initialization we are not able to get results for $\alpha > 6$ for the importance sampling estimators since corresponding control based optimization algorithms exceed the memory constraints. The reason for this is that the first sampling of the gradient estimator takes very long and thus the allocated memory capacity is exceeded. With the adapted metadynamics based initializations, on the other hand, we can observe that the optimal control importance sampling strategies yield valid estimators with low relative error even for large metastabilities. However, the neural network representation with initialization results in a much smaller relative error than the Gaussian representation also with initialization. Moreover, we stress the fact that doing importance sampling right after metadynamics does not guarantee satisfactory results. This is probably due to $u^\text{init}$ still being off from $u^*$, noting that there is an exponential dependency on the variance in the distance of the used control to the optimal control \cite{Hartmann2021nonasymptotic}.

\begin{figure}[tbhp]
\centering
\subfloat[]{\label{fig: 7a}\includegraphics[width=60mm]{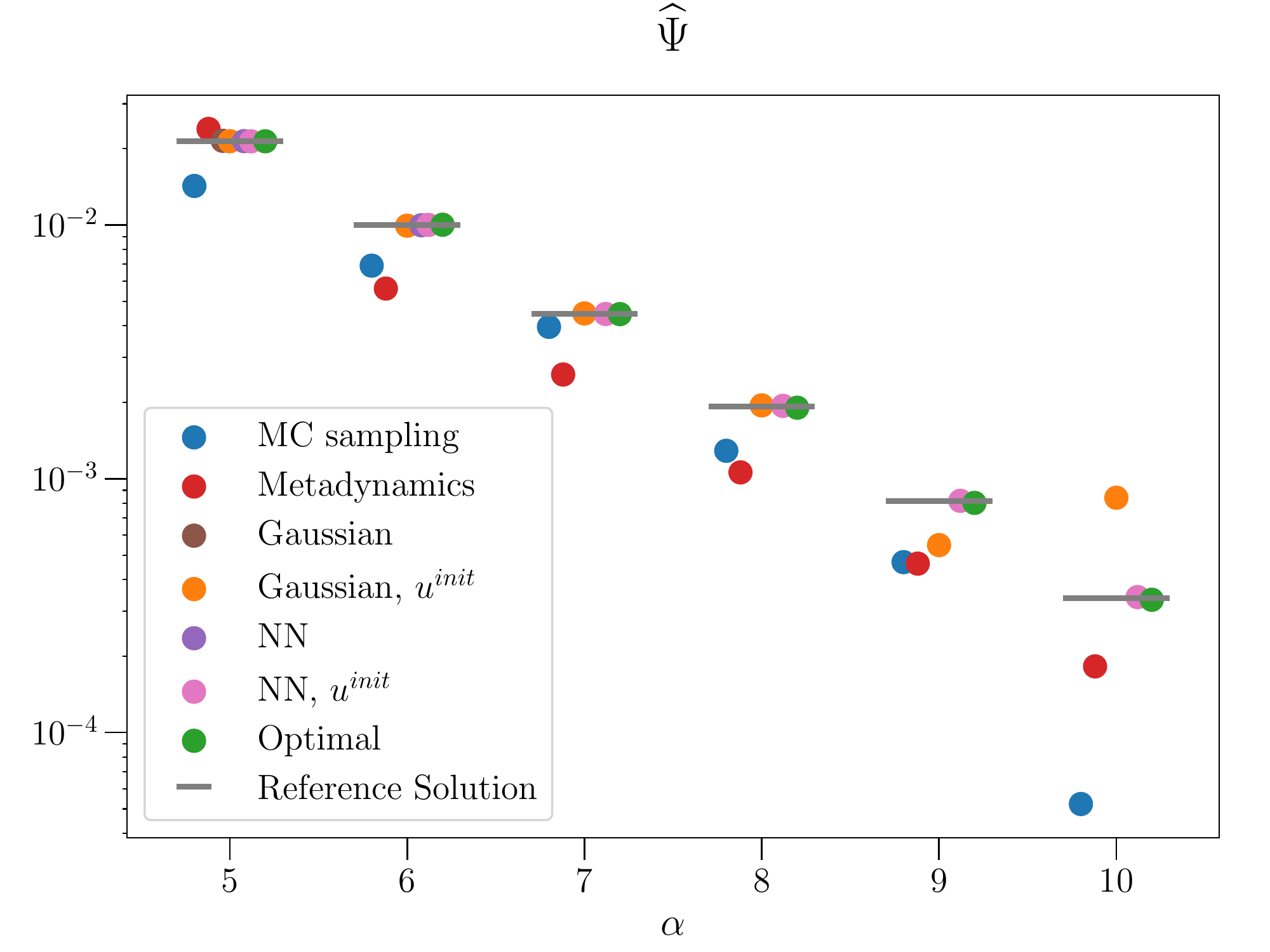}}
\subfloat[]{\label{fig: 7b}\includegraphics[width=60mm]{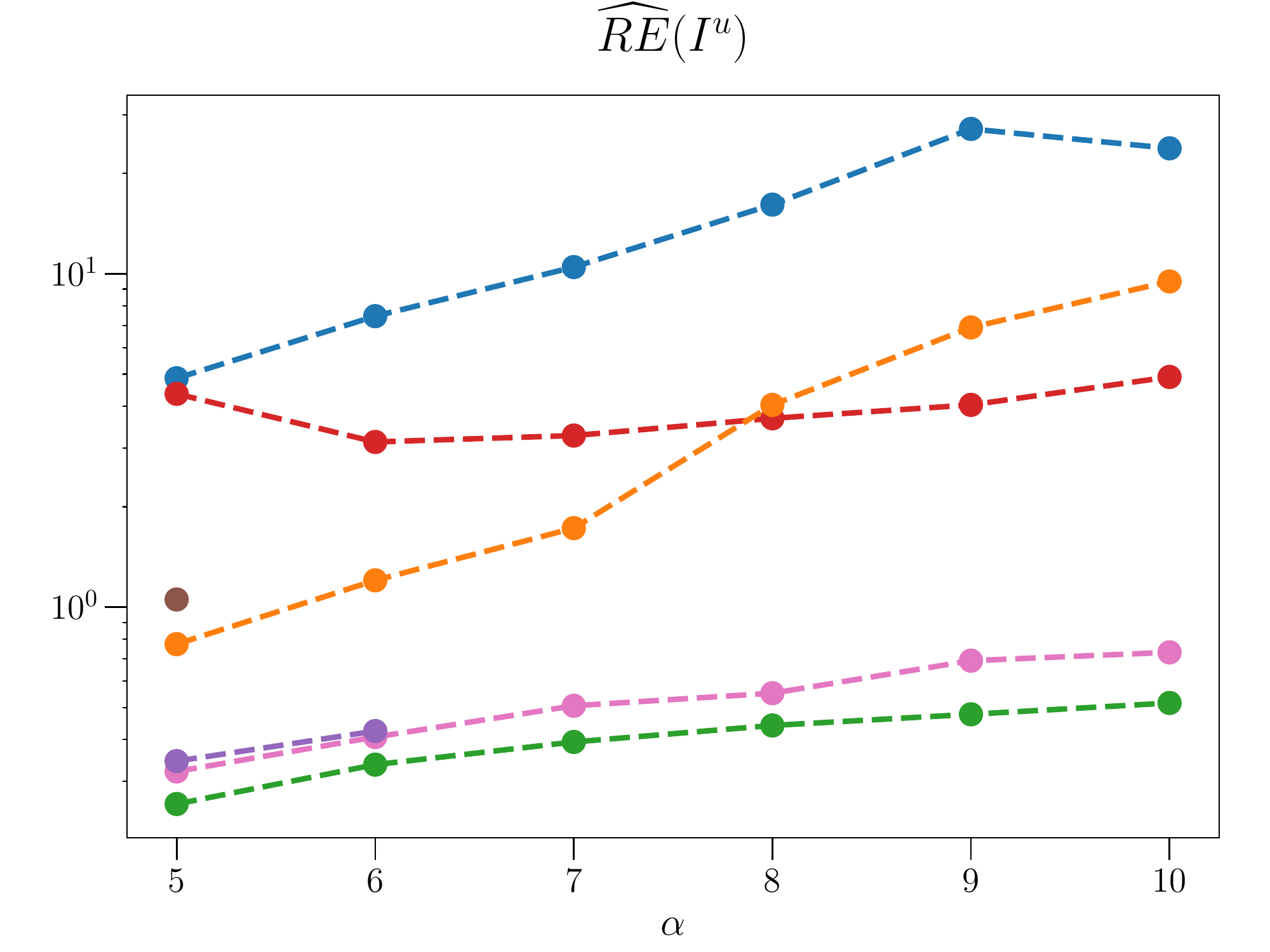}}
\caption{\textbf{(a)} Estimation $\widehat{\Psi}$ and \textbf{(b)} importance sampling relative error of the different methods described above.}
\label{fig: 1d psi and re}
\end{figure}

\subsection{Multi-dimensional extension of the double well potential}
Let us now repeat the above experiment for multidimensional problems. We start with $d = 2$, for which we can still compute a reference solution. 
This example is followed by an example in $d=4$ for which the PDE \eqref{eq: bvp psi} cannot be discretized anymore due to the curse of dimensionality. Here we compute a reference value of $\Psi$ by Monte Carlo estimation using a very large batch size.

\subsubsection*{Example in $\mathbf{d = 2}$}
For the $2d$ example we choose $\alpha = (5, 5)^\top$. As before we compute the metadynamic based initialization of the control according to \Cref{alg: one trajectory metadynamics} with time interval $\delta = 1.0$, weight $\eta = 1.0$, and covariance matrix $\Sigma = 0.5 \, \operatorname{Id}$. The resulting modified potential consist of $M=10$ unnormalized Gaussian functions. For the linear combination of ansatz functions we choose $p=100$ Gaussians, again placed on an equidistant grid in the domain $\mathcal{D}$.

Let us highlight two important aspects of the experiment. First, we see in \Cref{fig: 2d re and l2} that even after a runtime of $10^4$ seconds both Gaussian ansatz approximation attempts in contrast neural networks have not converged. We have observed in our experiments that the Gaussian approximation is sensible with respect to the number of ansatz functions, their placing in the considered domain, and the choice of the covariance matrix. Note that such hyperparameters do not have to be tuned for neural networks. Second, we can observe that the optimization initialized with the adapted metadynamics algorithm results in a faster convergence. In \Cref{fig: 2d re and l2} we see that our suggested approach needs only half the computational time to converge in comparison with the simulation using zero initialization. Note that the applied control yields shorter trajectories and thus reduced overall computational costs -- this can, for instance, be seen in \Cref{fig: 2d setting and nn control approximation}.

\begin{figure}[tbhp]
\centering
\subfloat[]{\label{fig: 8a}\includegraphics[width=60mm]{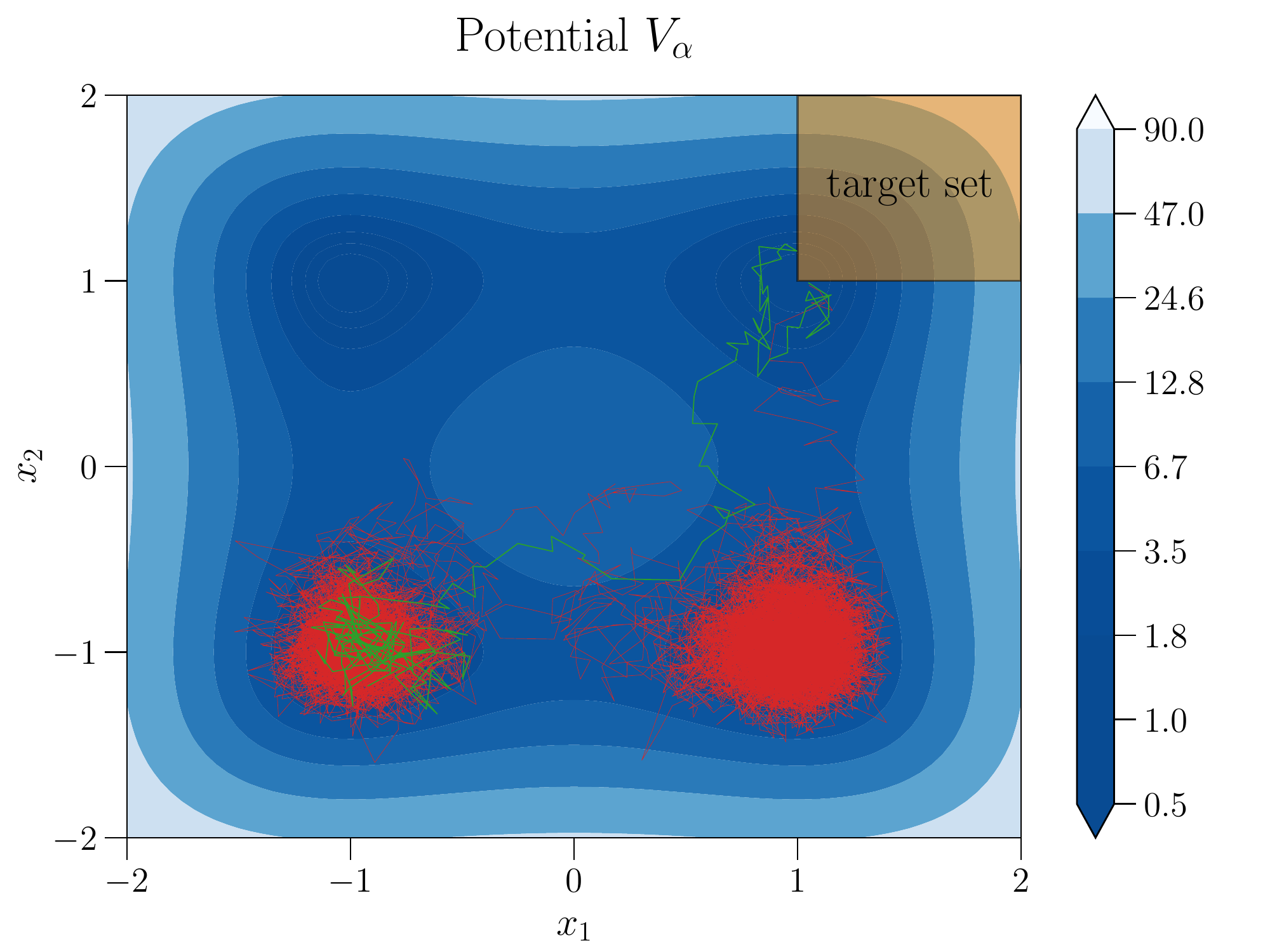}}
\subfloat[]{\label{fig: 8b}\includegraphics[width=60mm]{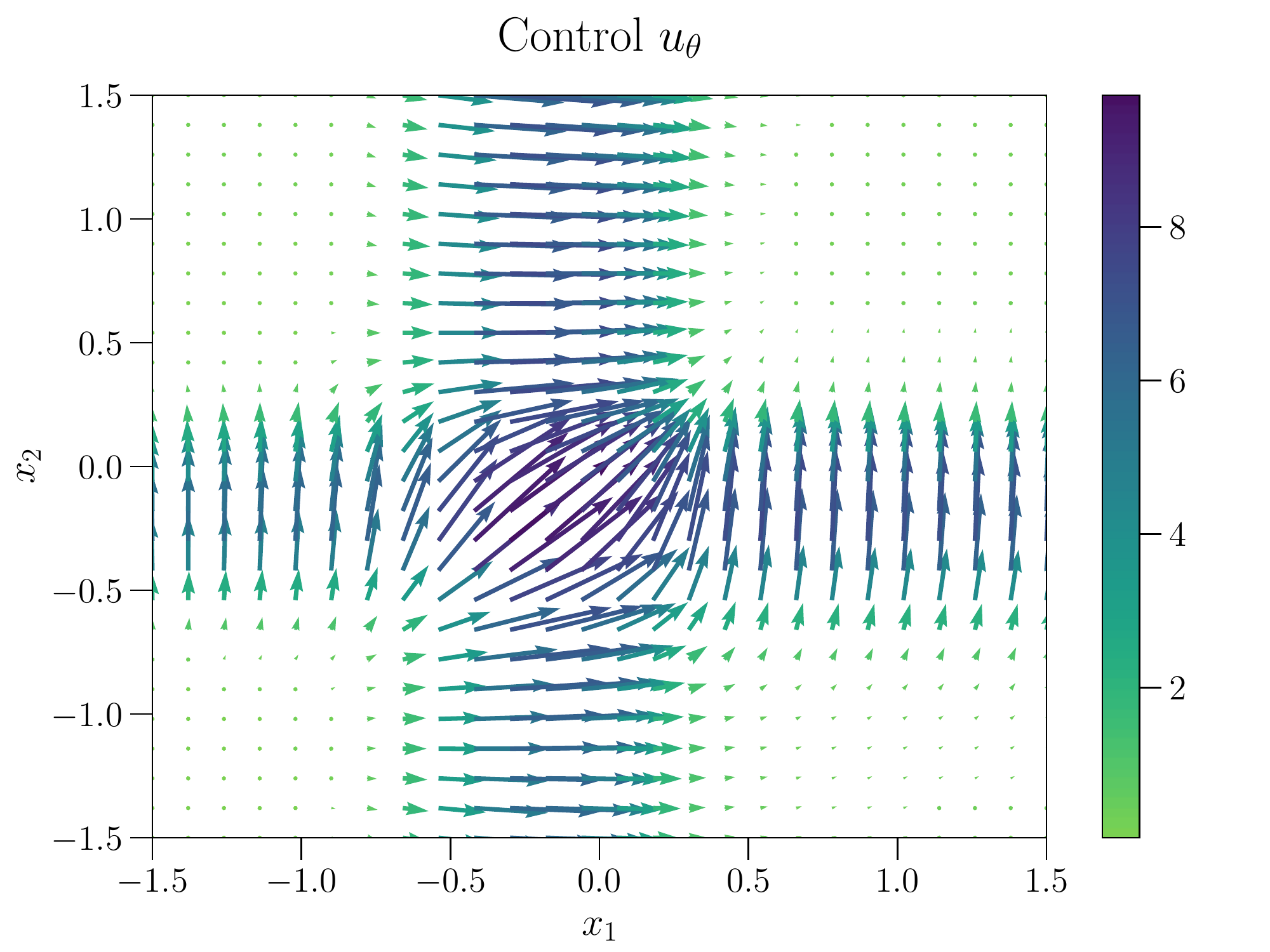}}
\caption{\textbf{(a)} The double well potential for $d=2$ has four minima at $x = (\pm 1, \pm 1)$ and a local maximum at $ x = (0, 0)$. The uncontrolled trajectory (in red) is long and gets trapped in two of the minima for a significant amount of time. The optimal importance sampling trajectory (in green) on the other hand is much shorter. \textbf{(b)} Control approximated by a neural network using \Cref{alg: Efficient importance sampling}.}
\label{fig: 2d setting and nn control approximation}
\end{figure}

\begin{figure}[tbhp]
\centering
\subfloat[]{\label{fig: 9a}\includegraphics[width=60mm]{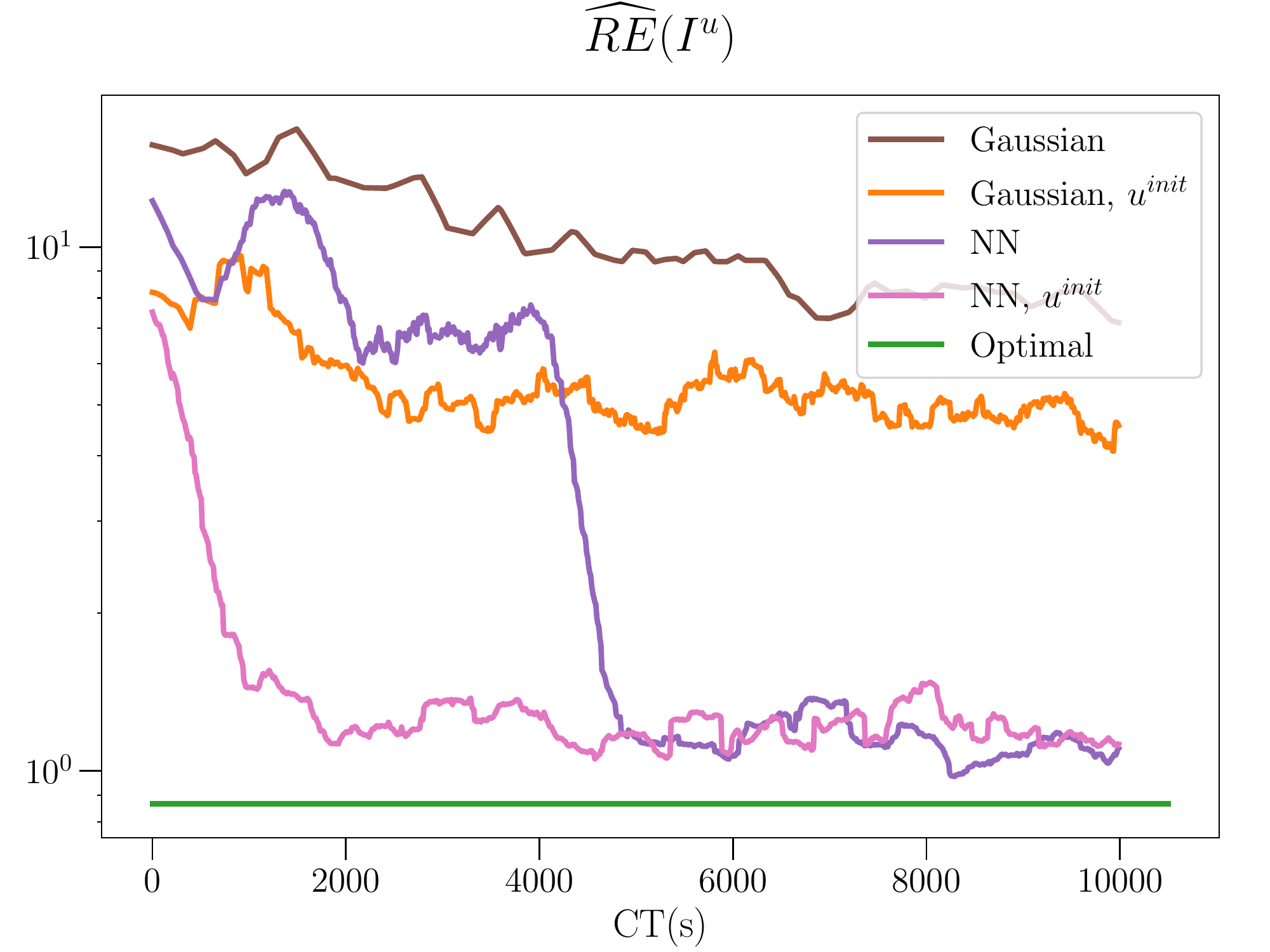}}
\subfloat[]{\label{fig: 9b}\includegraphics[width=60mm]{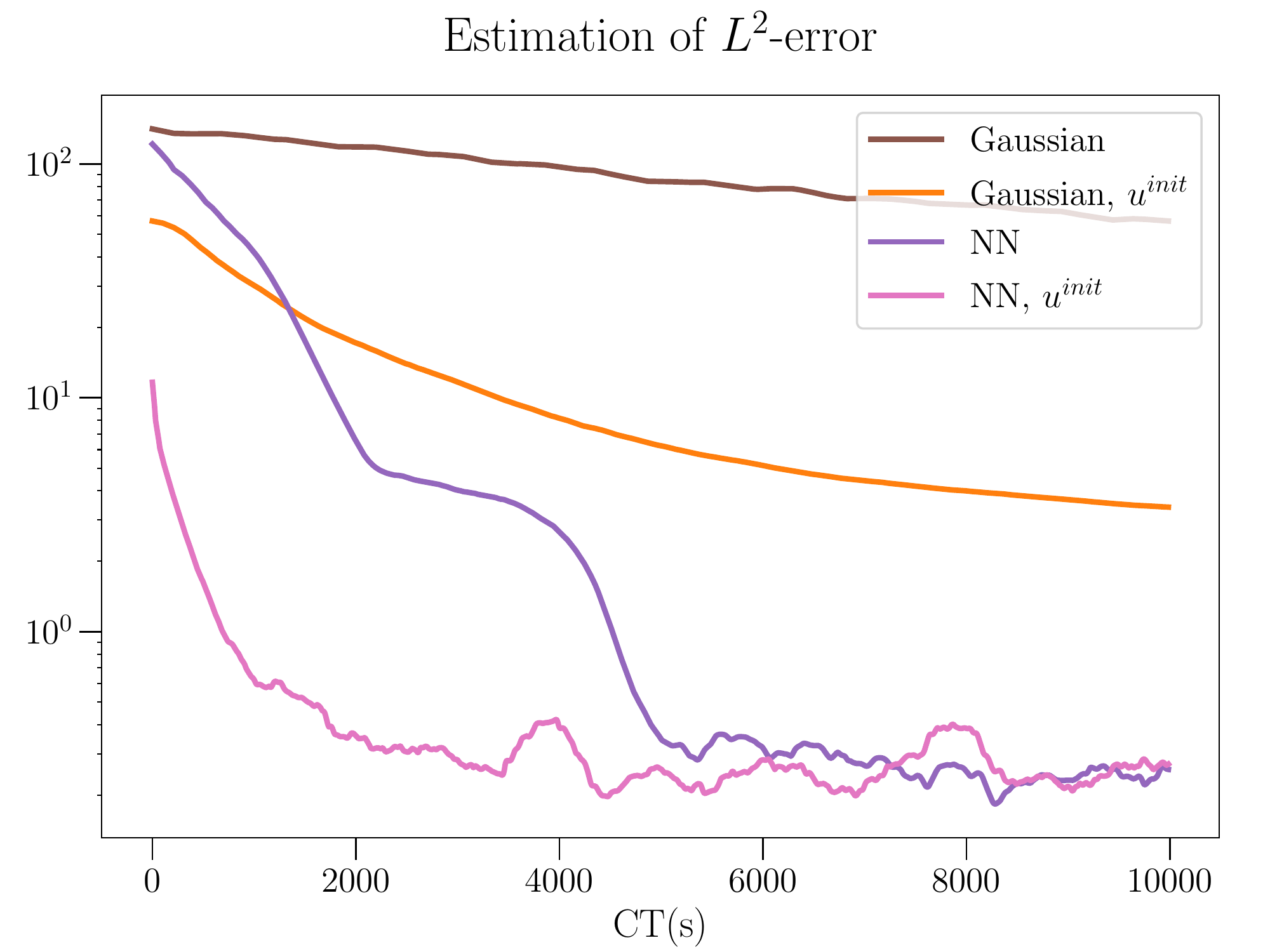}}
\caption{\textbf{(a)} Importance sampling relative error and \textbf{(b)} estimation of $L^2(u)$ as a function of the computation time.}
\label{fig: 2d re and l2}
\end{figure} 

\subsubsection*{Example in $\mathbf{d = 4}$}
Let us now consider an example in $d = 4$. We again use the potential stated in \eqref{eq: nd double well potential} and set $\alpha = (5, 5, 5, 5)^\top$. Note that the potential now has 16 minima. This time we rely on \Cref{alg: cumulative metadynamics} for our metadynamics based control initialization, for which we choose $\delta = 5$, $K^\text{meta}=100$, $\eta = 1$, $r=0.95$, and $\Sigma = 0.5 \operatorname{Id}$. The main advantage of the cumulative version of the adapted metadynamics algorithm is that we now rely on multiple trajectories for finding a good control initialization. This method is more robust and the trajectories explore a larger part of the domain. For the chosen parameters the resulting modified potential consists of $M=284$ unnormalized Gaussian functions. For $d \geq 4$ a relevant step is the initialization of the control with $u^\text{init}$ by minimizing the loss stated in \eqref{eq: approximation problem loss}. In such a case the support of the bias potential can be small in comparison with the considered domain and sampling the training data uniformly might not be feasible. Instead, we sample new training data for each gradient step following a normal distribution centered in the different unnormalized Gaussian which the chosen adapted version of the metadynamics algorithm has added. A total of $10^4$ gradient steps and $10^3$ sampled points for each gradient step suffice to provide a good initialization. \Cref{fig: 4d mean and re} shows the estimation of $\Psi$ provided by the neural network approximation with the abovementioned metadynamics based initialization as well as the relative error of the importance sampling estimator as a function of the gradient steps. As a reference value for $\Psi$ we take a Monte Carlo estimator that relies on $K=10^8$ trajectories.

\begin{figure}[tbhp]
\centering
\subfloat[]{\label{fig: 10a}\includegraphics[width=60mm]{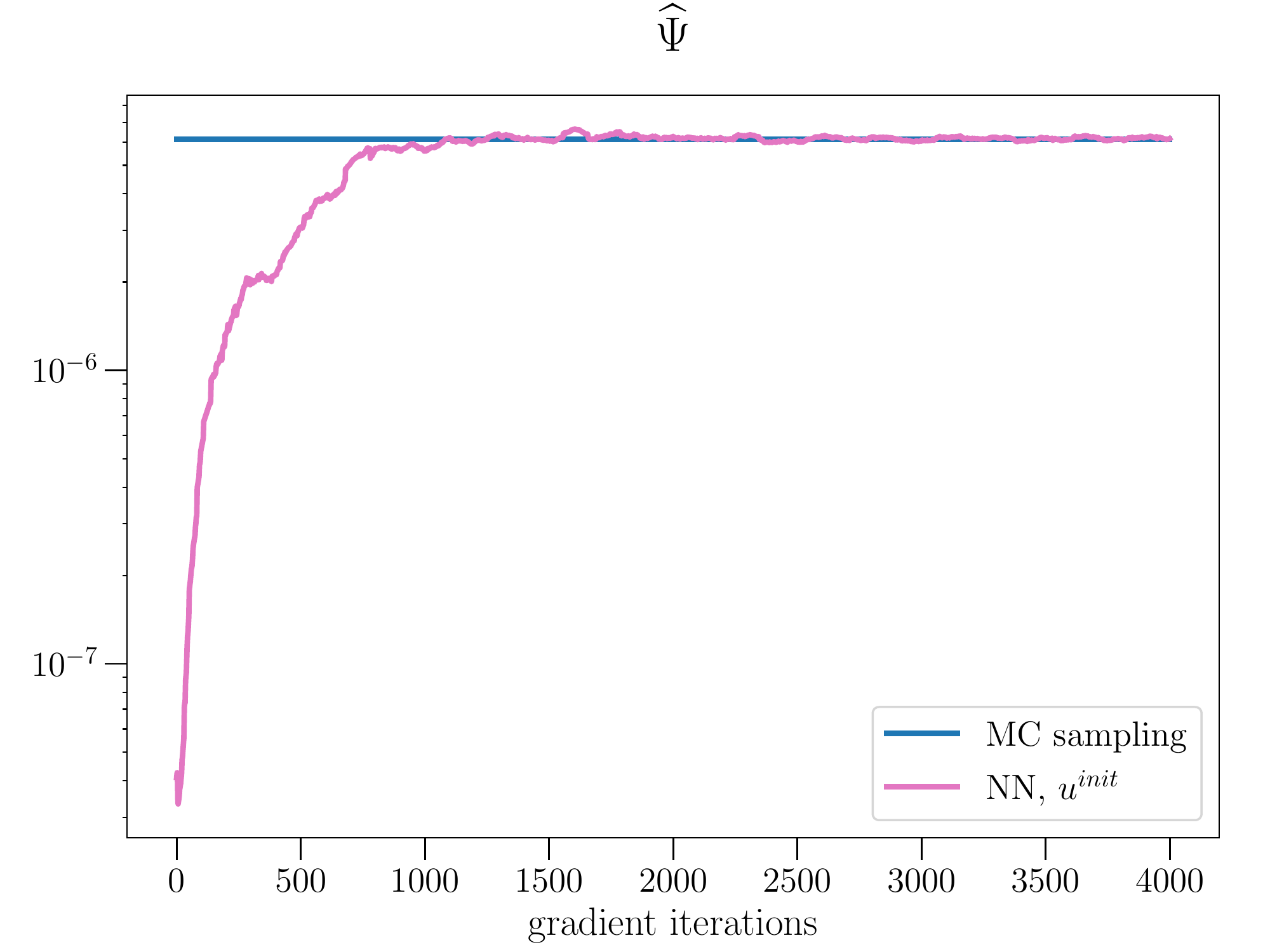}}
\subfloat[]{\label{fig: 10b}\includegraphics[width=60mm]{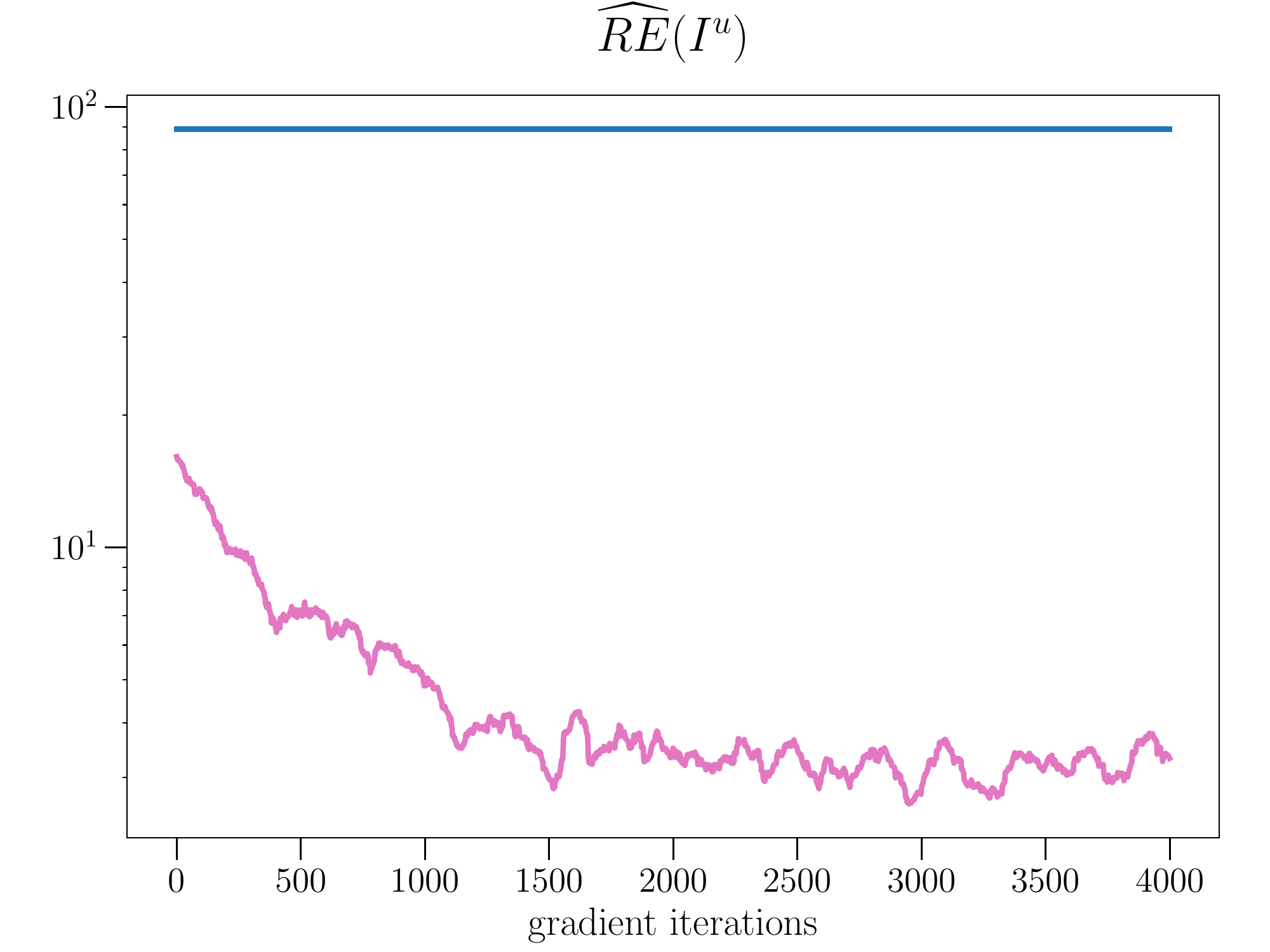}}
\caption{\textbf{(a)} Estimation of $\Psi$ and \textbf{(b)} the importance sampling relative error as a function of the gradient iterations.}
\label{fig: 4d mean and re}
\end{figure}

In our experiments, it turns out that we indeed need to rely on the metadynamics based control initialization, since the zero initialization does not work due to long trajectories and memory issues. Note that the target set is much smaller in comparison to the rest of the domain, which, together with the intrinsic metastability of the system, implies very long trajectories. In \Cref{fig: 4d mean and re} we notice that the suggested metadynamics based initialization converges and gives an accurate estimator with a smaller relative error. In \Cref{fig: 4d mc vs is} we observe that the estimation via naive Monte Carlo sampling requires the simulation of a huge amount of trajectories and is less accurate than the estimation via our suggested procedure. 

\begin{figure}[tbhp]
\centering
\subfloat[]{\label{fig: 11a}\includegraphics[width=60mm]{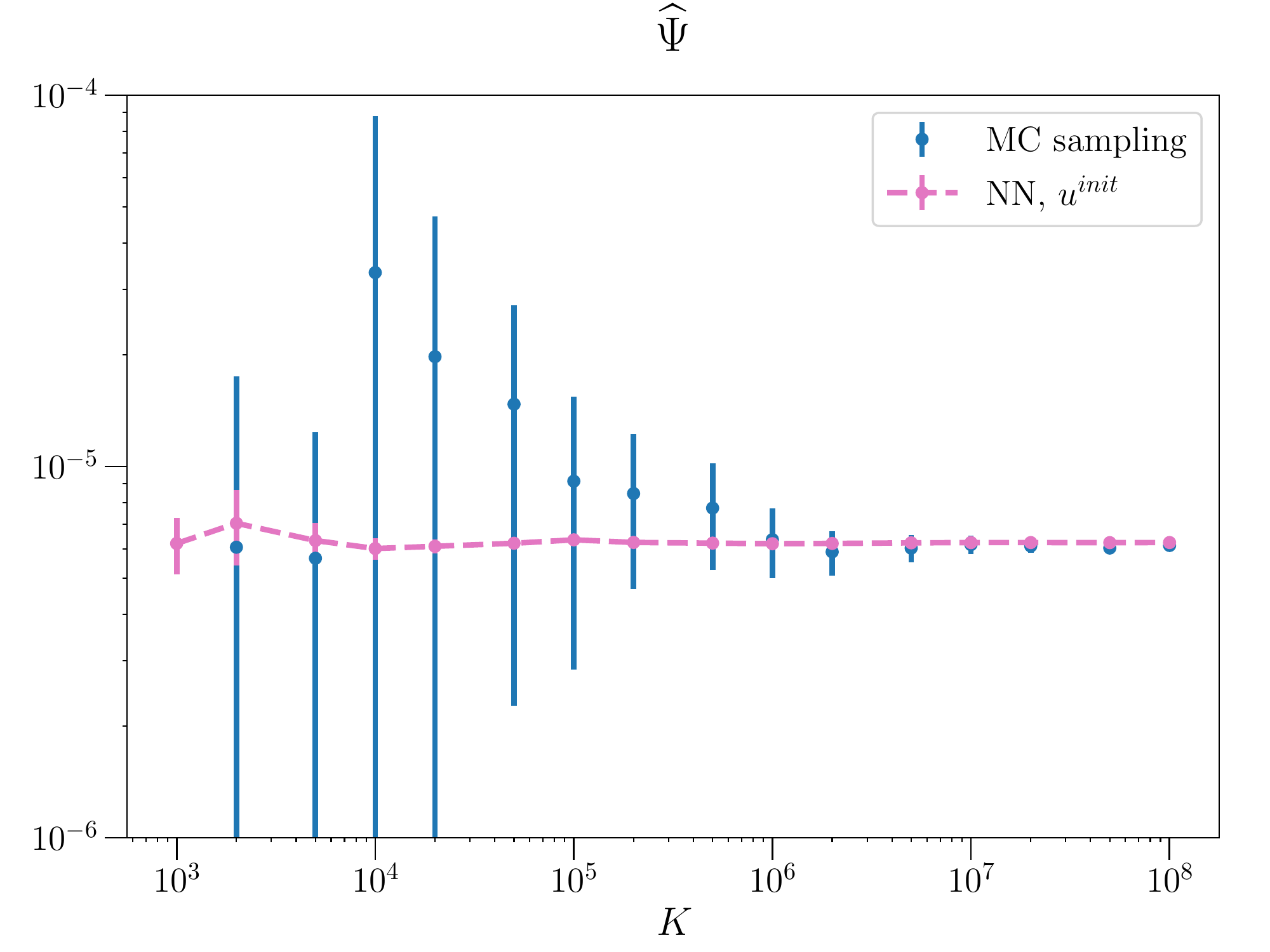}}
\subfloat[]{\label{fig: 11b}\includegraphics[width=60mm]{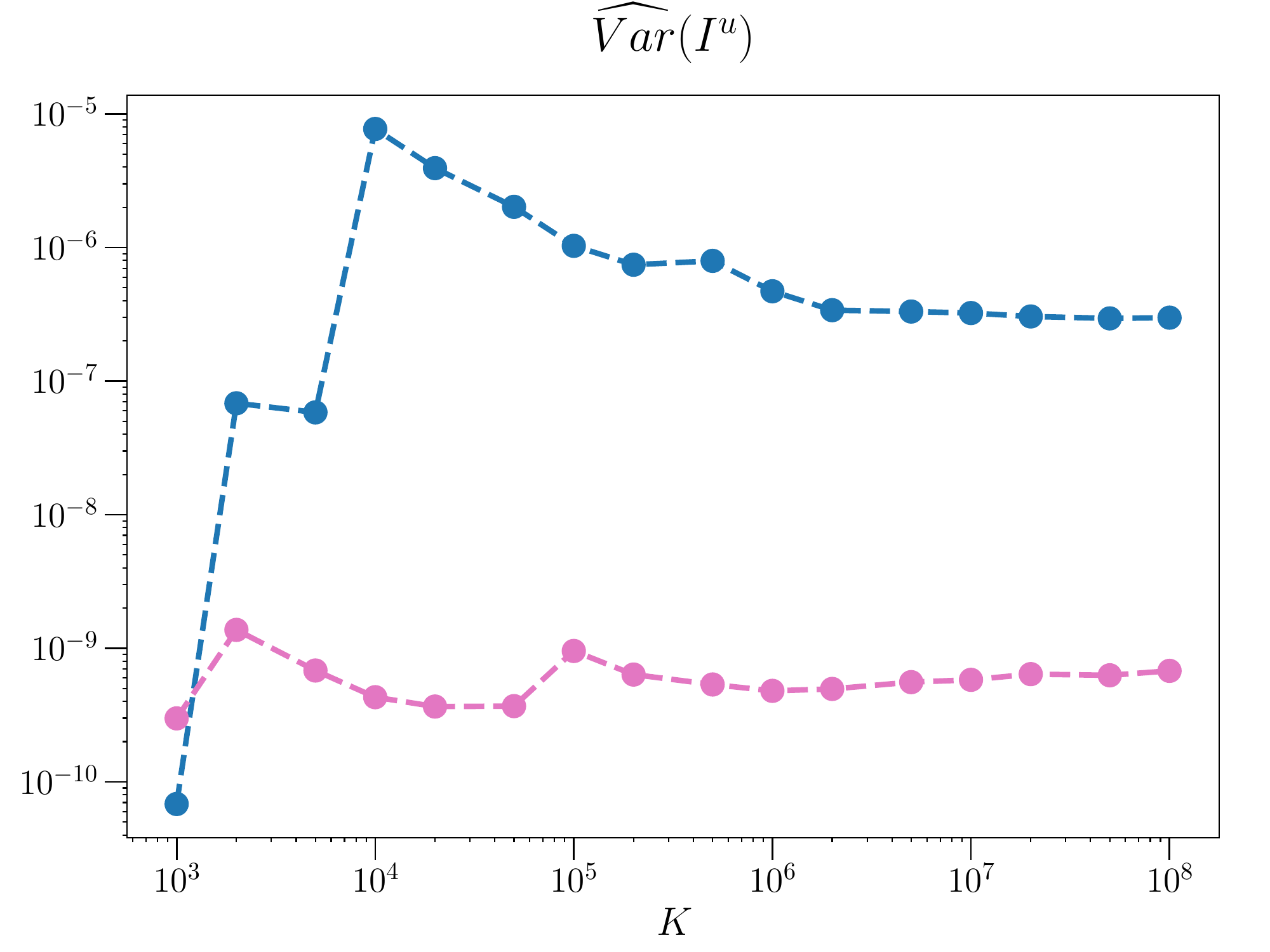}}
\caption{\textbf{(a)} Estimation of $\Psi$ as well as \textbf{(b)} the corresponding sampling variance when using either naive Monte Carlo or the learned importance sampling estimating according to \Cref{alg: Efficient importance sampling} as a function of different batch sizes $K$. Note that the Monte Carlo estimator corresponds to a noncontrolled sampling, i.e., $u=0$.}
\label{fig: 4d mc vs is} 
\end{figure}

\subsection{High-dimensional example with metastable features}
In our final example we consider a high-dimensional example in $d = 20$, where one dimension is particularly metastable. We set $\alpha_1 = 5$ and $\alpha_j=0.5$ for $j \in \{2, \dots, 20\}$ and we choose the target set to be \sloppy${\mathcal{T} = [1, 3] \times [-3, 3]^{d-1}}$, so that the trajectories stop after overcoming the potential barrier in the metastable coordinate. Here we again do not have a reference solution due to the curse of dimensionality. 

For this example we use two different approaches to generate a good control initialization $u^{\text{init}}$. First, we implement the cumulative version of the adapted metadynamics algorithm in the full state space using \Cref{alg: cumulative metadynamics}, with $\delta = 2$, $K^\text{meta}=100$, $\eta = 1$, $r=0.95$, and $\Sigma = 0.5 \operatorname{Id}$. Second, we use again \Cref{alg: cumulative metadynamics} with the same choice of parameters in a reduced collective variable space. The reaction coordinate is chosen to be the projection on the first coordinate, $\xi(x) = x_1$, since it describes the most important characteristics of the dynamics. For this choice, the corresponding controlled effective dynamics is again of Langevin type.

In our experiments we can see that the optimization procedure benefits from both metadynamics based initializations. 
In \Cref{fig: 20d mean and re metastable} we can observe that although all methods eventually find the same minimum, the two initialized versions converge much faster. Further, we notice that the initialization with reaction coordinates converges faster than the initialization relying on the full state space. 

\begin{figure}[tbhp]
\centering
\subfloat[]{\label{fig: 12a}\includegraphics[width=60mm]{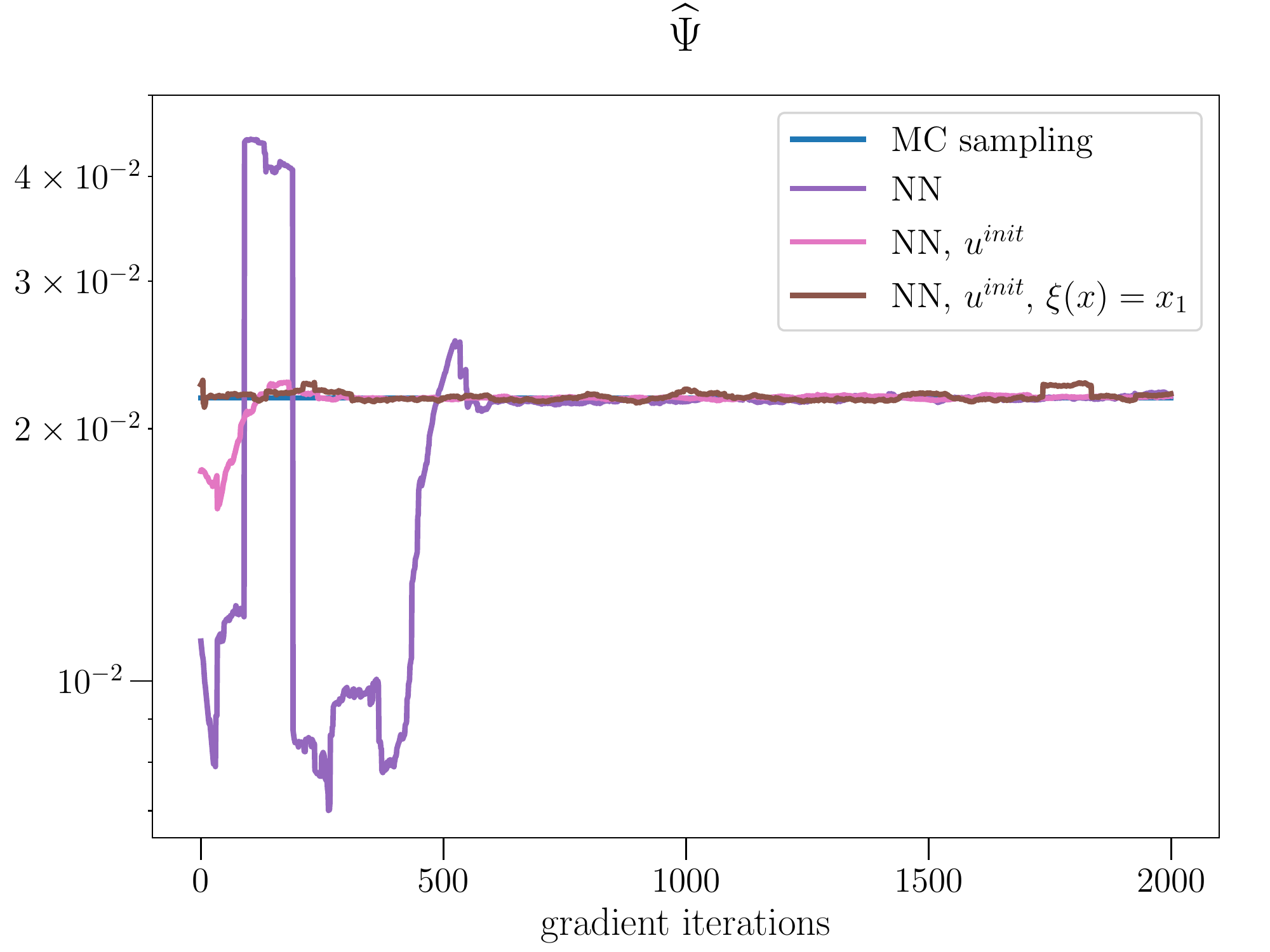}}
\subfloat[]{\label{fig: 12b}\includegraphics[width=60mm]{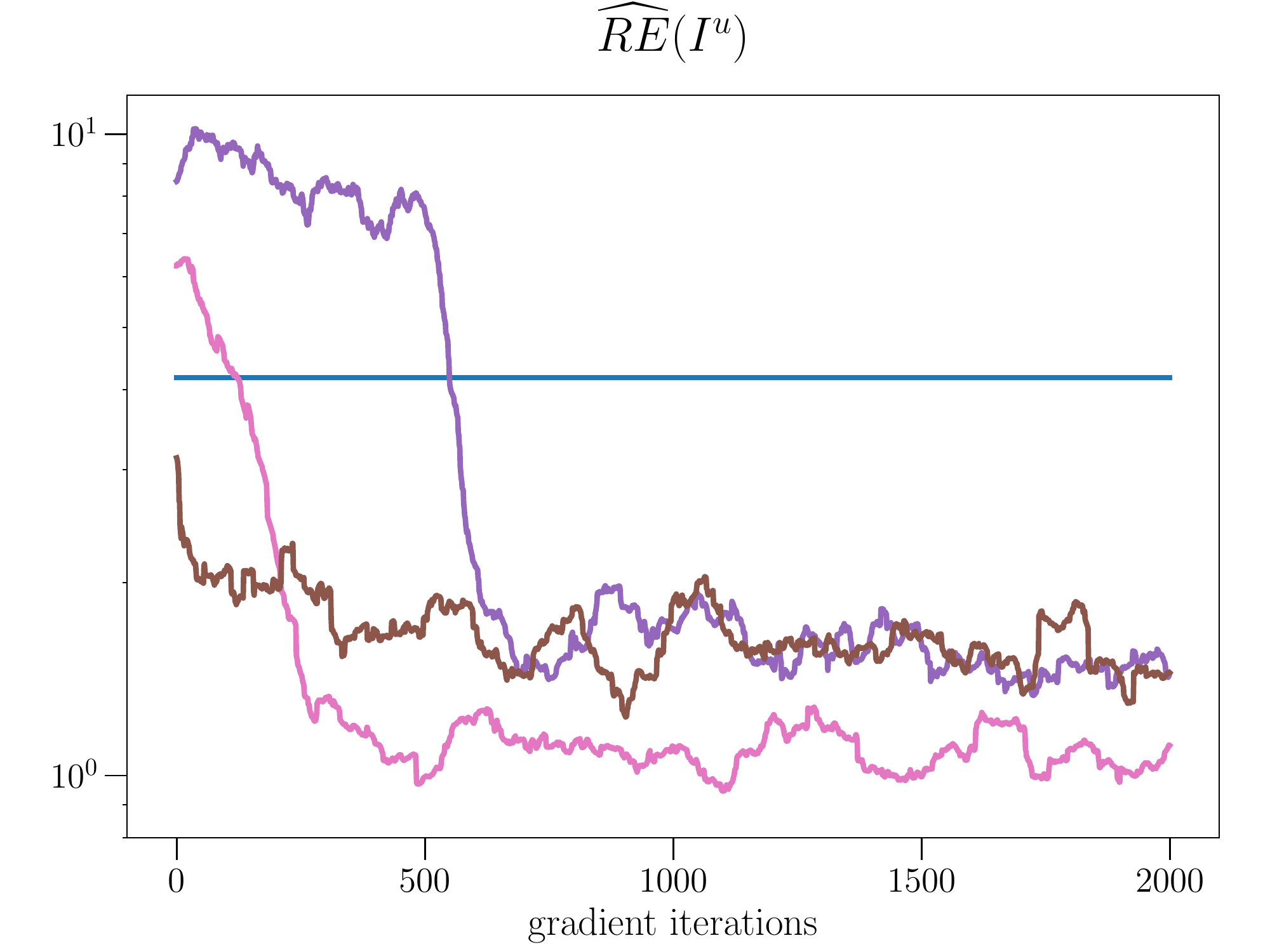}}
\caption{\textbf{(a)} Estimation of $\Psi$ and \textbf{(b)} importance sampling relative error at each gradient step.}
\label{fig: 20d mean and re metastable}
\end{figure}

\begin{figure}[tbhp]
\centering
\subfloat[]{\label{fig: 13a}\includegraphics[width=60mm]{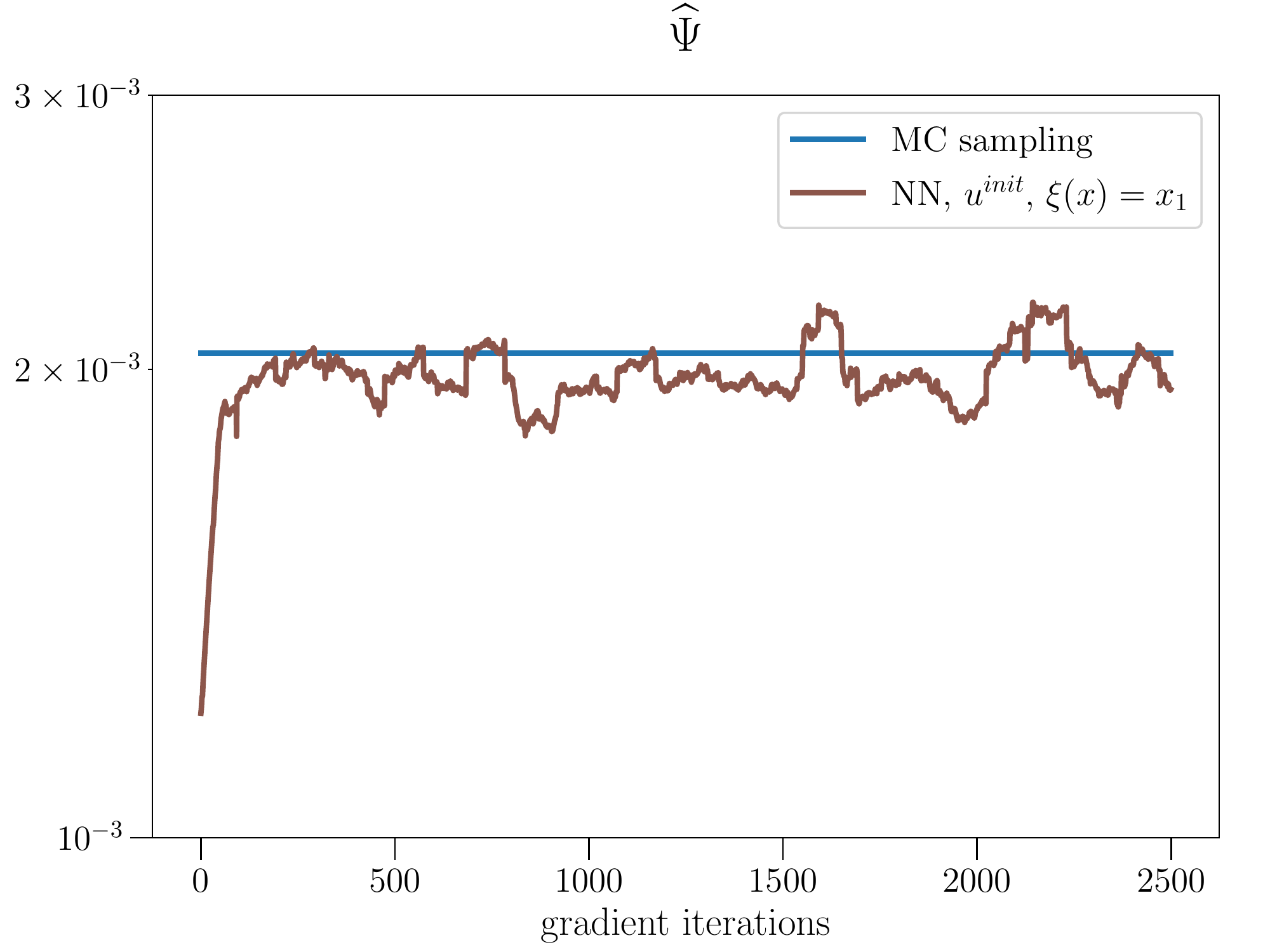}}
\subfloat[]{\label{fig: 13b}\includegraphics[width=60mm]{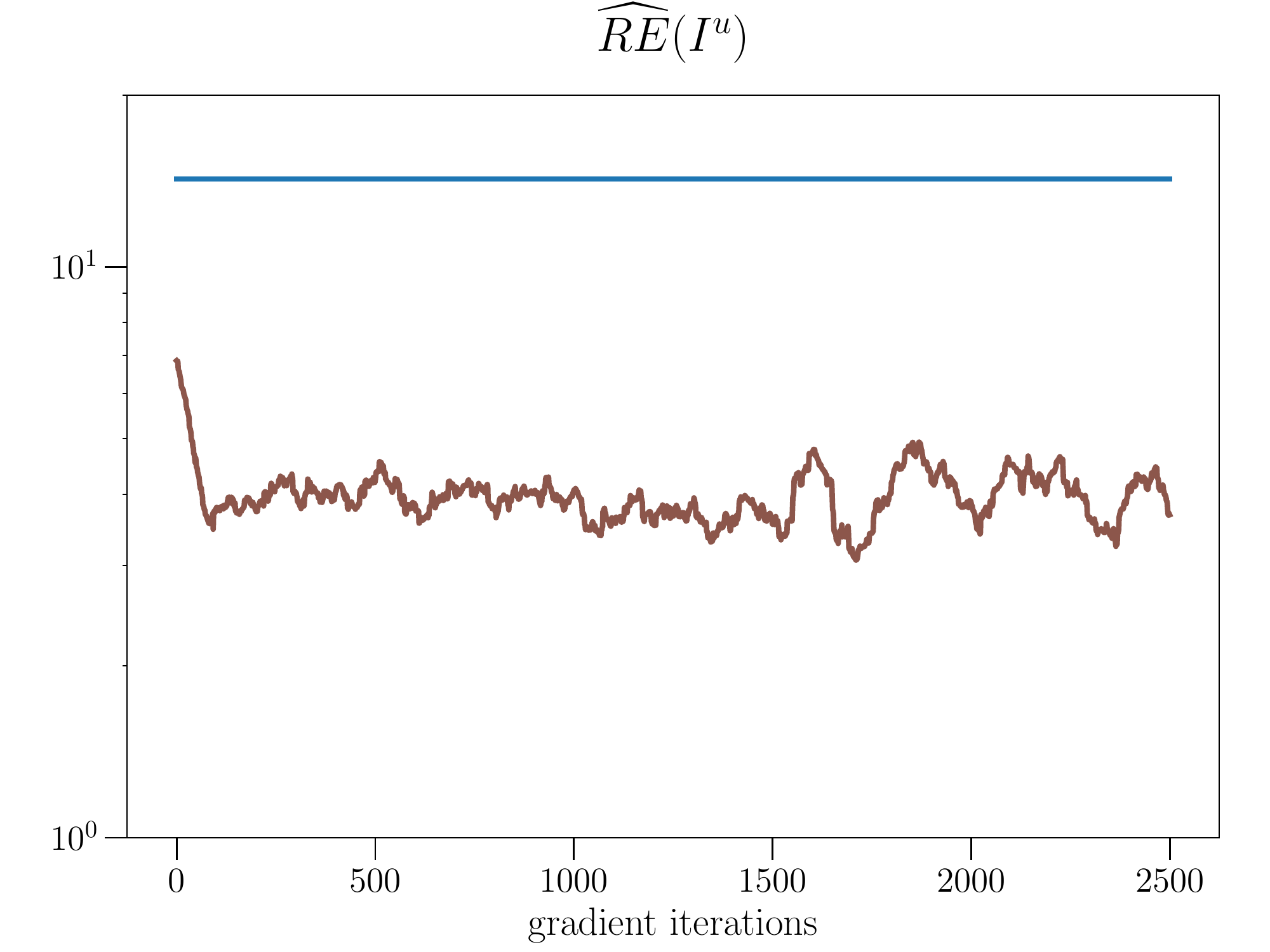}}
\caption{\textbf{(a)} Estimation of $\Psi$ and \textbf{(b)} importance sampling relative error at each gradient step.}
\label{fig: 20d mean and re high-metastable}
\end{figure}

Finally, we repeat the above experiments in an even more metastable setting, taking $\alpha_1 = 8$, with the same choice of hyperparameters. Our experiments reveal that both the noninitialized case as well as the metadynamics based initialization in the full state space fail due to the fact that the allocated memory is exceeded because of long trajectories. In \Cref{fig: 20d mean and re high-metastable} we can observe that the metadynamics based initialization in the collective variable space, however, provides an accurate estimator with smaller relative error than the plain Monte Carlo sampling. In \Cref{fig: 20d control slice meta init} we display the projection of the control in the $i$th coordinate as a function of $x_i$ for a fixed value of $x_j$, $j \in \{2, \dots, 20\}$, once before starting the optimization procedure and once after convergence. We compare the metastable direction $i=1$ with the others, e.g., $i=2$, and see that, as expected, the control gets particularly large in the metastable region. \pagebreak

\begin{figure}[tbhp]
\centering
\subfloat[]{\label{fig: 14a}\includegraphics[width=60mm]{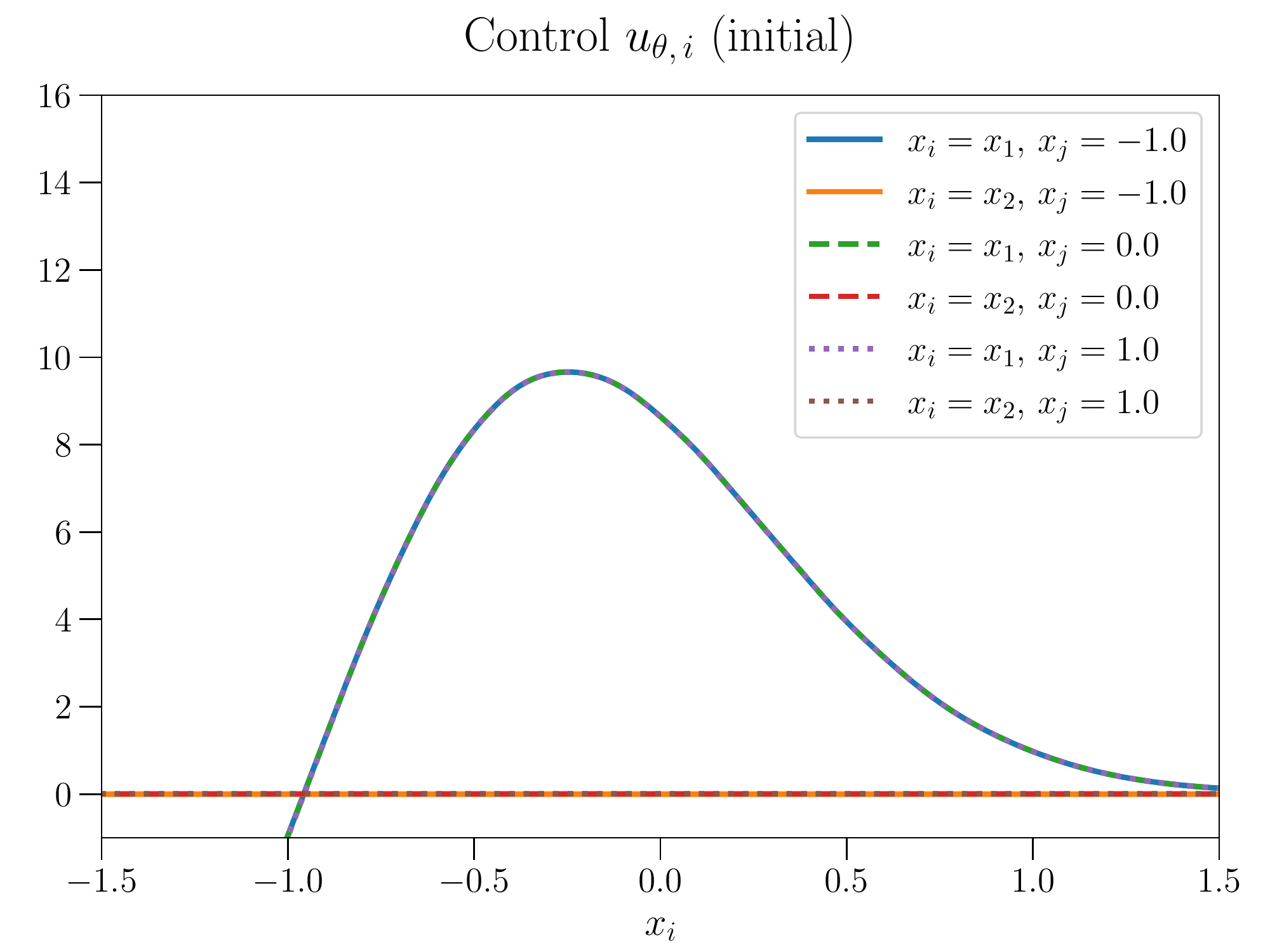}}
\subfloat[]{\label{fig: 14b}\includegraphics[width=60mm]{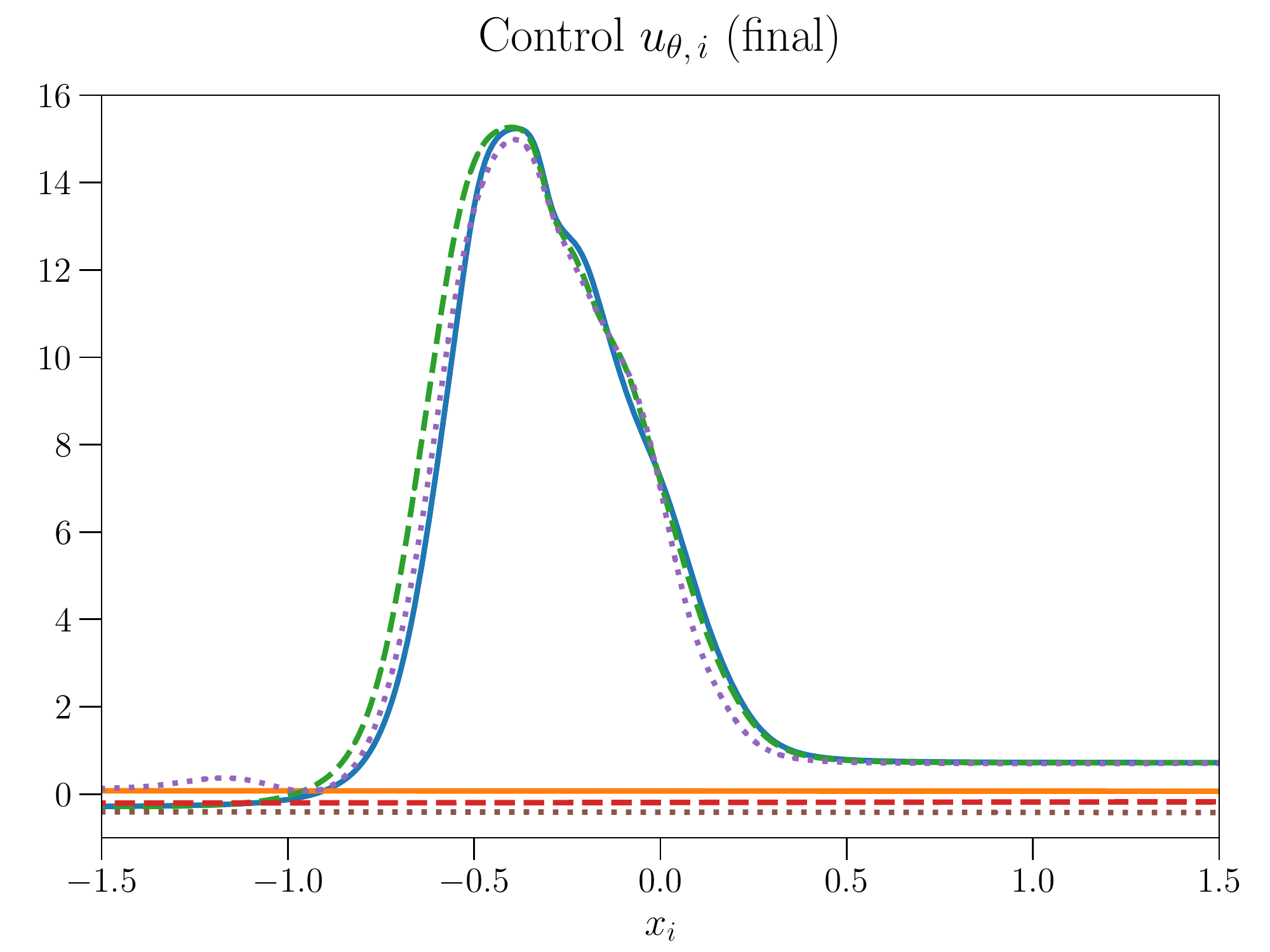}}
\caption{\textbf{(a)} Different components of the control function before the optimization procedure (coming from the metadynamics based initialization) and \textbf{(b)} after convergence.}
\label{fig: 20d control slice meta init}
\end{figure}

\section{Conclusion and outlook}
\label{sec: summary and outlook}
In this paper we have presented a novel method that improves the sampling of metastable diffusions. To be precise, this method is able to cope with two particular challenges, namely (1) the simulation of certain rare events of interest and (2) high variances in the gradient computation of importance sampling optimization algorithms. To overcome those issues, we have suggested combining optimal control based importance sampling with the metadynamics algorithm. In fact, we are able to combine the best of those two algorithms: offering reasonable initializations, but still striving for systematic variance reductions of Monte Carlo estimators. For the control function approximation we rely on neural networks, which allow for high-dimensional applications and offer further advantages in contrast to a linear combination of ansatz functions, since those would have to be placed explicitly in the domain of interest and would require the tuning of additional hyperparameters. We have further derived a gradient estimator that can easily be computed with automatic differentiation tools and thereby allows for efficient computations. We could demonstrate in multiple numerical experiments that our method improves the convergence of optimizing control based importance sampling significantly. In particular, we note that often our methods work well when alternative strategies do not produce reasonable results anymore. Overall, we are thus able to design low variance importance sampling estimators even in very metastable scenarios.

The stochastic optimization algorithms that we have considered in this work can be understood as some sort of reinforcement learning. For future work, it might therefore be interesting to consider tricks that have been developed in this fruitful field of machine learning in recent years. In particular, it might be promising to alter optimization objectives by, for instance, adding additional terms that strive for a minimization of the hitting times, while still keeping the variance of the estimators low. Such strategies might become particularly relevant in realistic examples in ever larger dimensions with even higher metastabilities. Furthermore, we believe that a connection of our algorithms to model reduction attempts might be fruitful such that (as in the original version of the algorithms) metadynamics based control initialization would only be executed in the relevant coordinates. With that we are optimistic that our proposed method will lead to more efficient sampling of real physical systems.

\section*{Acknowledgement}
The research of E.R.B and L.R. has been funded by Deutsche Forschungsgemeinschaft (DFG) through grant CRC 1114 ``Scaling Cascades in Complex Systems,'' A05 (Probing scales in equilibrated systems by optimal nonequilibrium forcing, project 235221301). The research of J.Q. has been funded by the Einstein Foundation Berlin. 

\section*{Code availability}
The code used for the numerical examples is available on GitHub at  \url{www.github.com/riberaborrell/sde-importance-sampling}.

\appendix
\section{Appendix}
\label{appendix}

\subsection{Alternative gradient computations}
\label{app: alternative gradient computations}

In this section we present an alternative way of computing the gradient of the control functional \eqref{eq: cost functional}, now relying on a discrete version of the controlled stochastic process. This gradient estimator has already been suggested in \cite{Hartmann2012}, however, we generalize it to more general function classes. The strategy is to consider the controlled process in discrete time, first with deterministic time horizons. Then, by relying on transition probabilities of the discretized process, we can compute the gradient of the discrete cost functional. Eventually, we can change to random stopping times, which yields a gradient that can be estimated by Monte Carlo.\\

Let us start by stating the discrete version of our controlled process \eqref{eq: controlled langevin sde} on a time grid $0 < t_1 < \cdots < t_N$, for a fixed $N$, namely

\begin{equation}
\label{eq: controlled langevin discretized sde}
\widehat{X}_{n+1}^u 
= \widehat{X}_{n}^u + \left(- \nabla V(\widehat{X}_n^u)  + \sigma(\widehat{X}_n^u) u_\theta(\widehat{X}_n^u)\right) \Delta t + \sigma(\widehat{X}_n^u) \xi_{n+1}  \sqrt{\Delta t} ,
\end{equation}

where $\Delta t = t_{n+1} - t_n$ is the time increment and $\xi_{n+1} \sim \mathcal{N}(0, \operatorname{Id})$ are standard normally distributed random variables. Moreover, the discrete version of the control functional \eqref{eq: cost functional} reads
\begin{equation}
\label{eq: discrete cost functional}
\widehat{J}(u_\theta ; x) \coloneqq \mathbb{E}^x [ h_\theta(\widehat{X}^u)],\qquad h_\theta(\widehat{X}^u ) \coloneqq \sum_{n=0}^{N-1} \Big( f(\widehat{X}^u_n) + \frac{1}{2} |u_\theta(\widehat{X}^u_n)|^2 \Big) \Delta t + g(\widehat{X}^u_{N}).
\end{equation}
The process $\widehat{X}^u$ is a discrete Markov process and by using the Chapman-Kolmogorov equation (see e.g. Section 2.2 in \cite{Pavliotis2014}) we can express its joint probability density function conditional on $\widehat{X}_0^u = x$ in terms of the transition densities $G: \mathbb{R}^d \times\mathbb{R}^d \to \mathbb{R}_{\ge 0}$
\begin{equation}
\label{eq: joint pdf}
\rho(\widehat{X}^u) = \prod_{n=0}^{N-1} G(\widehat{X}^u_{n+1} | \widehat{X}^u_n).
\end{equation}
From \eqref{eq: controlled langevin discretized sde} we know that for any discrete steps $x, y \in \mathbb{R}^d$, $G$ is of multivariate normal form, namely
\begin{equation}
    G(y | x) = \mathcal{N}\left(y; x+ \left(-\nabla V(x) + \sigma(x)u(x)\right)\Delta t, \sigma(x)\sigma^\top(x) \Delta t \right).
\end{equation}

Let us for computational convenience assume $\sigma(x) = \sqrt{2\beta^{-1}}\operatorname{Id}$, then

\begin{equation}
G(y | x) 
= \frac{1}{(4 \pi \beta^{-1} \Delta t)^{d/2}} \exp{\left( - \frac{\beta \Delta t } {4} \left |\frac{y - x}{\Delta t} + \nabla V(x)  - \sqrt{2\beta^{-1}} u_\theta(x) \right|^2 \right)}. \\
\end{equation}

By combining the transition densities in the above expression \eqref{eq: joint pdf}, we get
\begin{equation}
\label{eq: discrete path density}
\rho_\theta(\widehat{X}^u) = \frac{1}{\mathcal{Z}} \exp\left(-S_\theta(\widehat{X}^u) \right),
\end{equation}
where the so called discrete action and its normalization factor are given by
\begin{align}
\label{eq: discrete action}
S_\theta(\widehat{X}^u) &= \frac{\beta \Delta t}{4} \sum\limits_{n=0}^{N -1} \Big| \frac{\widehat{X}^u_{n+1} - \widehat{X}^u_n}{\Delta t} + \nabla V(\widehat{X}^u_n) - \sqrt{2 \beta^{-1}} \, u_\theta(\widehat{X}^u_n) \Big|^2, \\
\mathcal{Z}
&= (4 \pi \beta^{-1} \Delta t)^{N d/2}.
\end{align}

With the help of \textit{infinitesimal perturbation analysis} we can now obtain an estimator for the gradient of the above discretized loss function with respect to the parameter vector $\theta$.

\begin{proposition}[Derivative of discrete cost functional]
\label{prop: grad discrete cost functional}
Consider $h_\theta$ as defined in \eqref{eq: discrete cost functional}. We can compute the derivative of the discrete cost functional with respect to a parameter $\theta_i$ as
\begin{subequations}
\begin{align}
\label{eq: grad discrete cost functional}
\frac{\partial}{\partial \theta_i} \widehat{J}(u_\theta ; x) 
&= \mathbb{E}^x \biggl[\frac{\partial}{\partial \theta_i} h_\theta(\widehat{X}^u) - h_\theta(\widehat{X}^u)  \frac{\partial}{\partial \theta_i} S_\theta(\widehat{X}^u) \biggr],\\
\begin{split}
& = \mathbb{E}^x\biggl[\sum\limits_{n=0}^{N-1} u_\theta(\widehat{X}^u_n) \cdot \frac{\partial}{\partial \theta_i} u_\theta(\widehat{X}^u_n) \Delta t  \\
& \qquad\quad + \biggl(\sum_{n=0}^{N-1} \biggl( f(\widehat{X}^u_n) + \frac{1}{2} |u_\theta(\widehat{X}^u_n)|^2 \biggr) \Delta t + g(\widehat{X}^u_{N})\biggr)  \\
&\qquad\qquad\qquad \times \biggl(\sum\limits_{n=0}^{N-1} \xi_{k+1} \cdot \frac{\partial}{\partial \theta_i} u_\theta(\widehat{X}^u_n)\sqrt{\Delta t} \biggr)\biggr]
\end{split}
\end{align}
\end{subequations}
where the expectation is to be understood with respect to the discrete path measure with density $\rho$ as defined in \eqref{eq: discrete path density} and the discrete action $S$ is defined in \eqref{eq: discrete action}.
\end{proposition}

\begin{proof}[Proof of \Cref{prop: grad discrete cost functional}]
We follow essentially the arguments of \cite{Hartmann2012} without restricting the choice of the space of possible controls to linear combination of vector fields related to Gaussian ansatz functions. First, let us compute the partial derivatives of the discrete control functional \eqref{eq: discrete cost functional} and the discrete action \eqref{eq: discrete action}, namely

\begin{align}
\label{eq: partial phi}
\frac{\partial}{\partial \theta_i}h_\theta(\widehat{X}^u) 
&= \Delta t \sum\limits_{n=0}^{N-1} u_\theta(\widehat{X}_n^u) \cdot \frac{\partial}{\partial \theta_i} u_\theta(\widehat{X}_n^u)
\end{align}
and
\begin{subequations}
\begin{align}
\label{eq: partial S}
\frac{\partial}{\partial \theta_i} S_\theta(\widehat{X}^u) 
&= - \frac{\sqrt{\beta} \Delta t}{\sqrt{2}} \sum\limits_{n=0}^{N-1} \left( \frac{\widehat{X}_{n+1}^u - \widehat{X}_n^u}{\Delta t} + \nabla V(\widehat{X}_n^u) - \sqrt{2 \beta^{-1}} \, u_\theta(\widehat{X}_n^u) \right)  \frac{\partial}{\partial \theta_i} u_\theta(\widehat{X}_n) \\
&= - \sqrt{\Delta t} \sum\limits_{n=0}^{N-1} \xi_{n+1} \frac{\partial}{\partial \theta_i} u_\theta(\widehat{X}_n^u) .
\end{align}
\end{subequations}
Let us write the expectation of the loss function as an integral over the state space
\[
\widehat{J}(u_\theta; x) 
= \mathbb{E}^x [ h_\theta(\widehat{X}^u)] 
= \int_{\mathbb{R}^d \times \dots \times \mathbb{R}^d} h_\theta \, \rho_\theta \, \mathrm dx_1 \dots \mathrm dx_N .
\]
Then, the partial derivative of the loss function with respect to $\theta_i$ can be computed like
\begin{equation}
\label{eq: partial loss}
\frac{\partial}{\partial \theta_i} \widehat{J}(u_\theta; x) 
= \int_{\mathbb{R}^d \times \dots \times \mathbb{R}^d} \left(\left(\frac{\partial}{\partial \theta_i} h_\theta \right) \rho_\theta + h_\theta \left( \frac{\partial}{\partial \theta_i} \rho_\theta \right) \right) \mathrm dx_1 \dots \mathrm dx_N .
\end{equation}
By using the fact that the normalization factor does not depend on the parameters the partial derivative of the probability density function with respect to $\theta_i$ simplifies to
\[
\frac{\partial}{\partial \theta_i}  \rho_\theta
= \frac{\partial}{\partial \theta_i} \left( \frac{1}{Z} \exp{(- S_\theta)} \right)
= \frac{1}{Z} \frac{\partial}{\partial \theta_i}  \exp{(-S_\theta)}
= - \left( \frac{\partial}{\partial \theta_i}  S_\theta \right) \rho_\theta,
\]
and the expression \eqref{eq: partial loss} finally reads
\begin{align*}
\frac{\partial}{\partial \theta_i} \widehat{J}(u_\theta; x)
&= \int_{\mathbb{R}^d \times \dots \times \mathbb{R}^d} \left( \left( \frac{\partial}{\partial \theta_i} h_\theta - h_\theta \left( \frac{\partial}{\partial \theta_i}  S_\theta \right) \right) \rho_\theta \right) \mathrm dx_1 \dots \mathrm dx_N \\
&= \mathbb{E}^x \left[\frac{\partial}{\partial \theta_i} h_\theta - h_\theta \left( \frac{\partial}{\partial \theta_i}  S_\theta \right) \right].
\end{align*}
\end{proof}

As mentioned before, \Cref{prop: grad discrete cost functional} holds for the case of a fixed time horizon. Let us now replace fixed times by random stopping times of the controlled dynamics, namely $\tau^u \coloneqq \inf \{s > 0 \mid X_s^u \in \mathcal{T}\}$. Note that this stopping time now depends on the control $u$ (and therefore on the parameter $\theta$) and one could be tempted to incorporate this dependency in the gradient computations. However, we have seen in \Cref{cor: gradient of cost functional} that the derived gradient is in fact exact.

\subsection{Girsanov's theorem}
\label{app: girsanov}

Girsanov's theorem \cite{Oksendal2003} provides a formula for changes of measures in path space, which are relevant for our importance sampling computations. Let us therefore provide a brief summary of the theorem.  

Let $\widetilde{\Omega} = C([0, \infty), \mathbb{R}^d)$ be the space of continuous paths equipped with the supremum norm and let $\mathcal{F} = \mathcal{B}(\widetilde{\Omega})$ denote the corresponding $\sigma$-algebra. First, we define $(M_t^u)_{0 \leq t \leq T}$ by
\begin{equation}
\label{eq: definition girsanov martingale}
M_t^u \coloneqq \exp{\left(-  \int_0^t u(X_s^u) \cdot \mathrm dW_s - \frac{1}{2} \int_0^t |u(X_s^u)|^2 \mathrm ds \right)}, \quad M_0^u \coloneqq 0,
\end{equation}
where $X^u$ is the controlled process following \eqref{eq: controlled langevin sde} and $(\widetilde{W}_t)_{0 \leq t \leq T}$ is a Brownian Motion with additional drift, determined for all $t \in [0, T]$ by
\[
\widetilde{W}_t \coloneqq W_t + \int_0^t u(X_s^u) \mathrm{d}s. 
\]
If $(M_t^u)_{0 \leq t \leq T}$ is a martingale w.r.t. the canonical filtration of the Brownian motion $(W_t)_{0 \leq t \leq T}$ then the Girsanov Theorem \cite[Thm 8.6.8]{Oksendal2003} states that there exists a probability measure $\mathbb{Q}$ absolutely continuous w.r.t.\ the original probability measure $\mathbb{P}$ characterized by $M_T^u = \frac{\mathrm d\mathbb{Q}}{\mathrm d\mathbb{P}}$, i.e. for all $A \in \mathcal{F}$
\[
\mathbb{Q}(A) = \mathbb{E}_\mathbb{Q}[\ind_A] = \mathbb{E}_\mathbb{P}[ M_T^u \ind_A], 
\]
such that $(\widetilde{W}_t)_{0 \leq t \leq T}$ is a Brownian motion with respect to $\mathbb{Q}$ and $(X_t^u, \widetilde{W}_t)$ is a weak solution of \eqref{eq: langevin sde}, i.e.
\[
\text{$\mathbb{Q}$-law of $(X_t^u)_{0 \leq t \leq T}$} = \text{$\mathbb{P}$-law of $(X_t)_{0 \leq t \leq T}$}.
\]

Notice that for applying Girsanov's theorem, one has to assume that the process \eqref{eq: definition girsanov martingale} is a martingale. Novikov's condition provides us with a sufficient requirement for stochastic processes of the form \eqref{eq: definition girsanov martingale} to be a martingale, see \cite{Nobuyuki1989}. Namely, it suffices that for all $t \in [0, T]$ 
\[
\mathbb{E}^x\left[\exp{\left(\frac{1}{2} \int\limits_0^t |u(X_s^u)|^2 \mathrm ds \right)}\right] < \infty.
\]

Girsanov's theorem can be extended to bounded stopping times (see \cite[Prop 1]{Quer2018}). If the stopping time is bounded the fulfillment of Novikov's condition has already been discussed in \cite{Lelievre2016}, \cite{Quer2018}. In this case, it holds that
\begin{equation}
\label{eq: girsanov stopping times}
\mathbb{E}^{x}_{\mathbb{P}}[\exp{(- \mathcal{W}(X))}]
= \mathbb{E}_{\mathbb{Q}}^{x}[\exp{(- \mathcal{W}(X^u))}] 
= \mathbb{E}_{\mathbb{P}}^{x}[\exp{(- \mathcal{W}(X^u))} M_{\tau^u}^u],
\end{equation}
where $\mathbb{Q}$ is the Wiener measure of the Brownian motion with drift, $\widetilde{W}$. This implies that the random variable $I^u: \widetilde{\Omega} \rightarrow \mathbb{R}$ given by
\begin{align}
\label{eq: is quantity of interest}
I^u 
&= \exp{(- \mathcal{W}(X^u))} M_{\tau^u}^u \\
&= \exp{\Big(
- \int_0^{\tau^u} f(X_s^u) \mathrm ds 
- g(X_{\tau^u}^u)
-  \int_0^{\tau^u} u(X_s^u) \cdot \mathrm dW_s 
-  \frac{1}{2} \int_0^{\tau^u} |u(X_s^u)|^2 \mathrm ds 
\Big)}
\end{align}
is equivalent to our quantity of interest $I$, as defined in \eqref{eq: quantity of interest}. We call the quantity $I^u$ the (re-weighted) importance sampling quantity of interest.

\subsection{Proofs}
\label{app: proofs}

\begin{proof}[Proof of \Cref{prop: gateaux derivative of control functional}]
The proof is adapted from \cite{Nusken2021solving} and we refer to a similar computation in \cite{Fournie1999applications} and to further technical details in \cite{Lie2016thesis}.

For $\varepsilon \in \mathbb{R}$ and $\phi \in C_b^1(\mathbb{R}^d, \mathbb{R}^d)$, let us define the change of measure
\begin{equation}
\Lambda_\tau(\varepsilon,\phi) = \exp \left( -\varepsilon \int_0^{\tau^u} \phi(X^u_s) \cdot \mathrm{d}W_s - \frac{\varepsilon^2}{2} \int_0^{\tau^u} \vert \phi(X^u_s)\vert^2 \, \mathrm{d}s\right), \qquad \frac{\mathrm{d}\mathbb{Q}}{\mathrm{d}\mathbb{P}} = \Lambda_\tau(\varepsilon,\phi).
\end{equation}
According to Girsanov's theorem, the process $(\widetilde{W}_t)_{0 \le t \le T}$, defined as
\begin{equation}
\widetilde{W}_t = W_t + \varepsilon \int_0^t \phi(X^u_s) \, \mathrm{d}s,
\end{equation}
is a Brownian motion under $\mathbb{Q}$. We therefore obtain
\begin{subequations}
\begin{align}
J(u + \varepsilon \phi;x)
& = \mathbb{E}^x \left[ \left( \frac{1}{2} \int_0^{\tau^{u+ \varepsilon \phi}}\vert (u + \varepsilon \phi) (X_s^{u+ \varepsilon \phi}) \vert^2 \, \mathrm{d}s + \int_0^{\tau^{u+ \varepsilon \phi}} f(X_s^{u+ \varepsilon \phi})\, \mathrm ds + g(X_{\tau^{u+ \varepsilon \phi}}^{u+ \varepsilon \phi})\right)  \right] \\
& = \mathbb{E}^x \left[ \left( \frac{1}{2} \int_0^{\tau^{u}} \vert (u + \varepsilon \phi) (X_s^u) \vert^2 \, \mathrm{d}s + \int_0^{\tau^{u}} f(X_s^u)\, \mathrm ds + g(X_{\tau^{u}}^u)\right) \Lambda^{-1}_\tau(\varepsilon,\phi) \right].
\end{align}
\end{subequations}
Using dominated convergence, we can interchange derivatives and integrals (for technical details, we refer to \cite{Lie2016thesis}) and compute
\label{eq: variation RE}
\begin{align}
\begin{split}
\frac{\mathrm{d}}{\mathrm{d} \varepsilon} \Big\vert_{\varepsilon = 0} J (u + \varepsilon \phi;x) 
& = \mathbb{E}^x \biggl[\int_0^{\tau^{u}} (u\cdot \phi)(X_s^u) \, \mathrm{d}s \\
& \qquad\quad + \biggl(\frac{1}{2} \int_0^{\tau^{u}} \vert u (X_s^u) \vert^2 \, \mathrm{d}s + \int_0^{\tau^{u}} f(X_s^u)\, \mathrm ds + g(X_{\tau^{u}}^u)\biggr) \\
& \qquad\qquad\quad \times \int_0^{\tau^{u}}\phi(X_s^u) \cdot \mathrm{d}W_s \biggl].
\end{split}
\end{align}
\end{proof}

\printbibliography

@book{Nobuyuki1989,
    author = {Ikeda, Nobuyuki and Watanabe, Shinzo},
    title = {Stochastic differential equations and diffusion processes},
    year = {1989},
    series = {North-Holland Mathematical Library},
    volume = {24},
    edition = {Second},
    publisher = {North-Holland Publishing Co., Amsterdam},
    pages = {xvi+555},
}

@inbook{Milstein1995,
	author = {G.N. Milstein},
	title = {Numerical integration of stochastic differential equations},
	series = {Mathematics and its Applications},
	volume = {313},
	year = {1995},
	pages = {viii+169},
	address = {Berlin, Heidelberg},
	doi = {https://doi.org/10.1007/978-94-015-8455-5},
	publisher = {Springer Berlin Heidelberg},
}

@inbook{Oksendal2003,
	author = {{\O}ksendal, Bernt},
	title = {Stochastic differential equations: An introduction with applications},
	year = {2003},
	address = {Berlin, Heidelberg},
	doi = {10.1007/978-3-642-14394-6_5},
	pages = {65-84},
	publisher = {Springer Berlin Heidelberg},
}

@book{Evans2010,
    address = {Providence, R.I.},
    author = {Evans, Lawrence C.},
    publisher = {American Mathematical Society},
    refid = {465190110},
    title = {Partial differential equations},
    year = 2010
}

@book{Owen2013,
    title = {Monte Carlo theory, methods and examples},
    author = {Owen, Art B.},
    publisher = {Self-published},
    year = {2013},
    url = {https://artowen.su.domains/mc},
}

@book{Pavliotis2014,
    author = {Pavliotis, Grigorios A.},
    title = {Stochastic processes and applications: Diffusion Processes, the Fokker-Planck and Langevin equations},
    year = {2014},
    series = {Texts in Applied Mathematics},
    volume = {60},
    publisher = {Springer, New York},
    pages = {xiv+339},
}

@phdthesis{Lie2016thesis,
	author = {Lie, Han Cheng},
	title = {On a strongly convex approximation of a stochastic optimal control problem for importance sampling of metastable diffusions},
	year = {2016},
	url = "http://dx.doi.org/10.17169/refubium-8010",
	school = {Free university Berlin}
}

@phdthesis{Richter2021phd,
    author = {Richter, Lorenz},
    title = {Solving high-dimensional {PDE}s, approximation of path space measures and importance sampling of diffusions},
    year = {2021},
    school = {BTU Cottbus-Senftenberg}
}

@article{Zwanzig1954,
    author = {Zwanzig, R.},
    title = {High Temperature Equation of State by a Perturbation Method. {I}. {N}onpolar Gases},
    year = {1954},
    journal = {The Journal of Chemical Physics},
    volume = {22},
    number = {8},
    pages = {1420-1426},
}

@article{Kumar1992,
    author = {Kumar, S. and Rosenberg, J. M. and Bouzida, D. and Swendsen, R. H. and Kollman, P. A.},
    title = {The weighted histogram analysis method for free-energy calculations on biomolecules. {I}. The method},
    year = {1992},
    journal = {Journal of Computational Chemistry},
    volume = {13},
    number = {8},
    publisher = {John Wiley & Sons, Inc.},
}

@article{Fournie1999applications,
    title = {Applications of {M}alliavin calculus to {M}onte {C}arlo methods in finance},
    author = {Fourni{\'e}, Eric and Lasry, Jean-Michel and Lebuchoux, J{\'e}r{\^o}me and Lions, Pierre-Louis and Touzi, Nizar},
    journal = {Finance and Stochastics},
    volume = {3},
    number = {4},
    pages = {391--412},
    year = {1999},
    publisher = {Springer}
}

@article{Higham2001,
    title = {An Algorithmic Introduction to Numerical Simulation of Stochastic Differential Equations},
	author = {Higham., Desmond J.},
	year = {2001},
	journal = {SIAM Review},
	volume = {43},
	number = {3},
	pages = {525-546},
	doi = {10.1137/S0036144500378302},
}

@article{Wang2001,
    title = {Determining the density of states for classical statistical models: A random walk algorithm to produce a flat histogram},
    author = {Wang, F. and Landau, D. P.},
    year = {2001},
    journal = {Phys. Rev. E},
    volume = {64},
    issue = {5},
    pages = {056101},
    numpages = {16},
    month = {10},
    publisher = {American Physical Society},
 }

@article{Laio2002,
    title = {Escaping free-energy minima},
    author = {Laio, A. and Parrinello, M.},
    year = {2002},
    journal = {PNAS},
    number = {10},
    volume={20},
    pages = {12562--12566},
}

@article{Dupuis2004,
    title = {Importance Sampling, Large Deviations, and Differential Games},
    author = {Dupuis, P. and Wang, H.},
    year = {2004},
    journal = {Stochastics and Stochastic Reports},
    volume = {76},
    number = {6},
    pages = {481-508},
}

@article{Henin2004,
    title = {Overcoming free energy barriers using unconstrained molecular dynamics simulations},
    author = {H{\'e}nin, J. and Chipot, C},
    year = {2004},
    journal = {The Journal of Chemical Physics},
    volume = {121},
    number = {7},
    pages = {2904-2914},
    doi = {10.1063/1.1773132},
 }

@article{Bussi2006,
    title = {Equilibrium Free Energies from Nonequilibrium Metadynamics},
    author = {Bussi, G. and Laio, A. and Parrinello, M.},
    journal = {Phys. Rev. Lett.},
    volume = {96},
    issue = {9},
    pages = {090601},
    numpages = {4},
    year = {2006},
    publisher = {American Physical Society},
    doi = {10.1103/PhysRevLett.96.090601},
 }

@article{Dupuis2007,
    title = {Subsolutions of an {I}saacs Equation and Efficient Schemes for Importance Sampling},
    author = {Dupuis, P. and Wang, H.},
    year = {2007},
    journal = {Mathematics of Operations Research},
    number = {3},
    pages = {723-757},
    publisher = {INFORMS},
    volume = {32},
}

@article{Hartmann2012,
	title = {Efficient rare event simulation by optimal nonequilibrium forcing},
	author = {Carsten Hartmann and Christof Sch\"utte},
	year = 2012,
	month = {11},
	doi = {10.1088/1742-5468/2012/11/p11004},
	journal = {Journal of Statistical Mechanics: Theory and Experiment},
	volume = {2012},
	number = {11},
	pages = {P11004},
	publisher = {{IOP} Publishing}
}

@article{Weare2012,
    title = {Rare Event Simulation of Small Noise Diffusions},
    author = {Eric Vanden-Einjden and Jonathan Weare},
    year = {2012},
    journal = {Communications on Pure and Applied Mathematics},
    volume = {65},
    pages = {1770 -1803},
}

@article{Dupuis2012,
    title = {Importance Sampling for Multiscale Diffusions},
    author = {Paul Dupuis and Hui Wang and Konstantinos Spiliopoulos},
    year = {2012},
    journal = {SIAM Multiscale Modeling and Simulation},
    volume = {10},
    pages = {1-27},
}

@article{Berglund2013,
    title = {Kramers' law: Validity, derivations and generalisations},
    author = {Berglund, Nils},
    journal = {Markov Processes and Related fields},
    volume = {19},
    number = {3},
    pages = {459-490},
    year={2013}
}

@article{Hartmann2014characterization,
    title = {Characterization of rare events in molecular dynamics},
    author = {Hartmann, Carsten and Banisch, Ralf and Sarich, Marco and Badowski, Tomasz and Sch\"utte, Christof},
    journal = {Entropy},
    volume = {16},
    number = {1},
    pages = {350--376},
    year = {2014},
    publisher = {Multidisciplinary Digital Publishing Institute}
}

@article{Kingma2014adam,
    title = {Adam: A method for stochastic optimization},
    author = {Kingma, Diederik P and Ba, Jimmy},
    journal = {arXiv preprint},
    year = {2014},
    url = {https://arxiv.org/abs/1412.6980}
}

@article{Spiliopoulos15,
    title = {Non-asymptotic performance analysis of importance sampling schemes for small noise diffusions},
    author = {Konstantinos Spiliopoulos},
    year = {2015},
    journal = {Journal of Applied Probability},
    volume = {52},
    pages = {1-14},
}

@article{Dupuis2015,
    title = {{Escaping from an attractor: Importance sampling and rest points I}},
    author = {Paul Dupuis and Konstantinos Spiliopoulos and Xiang Zhou},
    year = {2015},
    volume = {25},
    journal = {The Annals of Applied Probability},
    number = {5},
    publisher = {Institute of Mathematical Statistics},
    pages = {2909 -- 2958},
    doi = {10.1214/14-AAP1064},
}

@article{Chipot2015,
    author = {Comer, J. and Gumbart, J. C. and H{\'e}nin, J. and Leli{\'e}vre, T. and Pohorille, A. and Chipot, Ch.},
    title = {The Adaptive Biasing Force Method: Everything You Always Wanted To Know but Were Afraid To Ask},
    journal = {The Journal of Physical Chemistry B},
    volume = {119},
    number = {3},
    pages = {1129-1151},
    year = {2015},
    doi = {10.1021/jp506633n},
    note ={PMID: 25247823},
}

@article{Zhang2016,
    title = {Effective dynamics along given reaction coordinates{,} and reaction rate theory},
    author = {Zhang, Wei and Hartmann, Carsten and Sch\"utte, Christof},
    year = {2016},
    journal = {Faraday Discuss.},
    volume = {195},
    pages = {365-394},
    publisher = {The Royal Society of Chemistry},
}

@article{Hartmann2016,
	author = {Carsten Hartmann and Christof Sch\"utte and Wei Zhang},
	title = {Model reduction algorithms for optimal control and importance sampling of diffusions},
	journal = {Nonlinearity},
	volume = {29},
	year = 2016,
	pages = {2298--2326},
	doi = {10.1088/0951-7715/29/8/2298},
	month = {06},
	publisher = {{IOP} Publishing},
	number = {8}
}

@article{Lelievre2016,
	title = {Partial differential equations and stochastic methods in molecular dynamics},
    author = {Leli\`{e}vre, Tony and Stoltz, Gabriel},
    year = {2016},
    journal = {Acta Numerica},
    volume = {25},
	pages = {681-880},
    DOI = {10.1017/S0962492916000039},
    publisher = {Cambridge University Press}
}

@article{Valsson2016,
	title = {Enhancing Important Fluctuations: Rare Events and Metadynamics from a Conceptual Viewpoint},
	author = {Valsson, Omar and Tiwary, Pratyush and Parrinello, Michele},
	year = {2016},
    journal = {Annual Review of Physical Chemistry},
    volume = {67},
    number = {1},
    pages = {159-184},
    note ={PMID: 26980304},
}

@article{Han2017deep,
    title = {Deep learning-based numerical methods for high-dimensional parabolic partial differential equations and backward stochastic differential equations},
    author = {E, Weinan and Han, Jiequn and Jentzen, Arnulf},
    journal = {Communications in Mathematics and Statistics},
    volume = {5},
    number = {4},
    pages = {349--380},
    year = {2017},
    publisher = {Springer}
}

@article{Hartmann2017,
	title = {Variational Characterization of Free Energy: Theory and Algorithms},
	author = {Hartmann, Carsten and Richter, Lorenz and Sch\"utte, Christof and Zhang, Wei},
	year = {2017},
	month = {11},
	journal = {Entropy},
	volume = {19},
	pages = {626},
	doi = {10.3390/e19110626}
}

@article{Galvelis2017,
    title = {Neural Network and Nearest Neighbor Algorithms for Enhancing Sampling of Molecular Dynamics},
    author = {Galvelis, Raimondas and Sugita, Yuji},
    year = {2017},
    journal = {Journal of Chemical Theory and Computation},
    volume = {13},
    number = {6},
    pages = {2489-2500},
    doi = {10.1021/acs.jctc.7b00188},
    note ={PMID: 28437616},
}

@article{Bach2017breaking,
    title = {Breaking the curse of dimensionality with convex neural networks},
    author = {Bach, Francis},
    journal = {The Journal of Machine Learning Research},
    volume = {18},
    number = {1},
    pages = {629--681},
    year = {2017},
    publisher = {JMLR. org}
}

@article{Quer2018,
	title = {An Automatic Adaptive Importance Sampling Algorithm for Molecular Dynamics in Reaction Coordinates},
	author = {Quer, J. and Donati, Luca and Keller, Bettina and Weber, M.},
	year = {2018},
	month = {01},
	journal = {SIAM Journal on Scientific Computing},
	volume = {40},
	pages = {A653-A670},
}

@article{Jentzen2018proof,
    title = {A proof that deep artificial neural networks overcome the curse of dimensionality in the numerical approximation of {K}olmogorov partial differential equations with constant diffusion and nonlinear drift coefficients},
    author = {Jentzen, Arnulf and Salimova, Diyora and Welti, Timo},
    journal = {arXiv preprint},
    year = {2018},
    url = {https://arxiv.org/abs/1809.07321}
}

@article{Hartmann2018,
	title = {Importance sampling in path space for diffusion processes with slow-fast variables},
	author = {Hartmann, Carsten and Sch\"utte, Christof and Weber, Marcus and Zhang, Wei},
	year = {2018},
	journal = {Probability Theory and Related Fields},
	volume = {170},
	pages = {177-228},
	DOI = {10.1007/s00440-017-0755-3}
}

@article{Hartmann2019variational,
    title = {Variational approach to rare event simulation using least-squares regression},
    author = {Hartmann, Carsten and Kebiri, Omar and Neureither, Lara and Richter, Lorenz},
    journal = {Chaos: An Interdisciplinary Journal of Nonlinear Science},
    volume = {29},
    number = {6},
    pages = {063107},
    year = {2019},
    publisher = {AIP Publishing LLC}
}

@article{Fackeldey2022approximative,
    title = {Approximative Policy Iteration for Exit Time Feedback Control Problems Driven by Stochastic Differential Equations using Tensor Train Format},
    author = {Fackeldey, Konstantin and Oster, Mathias and Sallandt, Leon and Schneider, Reinhold},
    journal = {Multiscale Modeling \& Simulation},
    volume = {20},
    number = {1},
    pages = {379-403},
    year = {2022},
    doi = {10.1137/20M1372500},
}

@article{Nusken2021solving,
    title = {Solving high-dimensional {H}amilton--{J}acobi--{B}ellman {PDE}s using neural networks: perspectives from the theory of controlled diffusions and measures on path space},
    author = {N{\"u}sken, Nikolas and Richter, Lorenz},
    journal = {Partial Differential Equations and Applications},
    volume = {2},
    number = {4},
    pages = {1--48},
    year = {2021},
    publisher = {Springer}
}

@article{Jourdain2021,
    title = {{Convergence of metadynamics: Discussion of the adiabatic hypothesis}},
    author = {Benjamin Jourdain and Tony Leli{\'e}vre and Pierre-Andr{\'e} Zitt},
    year = {2021},
    volume = {31},
    journal = {The Annals of Applied Probability},
    number = {5},
    publisher = {Institute of Mathematical Statistics},
    pages = {2441 -- 2477},
    doi = {10.1214/20-AAP1652},
    URL = {https://doi.org/10.1214/20-AAP1652}
}

@article{Hartmann2021nonasymptotic,
    title = {Nonasymptotic bounds for suboptimal importance sampling},
    author = {Hartmann, Carsten and Richter, Lorenz},
    journal = {arXiv preprint},
    year = {2021},
    url = {https://arxiv.org/abs/2102.09606},
}

@article{Nusken2021interpolating,
    title = {Interpolating between {BSDE}s and {PINN}s--deep learning for elliptic and parabolic boundary value problems},
    author = {N{\"u}sken, Nikolas and Richter, Lorenz},
    journal = {arXiv preprint},
    year = {2021},
    url = {https://arxiv.org/abs/2112.03749}
}

@inproceedings{Richter2021solving,
    title = {Solving high-dimensional parabolic {PDE}s using the tensor train format},
    author = {Richter, Lorenz and Sallandt, Leon and N{\"u}sken, Nikolas},
    booktitle = {International Conference on Machine Learning},
    pages = {8998--9009},
    year = {2021},
    organization = {PMLR}
}

@article{Weinan2021algorithms,
    title = {Algorithms for solving high dimensional {PDE}s: From nonlinear {M}onte {C}arlo to machine learning},
    author = {Weinan, E and Han, Jiequn and Jentzen, Arnulf},
    journal = {Nonlinearity},
    volume = {35},
    number = {1},
    pages = {278},
    year = {2021},
    publisher = {IOP Publishing}
}

@article{Zhou2021actor,
    title = {Actor-Critic method for high dimensional static {H}amilton--{J}acobi--{B}ellman partial differential equations based on neural networks},
    author = {Zhou, Mo and Han, Jiequn and Lu, Jianfeng},
    journal = {SIAM Journal on Scientific Computing},
    volume = {43},
    number = {6},
    pages = {A4043--A4066},
    year = {2021},
    publisher = {SIAM}
}

\end{document}